\documentclass[11pt,reqno,a4paper]{amsart}
\usepackage{amsmath,amssymb,amsthm,enumitem,xcolor,tikz,tikz-cd,subcaption,hyperref,mathtools,float,graphicx}

\usepackage[style=alphabetic,sorting=nty,doi=false,isbn=false,url=false]{biblatex}

\usepackage{cleveref}

\tikzset{
	on each segment/.style={
		decorate,
		decoration={
			show path construction,
			moveto code={},
			lineto code={
				\path [#1]
				(\tikzinputsegmentfirst) -- (\tikzinputsegmentlast);
			},
			curveto code={
				\path [#1] (\tikzinputsegmentfirst)
				.. controls
				(\tikzinputsegmentsupporta) and (\tikzinputsegmentsupportb)
				..
				(\tikzinputsegmentlast);
			},
			closepath code={
				\path [#1]
				(\tikzinputsegmentfirst) -- (\tikzinputsegmentlast);
			},
		},
	},
	mid arrow/.style={postaction={decorate,decoration={
				markings,
				mark=at position .5 with {\arrow[#1]{stealth}}
	}}},
	mid arrowa/.style 2 args={postaction={decorate,decoration={
				markings,
				mark=at position #2 with {\arrow[#1]{stealth}}
	}}},
	mid arrowb/.style={postaction={decorate,decoration={
				markings,
				mark=at position 0.53 with {\arrow[#1]{stealth}}
	}}},
	every loop/.style={decorate,min distance=15mm,in=50,out=130,looseness=10
	},
	mid arrowa/.default={black}{0.55}
}

\usetikzlibrary{positioning,shapes,fit,arrows,shadows,matrix,arrows.meta,decorations.pathreplacing,decorations.markings}

\numberwithin{equation}{section}
\hypersetup{  pdfborder={0 0 0},
	colorlinks   = true,
	urlcolor     = blue,
	linkcolor    = blue,
	citecolor   = red}

\makeatletter
\newcommand\bigcdot{\mathpalette\bigcdot@{.5}}
\newcommand\bigcdot@[2]{\mathbin{\vcenter{\hbox{\scalebox{#2}{$\m@th#1\bullet$}}}}}
\makeatother

\newcommand{\N}{\mathbb{N}}
\newcommand{\Z}{\mathbb{Z}}
\newcommand{\R}{\mathbb{R}}
\newcommand{\Q}{\mathbb{Q}}

\newcommand{\cT}{\mathcal{T}}

\newcommand{\cP}{\mathcal{P}}
\newcommand{\fG}{\mathfrak{G}}

\DeclareMathOperator{\St}{St}
\DeclareMathOperator{\Sym}{Sym}
\DeclareMathOperator{\C}{C}

\DeclareMathOperator{\Aut}{Aut}

\DeclareMathOperator{\cd}{cd}

\newcommand{\vr}{ \leqslant_{vr}}
\newcommand{\avr}{ \leqslant_{avr}}

\newcommand{\n}{\vartriangleleft}

\newcommand\restr[2]{{% we make the whole thing an ordinary symbol
		\left.\kern-\nulldelimiterspace % automatically resize the bar with \right
		#1 % the function
		\littletaller % pretend it's a little taller at normal size
		\right|_{#2} % this is the delimiter
}}

\newcommand{\littletaller}{\mathchoice{\vphantom{\big|}}{}{}{}}

\newtheorem{thm}{Theorem}[section]
\newtheorem{prop}[thm]{Proposition}
\newtheorem{lemma}[thm]{Lemma}
\newtheorem{cor}[thm]{Corollary}

\theoremstyle{definition}
\newtheorem{defn}[thm]{Definition}
\newtheorem{question}[thm]{Question}
\newtheorem{problem}[thm]{Problem}

\theoremstyle{remark}
\newtheorem{ex}[thm]{Example}
\newtheorem{rem}[thm]{Remark}

\newtheorem{notation}[thm]{Notation}
\newtheorem{convention}[thm]{Convention}

\title[Property (VRC) and virtual fibering for amalgams]{Property (VRC) and virtual fibering for amalgamated free products}
\author{Jon Merladet Uriguen} 
\author{Ashot Minasyan}\email{J.F.Merladet@soton.ac.uk, aminasyan@gmail.com}
\address{CGTA, School of Mathematical Sciences, University of Southampton, Highfield, Southampton, SO17~1BJ, United Kingdom}


\begin{document}

	\begin{abstract} This paper focuses on studying properties of amalgamated free products $G=G_1*_{G_0} G_2$, where the amalgamated subgroup $G_0$ is virtually cyclic.
    
    First, we prove that if the factors $G_1$ and $G_2$ are finitely generated virtually abelian groups then $G$ can be mapped  
    to another virtually abelian group so that this homomorphism is injective on each factor.
	 
     We then present several applications of this result. In particular, we show that if $G_1$ and $G_2$ have property (VRC) (that is, every cyclic subgroup is a virtual retract), then the same is true for  $G$. We also prove that $G$ inherits some residual properties (such as residual finiteness or virtual residual solvability)  from the factors $G_i$, provided $G_0$ is a virtual retract of $G_i$, for $i=1,2$.
     
     Finally, we give necessary and sufficient conditions for $G$ to be (virtually) $F_m$-fibered. In particular, we fully characterize when an amalgamated product of two (finitely generated free)-by-cyclic groups over a cyclic subgroup is free-by-cyclic or virtually free-by-cyclic. 
	\end{abstract}

	\maketitle
	
%%%%%%%%%%%%%%%%%%%                             %%%%%%%%%%%%%%%%
%%%%%%%%%%%%%%%%%%%           New Section       %%%%%%%%%%%%%%%%
%%%%%%%%%%%%%%%%%%%                             %%%%%%%%%%%%%%%%
\section{Introduction}
In this paper we study  free products of groups amalgamating  virtually cyclic subgroups. Our main technical result is the following statement.

\begin{thm}\label{thm:amalg_of_virt_ab} Let $G=G_1*_{G_0} G_2$ be an amalgamated free product of two finitely generated virtually abelian groups $G_1,G_2$ over a common virtually cyclic subgroup $G_0$. Then there exists a finitely generated virtually abelian group $E$ and a homomorphism $\nu: G \to E$ such that the restriction of $\nu$ to $G_i$ is injective, for $i=1,2$. 
\end{thm}

While the statement of \Cref{thm:amalg_of_virt_ab} is elementary, its proof requires significant work. As the reader will see below, this theorem has plenty of applications, and we believe that it will become a part of the standard toolkit for working with amalgamated free products over virtually cyclic subgroups. 

The idea is that \Cref{thm:amalg_of_virt_ab} can be applied to amalgamated free products 
\begin{equation}\label{eq:amalg-G}
G=G_1*_{G_0} G_2,~\text{ where } \end{equation}
$G_0 $ is virtually cyclic and is a virtual retract of $G_i$  (which tends to happen quite often: see the discussion of property (VRC) below), for $i=1,2$,
but $G_1$ and $G_2$ are not necessarily virtually abelian. Indeed, in this case $G_i$ will have many homomorphisms to virtually abelian groups $G_i'$ that are injective on $G_0$ (see, \cite[Corollary~3.8]{MM-vr_in_free_constr}  or \Cref{lem:further_props_of_VRC} below), $i=1,2$. By the universal property of amalgamated free products, these homomorphisms can then be combined to produce homomorphisms $G \to G_1'*_{G_0} G_2'$, which, in view of \Cref{thm:amalg_of_virt_ab}, give rise to homomorphisms from $G$ to virtually abelian groups $E$ that are injective on $G_0$ and ``remember'' much information about $G$, see \Cref{prop:maps_from_amalgams_to_vab_gps}. \Cref{thm:amalg_of_virt_ab} also has structural consequences for any amalgamated product $G=G_1*_{G_0} G_2$ as in \eqref{eq:amalg-G}. For example, when $|G_0|=\infty$ it implies that $G$ has a finite index normal subgroup $K$ splitting as an extension of a finite free product by $\Z$: see \Cref{cor:G_0_is_vr_in_G}.

In the case when the factors $G_1$ and
$G_2$ are abelian, the claim of \Cref{thm:amalg_of_virt_ab} holds without requiring $G_0$ to be virtually cyclic (this follows from the construction of central products, see \cite[Lemma~7.4]{MM-vr_in_free_constr}). However, the \emph{virtually} abelian case is more subtle, as the following example shows.

\begin{ex}\label{ex:SLnZ-intro}  
Let $n \ge 2$, $G_0\coloneq \Z^n$ and let $x,y\in \mathrm{GL}(n,\Z)$ be two finite order matrices  such that their product has infinite order. Set $G_1=\coloneq G_0 \rtimes \langle x \rangle$, $G_2\coloneq G_0 \rtimes \langle y \rangle$ and $G \coloneq G_1*_{G_0} G_2$. If $E$ is a group and  $\varphi:G \to E$ is a homomorphism that is injective on $G_0$ then the image of $\varphi(xy)$ will act on $\varphi(G_0) \cong \Z^n$ as an infinite order matrix. In particular, no power of $\varphi(xy)$ will commute with a finite index subgroup of $\varphi(G_0)$, showing that $E$ cannot be virtually abelian.

In fact, by \cite[Example~1.4 and Corollary~6.7]{MM-vr_in_free_constr}, 
there is a homomorphism from $G$ to a virtually abelian group that is injective on each factor if and only if $\langle x,y \rangle $ is a finite subgroup of $\mathrm{GL}(n,\Z)$.
\end{ex}

We will now present several applications of \Cref{thm:amalg_of_virt_ab}, starting from more basic ones and ending with our main application to virtual fibering in Subsection~\ref{subsec:app_to_v_fib}.

\subsection{Applications to (VRC) and residual properties}
A subgroup $H$ of a group $G$ is a \emph{virtual retract} (denoted $H \vr G$) if $H$ is a retract of a finite index subgroup $K \leqslant_f G$ (that is, there is a homomorphism $\rho:K \to H$ such that $\rho|_{H}=\mathrm{Id}_H$). 
A group $G$ has property \emph{(VRC)} if every cyclic subgroup is a virtual retract;  equivalently, $G$ has (VRC) if and only if 
for every $g \in G$ there is a finitely generated virtually abelian group $P$ and a homomorphism $\varphi:G \to P$ such that $\varphi$ is injective on $\langle g \rangle$ (see \cite[Corollary~1.3]{MM-vr_in_free_constr}). In a finitely generated group $G$ with property (VRC), every finitely generated virtually abelian subgroup $H$ is a virtual retract; in particular, $H$ is closed in the profinite topology and undistorted (i.e., quasi-isometrically embedded) in $G$ \cite{virtprops}. In \cite{MM-vr_in_free_constr} it was observed that property (VRC) naturally generalizes property RFRS, introduced by Agol \cite{Agol-virt_fib}.

Basic examples of groups with (VRC) are \emph{virtually special} groups in the sense of Haglund and Wise \cite{Haglund-Wise} (that is, groups that virtually embed in right angled Artin groups), see \cite[Corollary~1.6]{virtprops}. Property (VRC) is known to be preserved under taking direct products, amalgamated free products and HNN-extensions over finite subgroups, graph products, and wreath products with abelian base (see \cite{virtprops}). For fundamental groups of finite graphs of groups, (VRC) was studied in \cite{MM-vr_in_free_constr}. In particular, it is shown in \cite{MM-vr_in_free_constr} that when the vertex groups are finitely generated and virtually abelian, property (VRC) implies that the fundamental group of the graph of groups admits a geometric action on a complete CAT($0$) space.

We use \Cref{thm:amalg_of_virt_ab} to obtain the following result, which positively answers \cite[Question~11.5]{virtprops}.
\begin{thm}\label{thm:amalgam_of_(VRC)_gps}
Property (VRC) is stable under taking amalgamated free products over virtually cyclic subgroups.
\end{thm}

% An inductive argument then shows that the fundamental group of a finite tree of groups where vertex groups have (VRC) and edge groups are virtually cyclic has (VRC).

Let us mention one application of \Cref{thm:amalgam_of_(VRC)_gps} to tubular groups, introduced by Wise in \cite{Wise-tubular}. Recall that a group $G$ is said to be \emph{tubular} if it is isomorphic to the fundamental group of a finite graph of groups $(\mathcal{G},\Gamma)$, where all vertex groups are free abelian of rank $2$ and all edge groups are infinite cyclic. Button \cite{Button-free-by-cyclic} proved that when the underlying graph $\Gamma$ is a tree, the corresponding tubular group is virtually special (and hence it has (VRC)), and in \cite{MM-vr_in_free_constr} the authors extended this statement to graphs of non-negative Euler characteristic, provided the tubular group is \emph{balanced} (that is, for any $g \in G$ of infinite order, if $g^k$ is conjugate to $g^l$ then $l= \pm k$). 

We will say that a non-empty finite connected graph $\Gamma$ is \emph{cycle-disjoint} if every vertex of $\Gamma$ is contained in at most one  non-trivial simple cycle (note that we allow loops and multiple edges).

\begin{cor}\label{cor:cacti} Let $G$ be a tubular group such that the underlying graph $\Gamma$ is cycle-disjoint. If  $G$ is balanced then all of the following hold: 
\begin{itemize}
    \item[(i)] $G$ has (VRC);
    \item[(ii)] $G$  is virtually special and CAT($0$);
    \item[(iii)] $G$ is virtually (finitely generated free)-by-$\Z$.    
\end{itemize}
\end{cor}

\begin{proof} Any  cycle-disjoint graph $\Gamma$ can be constructed starting from a single vertex by inductively joining leaf edges or cycles, with the condition that a cycle can only be attached to a vertex of degree at most $1$. Thus, if $G$ is a tubular group, with the underlying graph $\Gamma$, then $G$ can be obtained by iteratively amalgamating free abelian groups of rank $2$ and tubular groups corresponding to cycle graphs along infinite cyclic subgroups. According to \cite[Corollary~1.8]{MM-vr_in_free_constr}, any balanced tubular group corresponding to a cycle has (VRC), hence $G$ has (VRC) by \Cref{thm:amalgam_of_(VRC)_gps}. Therefore, $G$ is virtually special, CAT($0$), and virtually (finitely generated free)-by-$\Z$, by \cite[Propositions~1.9, 1.10 and 12.3]{MM-vr_in_free_constr}.    
\end{proof}

A group $G$ satisfying one of the conditions (i)--(iii) from \Cref{cor:cacti} is always balanced (cf. \cite[Remark~9.2]{MM-vr_in_free_constr}). The requirement on the structure of the graph $\Gamma$ is necessary: \cite[Corollary~12.4]{MM-vr_in_free_constr} provides many examples of balanced tubular groups $G_k$, where the underlying graph consists of one vertex and two loops at that vertex, such that $G_k$ does not satisfy any of the properties (i)--(iii).

\Cref{thm:amalg_of_virt_ab} is also useful for studying residual properties.
\begin{cor}\label{cor:virt_res_solv} Suppose that $G=G_1*_{G_0} G_2$, where 
$G_0$ is virtually cyclic and $G_0 \vr G_i$, for $i=1,2$.
\begin{itemize}
    \item[(i)] If $G_1$, $G_2$ are residually finite then so is $G$.
    \item[(ii)] If $G_1$, $G_2$ are virtually residually solvable then so is $G$.
\end{itemize}
\end{cor}
Claim (i) of \Cref{cor:virt_res_solv}  probably can also be established using the methods of Evans \cite{Evans} and Burillo--Martino \cite{Bur-Mar}, but we are not aware of ways to prove (ii) avoiding \Cref{thm:amalg_of_virt_ab}.
The assumption that $G_0$ is a virtual retract in $G_i$ is important: in \cite{Kar-Nik} Kar and Nikolov constructed an example of a double of $\mathrm{SL}(3,\Z[1/2])$ over a cyclic subgroup that is not residually amenable (hence, it is neither residually finite nor virtually residually solvable
\cite[Lemma~2.3]{Berlai}). Note that $\mathrm{SL}(3,\Z[1/2])$ is both residually finite and virtually residually solvable (in fact, it is virtually residually $p$-finite, for all but finitely many primes $p$ by Platonov's theorem \cite{Plat}). 

Since right angled Artin groups are residually nilpotent \cite[Chapter~3]{Droms-thesis}, 
\Cref{cor:virt_res_solv}.(ii) implies that amalgamated free products of virtually special groups over virtually cyclic subgroups are virtually residually solvable. Note that even if both factors are finite and solvable, their amalgamated free product may not be residually solvable: see \Cref{ex:amalg_not_res_solv} below.

\subsection{Applications to virtual \texorpdfstring{$F_m$}{Fm}-fibering}\label{subsec:app_to_v_fib}
Our main application of \Cref{thm:amalg_of_virt_ab} is to virtual fibering results for amalgamated free products. This application is motivated by the following basic question.

\begin{question}\label{q:amalg_are_free-by-cyclic} When is an amalgamated free product of two free-by-cyclic groups over a cyclic subgroup free-by-cyclic? When is it virtually free-by-cyclic?
\end{question}

\begin{convention}
In this paper by a \emph{free-by-cyclic group} we will mean a group isomorphic to a semidirect product of a \emph{finite rank} free group with $\Z$. On the other hand, groups splitting as semidirect products $F \rtimes \Z$, for an arbitrary free group $F$, will be called \emph{$F$-by-$\Z$},    
\end{convention}

Showing that various groups are virtually $F$-by-$\Z$ has received significant attention  recently: see the papers of Hagen--Wise \cite{Hagen-Wise-spec_gps_with_hier}, Kielak--Linton \cite{Kielak-Linton} and Fisher \cite{Fisher-on_cd_of_ker_to_Z}. However, all of these results assume that the group is virtually special or virtually RFRS, which is not sufficient to answer \Cref{q:amalg_are_free-by-cyclic}. Indeed, in \cite{Wu-Ye} Wu and Ye produced examples of cyclic amalgamations of free-by-cyclic groups  that are not virtually $F$-by-$\Z$. Moreover, in \Cref{ex:non-RFRS_amalg_of_free-by-Z} below, we use our results to construct an amalgam of two free-by-cyclic groups over a cyclic subgroup that is virtually free-by-cyclic but is not virtually RFRS.

We give complete answers to both parts of \Cref{q:amalg_are_free-by-cyclic} in Corollaries~\ref{cor:amalg_is_free-by-cyclic} and \ref{cor:virt_free-by-cyclic} below, and our answer to the second part (about being virtually free-by-cyclic) uses \Cref{thm:amalg_of_virt_ab} in a crucial way. We also give new sufficient criteria in the more general case when the factors $G_1$, $G_2$ are $F$-by-$\Z$, see \Cref{cor:amalg_is_F-by_Z}.

% Of course, in order to be free-by-cyclic in the above sense a group needs to have an epimorphism onto $\Z$ with finitely generated kernel. In particular, it must be fibered, and a virtually free-by-cyclic group must be virtually fibered. 

An investigation of \Cref{q:amalg_are_free-by-cyclic} naturally leads to studying (virtual) fibering of amalgamated free products.

\begin{defn}
Given $m \in \N$ and a group $G$, we will say that this group \emph{$F_m$-fibers} (or that it is \emph{$F_m$-fibered}) if there is an  non-zero homomorphism $\chi: G \to \Z$  such that $\ker\chi$ is of type $F_m$. In the case when $m=1$, we will simply say that $G$ \emph{fibers} (or that $G$ is \emph{fibered}).

Similarly, the group $G$ \emph{virtually $F_m$-fibers} (\emph{virtually fibers}) if there is a finite index subgroup $H \leqslant_f G$ such that $H$ $F_m$-fibers (respectively, fibers).
\end{defn}

\begin{notation}\label{not:G} Further in this subsection we assume that $G=G_1*_{G_0} G_2$, where $G_0$ is virtually cyclic and $G_0\lneqq G_i$, $i=1,2$.     
\end{notation}

\begin{prop}\label{prop:crit_for_fibering_of_amalg} Suppose that for some $m \in \N$, $G_1$ and $G_2$ are of type $F_m$. Then the amalgamated product $G$ $F_m$-fibers if and only if the following two conditions hold for each $i=1,2$:
\begin{itemize}
    \item[(i)] $G_i$ $F_m$-fibers;
    \item[(ii)] the natural image of $G_0$ in the abelianization $G_i/[G_i,G_i]$ is infinite.
\end{itemize}
\end{prop}

In fact condition (ii) of \Cref{prop:crit_for_fibering_of_amalg} can be strengthened to say that $G_0$ must have non-zero image under {every} epimorphism $G \to \Z$ with finitely generated kernel (in particular, the image of $G_0$ in the abelianization of $G$ is infinite): see \Cref{prop:graph_of_gps_crit_for_chi_to_be_in_higher_invar} below, which describes the higher discrete Bieri--Neumann--Strebel--Renz (BNSR) invariants for fundamental groups of finite graphs of groups with virtually polycyclic edge groups, extending the work of  Cashen--Levitt \cite{Cash-Lev}, who did this in the case $m=1$. The other direction (sufficiency) in \Cref{prop:crit_for_fibering_of_amalg}, is an elementary consequence of the openness of the higher BNSR invariants.

We can now obtain an answer to the first part of \Cref{q:amalg_are_free-by-cyclic}.
\begin{cor}\label{cor:amalg_is_free-by-cyclic} Suppose $G_1$, $G_2$ are free-by-cyclic and $G_0$ is cyclic. Then 
$G=G_1*_{G_0} G_2$ is free-by-cyclic if and only if $|G_0|=\infty$ and $G_0 \cap [G_i,G_i]=\{1\}$ in $G_i$, for each $i=1,2$.  \end{cor}

\begin{proof} The necessity is given by \Cref{prop:crit_for_fibering_of_amalg}, applied to the case when $m=1$. For the sufficiency, since finite rank free groups are finitely presented, we can use the same proposition for $m=2$, to conclude that $G$ has a finitely presented normal subgroup $N \n G$ such that $G/N \cong \Z$.
Since $G_i$ are free-by-cyclic and $G_0 \cong \Z$, the standard properties of cohomological dimension (see  \cite[Chapter~II]{Bieri-book}) imply that $\cd(G) \le 2$, so $N$ is free by a result of Bieri \cite[Corollary~8.6 in Chapter~II]{Bieri-book}. Thus $G$ is free-by-cyclic.    
\end{proof}

For the second part of part of \Cref{q:amalg_are_free-by-cyclic}, we need an analogue of \Cref{prop:crit_for_fibering_of_amalg} for virtual fibering. 
We say that a subgroup $A$ of a group $G$ is an \emph{almost virtual retract}, denoted $A\avr G$, if there is a finite index subgroup $H \leqslant_f A $ such that $H \vr G$. The next theorem uses \Cref{not:G}.

\begin{thm}\label{thm:virt_fib_of_amalg}
Suppose that for some $m \in \N$, $G_1$ and $G_2$ are of type $F_m$. 

(a)  If $G$ virtually $F_m$-fibers and is not virtually abelian then all of the following must be true:
    \begin{itemize}
    \item $G_i$ virtually $F_m$-fibers, for $i=1,2$; 
    \item $|G_0|=\infty$; 
    \item $G_0 \avr G_i$, for $i=1,2$ (in fact, $G_0 \avr G$).
\end{itemize}

(b) The group $G$ virtually $F_m$-fibers provided all of the following hold:
\begin{itemize}
    \item $G_i$ is $F_m$-fibered, for $i=1,2$; 
    \item $|G_0|=\infty$; 
    \item $G_0 \vr G_i$, for $i=1,2$.
\end{itemize}
\end{thm}

\begin{rem}\label{rem:G_is_virt_ab} The Normal Form Theorem for amalgamated free products \cite[Theorem~IV.2.6]{LS} easily implies that the group $G$ from \Cref{not:G} contains a free subgroup of rank $2$, provided $\max\{|G_1:G_0|,|G_2:G_0|\} \ge 3$. Thus, $G$
is virtually abelian if
 and only if $|G_i:G_0|=2$, for $i=1,2$, in which case it has a finite index subgroup splitting as a direct product $G_0 \times \Z$. Of course such a direct product $F_m$-fibers even when $|G_0|<\infty$, so the assumption that $G$ is not virtually abelian in \Cref{thm:virt_fib_of_amalg}.(a) cannot be dropped.
\end{rem}

Claim (a) of \Cref{thm:virt_fib_of_amalg} is established using a ``virtual'' version of \Cref{prop:graph_of_gps_crit_for_chi_to_be_in_higher_invar} (see \Cref{prop:necessary_crit_for_virt_fib}), while claim (b) uses a new criterion for fibering of fundamental groups of graphs of groups (\Cref{thm:fib_crit_for_graph_of_gps}) and \Cref{prop:maps_from_amalgams_to_vab_gps}. The latter proposition is a consequence of \Cref{thm:amalg_of_virt_ab} and implies that, under the assumptions of claim (b), $G_0$ is ``visible'' in the character sphere of a finite index subgroup of $G$.

Part (b) of  \Cref{thm:virt_fib_of_amalg} does not seem to have any analogues in the literature. The most powerful known virtual fibering results, due to Kielak \cite{Kielak2020} and Fisher \cite{Fisher}, require that the group $G$ virtually satisfies Agol's condition RFRS from \cite{Agol-virt_fib}. But \Cref{thm:virt_fib_of_amalg}.(b) applies in many cases where the group $G$ is not virtually RFRS. 

\begin{ex}\label{ex:non-RFRS_amalg_of_free-by-Z}
Let $G_1= H$ be Gersten's free-by-cyclic group from \cite{Gersten}
\begin{equation}\label{eq:Gerstens_gp}
H \coloneq \langle a,b,s,t \mid [a,b]=1,~sbs^{-1}=ab,~tbt^{-1}=a^2b\rangle. 
\end{equation}
In \cite[Lemma~5.18]{Wu-Ye} Wu and Ye showed that $\langle a \rangle  \nleqslant_{vr} H$, in particular, $G_1$ does not have (VRC) and it is not virtually RFRS  (see \cite{MM-vr_in_free_constr}).  

Let $G_2\coloneq \langle x,y \mid x^p=y^q \rangle$ be the fundamental group of the complement of a torus knot/link, for $|p|,|q| \ge 2 $. Then $G_2$ is an amalgamated free product of two infinite cyclic groups $\langle x \rangle$ and $\langle y \rangle$, so it is fibered and has (VRC) (e.g., by \Cref{prop:crit_for_fibering_of_amalg} and \Cref{thm:amalgam_of_(VRC)_gps}; of course, in this case both statements were known previously). Finally, consider the amalgamated free product
\[G \coloneq G_1*_{\langle b \rangle=\langle[x,y]\rangle} G_2.\]
Clearly, $\langle b \rangle $ is a retract of $G_1$ (under the map sending $a,s,t$ to $1$) and $\langle [x,y] \rangle \vr G_2$, as $G_2$ has (VRC). Therefore, $G$ is virtually fibered by \Cref{thm:virt_fib_of_amalg}.(b) (in fact, it is virtually free-by-cyclic, by \Cref{cor:virt_free-by-cyclic} below). However, $G$ is not fibered by \Cref{prop:crit_for_fibering_of_amalg} (because $[x,y] \in [G_2,G_2])$ and it is not virtually RFRS (because $G_1$ is not).
\end{ex}

\begin{cor}
Suppose that for some $m \in \N$, $G_1$ and $G_2$ are $F_m$-fibered and have (VRC). Then the amalgamated free product $G$, from \Cref{not:G}, is virtually $F_m$-fibered and has (VRC).
\end{cor}

The reader will notice that there is a gap between the necessary conditions for virtual fibering of $G$, provided by part (a) of \Cref{thm:virt_fib_of_amalg}, and the sufficient conditions, given by part (b). One difference is that part (a) only concludes that $G_i$ is \emph{virtually} $F_m$-fibered, for $i=1,2$, while part (b) requires both factors to be $F_m$-fibered on the nose. This disparity is the subject of \Cref{q:vir_fibered}. The second difference is that part (a) implies that $G_0$ is an \emph{almost} virtual retract of $G_i$ (and of $G$), $i=1,2$, while part (b) requires it to be a virtual retract. Examples~\ref{ex:ind_avr_but_not_virt_fibered} and \ref{ex:second_non_rf_example} show that this second difference is essential; however,  it can only occur if  at least one of the factors $G_i$ is not residually finite and $G_0$ has non-trivial finite normal subgroups (see \Cref{lem:vab+avr+t-f=>vr}).

The next corollary answers the second part of \Cref{q:amalg_are_free-by-cyclic}.
\begin{cor}\label{cor:virt_free-by-cyclic}
Using \Cref{not:G}, assume that $G_1$, $G_2$ are free-by-cyclic groups and $G_0$ is cyclic. Then $G$ is virtually free-by-cyclic if and only if $G_0$ is infinite and $G_0 \vr G_i$, for $i=1,2$.   
\end{cor}

\begin{proof} If $G$ is virtually free-by-cyclic then it virtually fibers. Hence,  \Cref{thm:virt_fib_of_amalg}.(a) tells us that $|G_0|=\infty$ and $G_0 \avr G_i$, for $i=1,2$. It follows that $G_0 \cong \Z$, so $G_0 \vr G_i$, $i=1,2$, by \Cref{lem:vab+avr+t-f=>vr}.

Conversely, if $|G_0|=\infty$ and $G_0 \vr G_i$, for $i=1,2$, then according to \Cref{thm:virt_fib_of_amalg}.(b),   there is  $K \leqslant_f G$ and a finitely presented normal subgroup $N \n K$ with $K/N \cong \Z$. The argument from the proof of \Cref{cor:amalg_is_free-by-cyclic} (using cohomological dimension) shows that $N$ must be free, so $G$ is virtually free-by-cyclic.
\end{proof}

A non-Euclidean Baumslag--Solitar group
\[BS(k,l)=\langle a,t \mid ta^kt^{-1}=a^l \rangle,~k,l \in \Z\setminus\{0\},~|k| \neq |l|, \]
is never virtually free-by-cyclic. This shows that neither \Cref{cor:virt_free-by-cyclic} nor part (b) of \Cref{thm:virt_fib_of_amalg} can be easily extended to HNN-extensions (see also \Cref{ex:Gerstens_amalgam} for a   more sophisticated construction, where the edge groups are all retracts of the resulting group). 

Although so far we have restricted ourselves to considering free-by-cyclic groups, where the free normal subgroup is finitely generated, by combining our results with recent work of Fisher \cite{Fisher-Novikov_cohom}, we can extend the sufficiency implications in Corollaries~\ref{cor:amalg_is_free-by-cyclic}, \ref{cor:virt_free-by-cyclic} to  finitely generated $F$-by-$\Z$ groups. 

\begin{cor}\label{cor:amalg_is_F-by_Z}
Suppose that $G=G_1*_{G_0} G_2$, where $G_1,G_2$ are finitely generated $F$-by-$\Z$ groups and $G_0$ is infinite cyclic. 
\begin{itemize}
    \item[(i)] If $G_0 \cap [G_i,G_i]=\{1\}$, for each $i=1,2$, then $G$ is $F$-by-$\Z$.  
    \item[(ii)] If $G_0 \vr G_i$, for each $i=1,2$, then $G$ is virtually $F$-by-$\Z$.  
\end{itemize}
\end{cor}
Unlike Corollaries~\ref{cor:amalg_is_free-by-cyclic} and \ref{cor:virt_free-by-cyclic}, the conditions listed in parts (i) and (ii) of \Cref{cor:amalg_is_F-by_Z} are not necessary. For example, the converse of (i) fails for the  fundamental group $G$ of a closed surface of genus $2$, and the converse of (ii) fails for $G \coloneq H*_{\langle a \rangle=\langle x \rangle} F_2$, where $H$ is Gersten's group \eqref{eq:Gerstens_gp} and $F_2$ is the free group with free basis $\{x,y\}$. 

\subsection{Outline of the paper}
In \Cref{sec:prelims} we give the necessary background on the profinite topology, virtual retractions,  and graphs of groups.
Sections~\ref{sec:splitting}--\ref{sec:proof_of_thm} concern the proof of \Cref{thm:amalg_of_virt_ab} (see Subsection~\ref{subsec:proof_idea} for the idea behind this proof). In \Cref{sec:amalgams_of_VRC} we deduce from \Cref{thm:amalg_of_virt_ab} the key \Cref{prop:maps_from_amalgams_to_vab_gps}, which then allows us to establish \Cref{thm:amalgam_of_(VRC)_gps}. In \Cref{sec:res_props} we use \Cref{prop:maps_from_amalgams_to_vab_gps} to prove \Cref{cor:virt_res_solv}. \Cref{sec:fin_props} discusses necessary and sufficient criteria for a normal subgroup of the fundamental group of a graph of groups to be finitely generated and to satisfy higher finiteness properties. In \Cref{sec:BNSR} we collect the necessary background on the Bieri--Neumann--Strebel--Renz invariants that play a key role in our fibering and virtual fibering results. In \Cref{sec:fib_graphs_of_gps} we prove \Cref{prop:graph_of_gps_crit_for_chi_to_be_in_higher_invar}, characterizing when the kernel of a discrete character $\chi:G \to \R$ is of type $F_m$, where $G$  is the fundamental group of a finite graph of groups with virtually polycyclic edge groups. We then use it to prove \Cref{prop:crit_for_fibering_of_amalg} and to give a new fibering criterion (\Cref{thm:fib_crit_for_graph_of_gps}) playing an important role in the results about virtual fibering in \Cref{sec:virt_fibering}, where it is combined with \Cref{prop:maps_from_amalgams_to_vab_gps} to prove \Cref{thm:virt_fib_of_amalg}. In \Cref{sec:adapt},  we show that our argument for part (b) of \Cref{thm:virt_fib_of_amalg} actually works for more general ``rationally open'' invariants, and use this to establish \Cref{cor:amalg_is_F-by_Z}. Finally, in \Cref{sec:examples} we discuss natural open problems and give several examples demonstrating the sharpness of our results.

\subsection{Idea of the proof of \texorpdfstring{\Cref{thm:amalg_of_virt_ab}}{Theorem 1.1}}\label{subsec:proof_idea}
Let us briefly describe the main steps in the proof of \Cref{thm:amalg_of_virt_ab} in the harder case when $|G_0|=\infty$. First, for $i=0,1,2$, we embed each group $G_i$ as a finite index subgroup in a larger virtually abelian group $P_{i}$, in such a way that 

\begin{enumerate}
    \item\label{it:(3)} the original embeddings $G_0 \hookrightarrow G_i$ extend to embeddings $P_{0} \hookrightarrow P_{i}$, $i=1,2$;
    \item\label{it:(1)} $P_{0}$ splits as a semidirect product $B_0 \rtimes Q$, where $B_0 \cong \Z$ and $|Q|<\infty$;
    \item  $P_{i}$ has a finite index free abelian subgroup $B_i \n P_{i}$ such that $P_{0}~\cap~B_i=B_0$ in $P_{i}$, for $i=1,2$.

\end{enumerate}
To achieve this, in \Cref{sec:splitting} we introduce a category $\mathcal C$, of group pairs, and embedding functors $\mathcal{F}_n: \mathcal{C} \to \mathcal{C}$, $n \in \N$. An object $(G,A) \in \mathcal{C}$ consists of a group $G$ and a torsion-free normal abelian subgroup $A \n G$. For each $n \in \N$, $\mathcal{F}_n((G,A))=(P,B)$, where $B \cong A$, $P/B \cong G/A$ and there is a monomorphism $G \to P$, sending $A$ inside $B$. And if $n$ is divisible by $|G/A|=|P/B|$ then  $P$ splits as a semidirect product $B \rtimes (P/B)$. 

Property \eqref{it:(3)} allows us to embed the amalgamated free product $G_1*_{G_0} G_2$ into the amalgamated free product $P_{1}*_{P_{0}} P_{2}$, which has an easier structure. In \Cref{sec:wreath} we  simplify this structure further by embedding $P_{1}$ and $P_{2}$ into permutational wreath products $E_1 \coloneq \Z\wr_{\Omega_k} S_k$ and   $E_2 \coloneq \Z\wr_{\Omega_l} S_l$, respectively, where $k,l \in \N$ and $S_k$ is the symmetric group of $\Omega_k \coloneq \{1,\dots,k\}$. We ensure that the images of $B_0$ (see \eqref{it:(1)}) under these embeddings are in the base groups $\Z^{\Omega_k}$ and $\Z^{\Omega_l}$, and the images of $Q$ are contained in the symmetric groups $S_k$ and $S_l$ and act freely on $\Omega_k$ and $\Omega_l$, respectively. 

In \Cref{sec:proof_of_thm}, we first construct monomorphisms from $E_1$ and $E_2$ to $E \coloneq \Z\wr_{\Omega_{kl}} S_{kl}$, which agree on the subgroup $B_0 \n_f P_0$, and then modify the first of these monomorphisms (by composing it with an inner automorphism of $S_k$) to ensure that they also agree on $Q$. Since $P_{0}=B_0 \rtimes Q$, this gives rise to a homomorphism $E_1*_{P_{0}} E_2 \to E$, which we then use to produce the desired  homomorphism $\nu: G_1*_{G_0} G_2 \to E$ that is injective on $G_i$, for $i=1,2$. 

\subsection*{Acknowledgments} This paper has  benefited greatly from the authors' discussions with Mark Hagen, Sam Hughes, Andrei Jaikin-Zapirain, Dawid Kielak, Peter Kropholler and Ian Leary. The second author would also like to thank the Isaac Newton Institute for Mathematical Sciences, Cambridge, for support and hospitality during the programme ``Actions on graphs and metric spaces'', where some of his work on this paper was undertaken. This programme was supported by EPSRC grant EP/Z000580/1.

%%%%%%%%%%%%%%%%%%%                             %%%%%%%%%%%%%%%%
%%%%%%%%%%%%%%%%%%%           New Section       %%%%%%%%%%%%%%%%
%%%%%%%%%%%%%%%%%%%                             %%%%%%%%%%%%%%%%

\section{Preliminaries}\label{sec:prelims}
\subsection{Profinite topology and separability}
Recall that the \emph{profinite topology} on a group $G$ is the topology whose basic open sets are cosets of finite index subgroups in $G$. A subset of $G$ is said to be \emph{separable} if it is closed in the profinite topology.

\begin{defn}\label{def:sep_props} A group $G$ is said to be 
\begin{itemize}
    \item \emph{residually finite} if $\{1\}$ is separable in $G$;
    \item \emph{cyclic subgroup separable} if every cyclic subgroup is separable in $G$;
    \item \emph{LERF} if every finitely generated subgroup is separable in $G$.
\end{itemize}
   \end{defn}

Clearly, LERF-ness implies cyclic subgroup separability and cyclic subgroup separability implies residual finiteness.  The following facts are well-known.

\begin{rem}\label{rem:basic_sep_facts}
Let $G$ be a group.
\begin{itemize}
    \item A subgroup $H \leqslant G$ is separable if and only if $H$ is the intersection of a collection of finite index subgroups in $G$.
    \item  A normal subgroup $N \n G$ is separable if and only if $G/N$ is residually finite.
    \item If $G$ is finitely generated and virtually abelian then every subgroup of $G$ is separable.
\end{itemize}    
\end{rem}

\begin{defn} 
The \emph{finite residual} $R(G)$, of a group $G$, is the intersection of all finite index subgroups of $G$.   
\end{defn}

The following basic fact about the finite residual is easy to verify.
\begin{rem}\label{rem:props_of_finite_residual} If $G$ is a group then $R(G)$ is the smallest normal subgroup of $G$ such that $G/R(G)$ is residually finite. In other words, $G/R(G)$ is residually finite and if $G/N$ is residually finite, for $N\n G$, then $R(G) \subseteq N$.
\end{rem}

\begin{lemma}\label{lem:fin_residual_of_fin_ind_sbgp} If $G$ is a group and $L \leqslant G$ then $R(L) \subseteq R(G)$. Moreover, if $|G:L|<\infty$ then
$R(L)=R(G)$.  
\end{lemma}

\begin{proof} The first claim follows from the fact that $|L:(L \cap K)|<\infty$, for every $K \leqslant_f G$. And if $|G:L|<\infty$ then
every finite index subgroup of $L$ has finite index in $G$, so we have $R(G) \subseteq R(L)$, and the second claim holds.
\end{proof}

\subsection{Virtual retractions}
Recall that a subgroup $H$ is a virtual retract of a group $G$ ($H \vr G$) if there is $K \leqslant_f G$ such that $H \subseteq K$ and there is a retraction $\rho:K \to H$. 
Equivalently, $H \vr G$ if there is a subgroup $N \leqslant G$ such that $N$ is normalized by $H$, $H \cap N=\{1\}$ and $|G:HN|<\infty$.

\begin{lemma}[{\cite[Lemma~3.2]{virtprops}}]\label{lem:basic_props_of_virt_retr} Let $G$, $A$ be  groups, let $H \leqslant K \leqslant G$ be subgroups of $G$, and let $\psi:G \to A$ be a group homomorphism. 
\begin{itemize}
    \item[(a)] If $\psi(H) \vr A$ and $\psi$ is injective on $H$ then $H \vr G$.
    \item[(b)] If $H \vr K$ and $K \vr G$ then $H \vr G$. In particular, if $K \vr G$ and $H \leqslant_f K$ then $H \vr G$.
    \item[(c)] If  $H \vr G$ then $H \vr K$.
\end{itemize}
\end{lemma}

\begin{lemma}[{\cite[Proposition~1.5 and Lemma~5.1.(ii)]{virtprops}}]\label{lem:VRC->vabs_sep} If $G$ has (VRC) then every finitely generated virtually abelian subgroup is a virtual retract and is separable in $G$. In particular, $G$ is cyclic subgroup separable.    
\end{lemma}

\begin{lemma}[{\cite[Corollary~4.3]{virtprops}}]\label{lem:vabs_have_LR} Every subgroup of a finitely generated virtually abelian group is a virtual retract.
\end{lemma}

The following statement is similar to \cite[Lemma~4.1]{virtprops}.
\begin{lemma}\label{lem:extend_homom_from_K} Let $G$ be a group with a finite index subgroup $K \leqslant_f G$. Suppose that $N \n K$ is such that $K/N$ is finitely generated and virtually abelian. Then there exists a finitely generated virtually abelian group $A$ and a homomorphism $\varphi:G \to A$ such that $\ker \varphi \subseteq N$.    
\end{lemma}

\begin{proof} Let $L \n_f G$ be the intersection of all conjugates of $K$ in $G$ and set $M \coloneq N \cap L \n L$. Then there are only finitely many different conjugates of $M$ in $G$: $M_1=M,M_2,\dots,M_k \n L$, and 
\[L/M_i \cong L/M \leqslant_f K/N\] is finitely generated and virtually abelian, for all $i=1,\dots,k$.
Observe that the subgroup $O \coloneq \bigcap_{i=1}^k M_i$ is normal in $G$ and we have a diagonal embedding $ L/O \hookrightarrow \bigtimes_{i=1}^k L/M_i$. Therefore,  $L/O$ is finitely generated and virtually abelian. By construction, $L/O \n_f G/O$ so $A \coloneq G/O$ is finitely generated and virtually abelian. It remains to define $\varphi:G \to G/O$ as the quotient map, and to note that $\ker\varphi=O \subseteq M \subseteq N$.
\end{proof}

\begin{lemma}   
\label{lem:further_props_of_VRC} Suppose that $G$ is a group, $K_i \leqslant_f G$, $N_i \n K_i$, $i=1,\dots,m$, and $H_j \vr G$, $j=1,\dots,n$. If the groups $K_i/N_i$, $i=1,\dots,m$, and $H_j$, $j=1,\dots,n$, are finitely generated and virtually abelian then there exist a finitely generated virtually abelian group $P$ and a  homomorphism $\varphi:G \to P$ such that $\ker\varphi \subseteq N_i$, for $i=1,\dots,m$, and $\varphi$ is injective on $H_j$, for each $j=1,\dots,n$. \end{lemma}
    
\begin{proof} By the assumptions, for each $j=1,\dots,n$ there exist $L_j\leqslant_f G$ and  $M_j\n L_j $ such that $H_j \subseteq L_j$, $H_j \cap M_j=\{1\}$ and $L_j=H_jM_j$. Now, \Cref{lem:extend_homom_from_K} gives finitely generated virtually abelian groups $A_1,\dots,A_m$ and $B_1,\dots,B_n$, together with  homomorphisms $\varphi_i:G \to A_i$ and  $\psi_j:G \to B_j$, such that $\ker\varphi_i \subseteq N_i$ and $\ker\psi_j \subseteq M_j$,  for all $i=1,\dots,m$ and $j=1,\dots,n$.

Evidently, the group 
\[P\coloneq A_1 \times \cdots \times A_m \times B_1 \times \dots \times B_n \] is finitely generated and virtually abelian, and the diagonal  homomorphism  
          \[
                \varphi \coloneq \varphi_1 \times \cdots \times \varphi_m \times \psi_1 \times \cdots \times \psi_n:  G \to P 
        \]
satisfies $\ker\varphi \subseteq \bigcap_{i=1}^m N_i \cap \bigcap_{j=1}^n M_j$. It follows that $\varphi$ enjoys the desired properties.
\end{proof}

\begin{defn}\label{def:almost_vr} A subgroup $A$ of a group $G$ will be called an \emph{almost virtual retract} if there is a finite index subgroup $H \leqslant_f A$ such that $H \vr G$.    In this case we shall write $A \avr G$.
\end{defn}

For example, every finite subgroup of a group $G$ is always an almost virtual retract of $G$.

An almost virtual retract $A$  of a group $G$ need not be a virtual retract of $G$ in general (see \cite[Examples~3.5 and 3.6]{virtprops}), but this is true in the case when $G$ is residually finite and $A$ is finitely generated virtually abelian \cite[Theorem~1.4]{virtprops}. This statement can be generalized as follows.

\begin{lemma}\label{lem:vab+avr+t-f=>vr}
Let $G$ be a group with a finitely generated virtually abelian subgroup $A \avr G$, and let $R \n G$ be the finite residual of $G$. Suppose that $H \leqslant_f A$ satisfies $H \vr G$. Then $H \cap R=\{1\}$ and $|A \cap R|<\infty$. Moreover, if $A \cap R=\{1\}$ then $A \vr G$. In particular, $A \vr G$ provided $G$ is residually finite or $A$ contains no non-trivial finite normal subgroups.
\end{lemma}

\begin{proof} By the assumptions, there exist $K \leqslant_f G$ and $N\n K$, such that $H \subseteq K$, $K=NH$ and $N \cap H=\{1\}$.

Since $K/N \cong H$ is residually finite, being a finitely generated virtually abelian group, $R(K) \subseteq N$ by \Cref{rem:props_of_finite_residual}, so $R=R(G) \subseteq N$ by \Cref{lem:fin_residual_of_fin_ind_sbgp}. It follows that $H \cap R=\{1\}$,  and so  $|A \cap R| \le |A:H|<\infty$.

Now, assume that $A \cap R=\{1\}$. Then the quotient map $\psi:G \to G/R$ is injective on $A$. Moreover, since $R \subseteq N$ and $N \cap H=\{1\}$ in $G$, we have $\psi(N) \cap \psi(H) =\{1\}$ in $G/R$, which implies that $\psi(H)$ is still a retract of $\psi(K) \leqslant_f G/R$, thus $\psi(A) \avr G/R$. 

According to \Cref{rem:props_of_finite_residual}, $G/R$ is residually finite. Since  $\psi(A) \cong A$ is finitely generated and virtually abelian, we can apply \cite[Theorem~1.4]{virtprops} to deduce that $\psi(A) \vr G/R$, hence $A \vr G$ by \Cref{lem:basic_props_of_virt_retr}.(a).
\end{proof}

\begin{lemma}\label{lem:virt_cyc-avr_crit} Let $A$ be an infinite virtually cyclic subgroup of a finitely generated group $G$. Then the following are equivalent:
\begin{itemize}
    \item[(i)] $A \avr G$;
    \item[(ii)] for some (equivalently, for every) infinite cyclic subgroup $C \leqslant A$ one has $C \vr G$;
  \item[(iii)] there is a homomorphism $\varphi:G \to P$, where $P$ is finitely generated and virtually abelian, such that $|\ker\varphi \cap A|<\infty$ (equivalently, $\varphi(A)$ is infinite);  
  \item[(iv)] there exist a  finite index subgroup $K \leqslant_f G$ and a homomorphism $\psi:K \to \R$ such that $\psi(K \cap A) \neq \{0\}$.
\end{itemize}
\end{lemma}

\begin{proof} If (i) holds, then there is a finite index subgroup $A' \leqslant_f A$ such that $A' \vr G$. Since $A$ is infinite virtually cyclic, there is an infinite cyclic subgroup $C \leqslant_f A'$. Therefore, $C \vr G$ by \Cref{lem:basic_props_of_virt_retr}.(b). Now, for any other infinite cyclic subgroup $C' \leqslant A$ we have $C' \cap C \leqslant_f C$, so $C' \cap C \vr G$, again by \Cref{lem:basic_props_of_virt_retr}.(b). On the other hand, $C' \cap C$ has finite index in $C'$ and $C'$ is torsion-free, hence $C' \vr G$,  by \Cref{lem:vab+avr+t-f=>vr}. Thus,  (ii) holds.

Now, suppose that (ii) is satisfied. Then, according to \Cref{lem:further_props_of_VRC}, there exist a finitely generated virtually abelian group $P$ and a homomorphism $\varphi:G \to P$ that is injective on $C$. Since $|A:C|<\infty$, we see that the intersection $\ker\varphi \cap A$ must be finite and $|\varphi(A)|=\infty$, i.e.,   (iii) holds.

If (iii) is true then choose a finitely generated free abelian finite index subgroup $L \leqslant_f P$ and set $K \coloneq \varphi^{-1}(L) \leqslant_f G$. Since every finitely generated free abelian group embeds in $\R$, the restriction of $\varphi$ to $K$ gives rise to a homomorphism $\psi:K \to \R$ such that $|\ker\psi \cap A|<\infty$. Since $K \cap A$ is infinite, we can conclude that $\psi(K \cap A) \neq \{0\}$, thus (iv) is satisfied.

Finally, assume that (iv) holds. Since $G$ is finitely generated, so is its image in $\R$, thus $\psi(G) \cong \Z^n$, for some $n \in \N_0$. It follows that every subgroup of $\psi(G)$ is a virtual retract (see \Cref{lem:vabs_have_LR}). 
Choose any element $c \in K \cap A$ with $\psi(c) \neq 0$. Then $c$ has infinite order in $G$, $\psi(\langle c \rangle) \vr \psi(G)$ and $\psi$ is injective on $\langle c \rangle$. Therefore, $\langle c \rangle \vr G$ by \Cref{lem:basic_props_of_virt_retr}.(a). Since $A$ is virtually cyclic, we know that $|A:\langle c \rangle|<\infty$, thus $A \avr G$. This shows that (iv) implies (i), so the proof is complete.
\end{proof}

\subsection{Graphs of groups}\label{subsec:graphs_of_gps}
In this paper we adopt Serre's notation for graphs. Namely, a graph $\Gamma$ is a $5$-tuple $(V\Gamma, E\Gamma, \alpha, \omega,\overline{\phantom{e}})$, where $ V\Gamma $ and  $E\Gamma$ are sets of \emph{vertices} and \emph{edges} of $\Gamma$, respectively, $\alpha, \omega: E\Gamma \to V\Gamma$ are the \emph{incidence maps}, and $ \overline{\phantom{e}}:E\Gamma \to E\Gamma$ is an involution sending each edge to its inverse.

We also use Serre's notation for graphs of groups $(\mathcal{G},\Gamma)$, where $\Gamma$ is a connected graph as above and 
$\mathcal{G}$ is the data consisting of \emph{vertex groups} $\{G_v\}_{v \in V\Gamma}$, \emph{edge groups} $\{G_e\}_{e \in E\Gamma}$ and monomorphisms $\{\alpha_e\}_{e \in E\Gamma}$, such that 
$G_e = G_{\overline{e}}$ and $ \alpha_e\colon G_e \to G_{\alpha(e)}$,  for all $e \in E\Gamma$
(see \cite[Subsection~2.2]{MM-vr_in_free_constr} for more details).
The Structure Theorem of Bass-Serre Theory \cite[Theorem~13 in Section~I.5.4]{Serre} gives the following statement.

\begin{thm}\label{thm:kurosh}
	Let $G$ be the fundamental group of a graph of groups $ (\mathcal{G},\Gamma) $ and let $H$ be a subgroup of $G$. Then $H$ is isomorphic to the fundamental group of a new graph of groups $(\mathcal{H},\Delta)$, with vertex/edge groups equal to the intersections of $H$ with $G$-conjugates of the vertex/edge groups of $(\mathcal{G},\Gamma)$. Moreover, if $\Gamma$ is finite and $|G:H|<\infty$ then $\Delta$ is also finite.
\end{thm}

More precisely, fix a maximal tree $T$  and an orientation on the edges $E\Gamma=E\Gamma^+ \sqcup E\Gamma^-$ in $\Gamma$. As in \cite[Section~I.5.3]{Serre}, using this data we can construct the \emph{Bass-Serre tree} $\cT$ such that the fundamental group $G\coloneq \pi_1(\mathcal{G},\Gamma,T, E\Gamma^+)$ acts on $\cT$ without inverting any edges, with the quotient $G \setminus \cT \cong \Gamma$ and with the stabilizers of vertices and edges being precisely the conjugates of $G_v$ and $\alpha_e(G_e)$ in $G$, for $v \in V\Gamma$ and $e \in E\Gamma$.

This induces an action of $H$ on $\cT$ with quotient graph $\Delta$, which, following  \cite[Section~I.5.4]{Serre}, gives rise to a splitting of $H$ as the fundamental group of a graphs of groups $(\mathcal{H},\Delta)$, whose vertex/edge groups are chosen as $H$-conjugacy class representatives of the intersections of $H$ with the $G$-stabilizers of vertices/edges in $\cT$. Moreover, \cite[Theorem~13 in Section~I.5.4]{Serre} tells us that we can choose a maximal tree in $\Delta$ and an orientation on the edges of $\Delta$ in such a way that the Bass-Serre tree associated to $(\mathcal{H},\Delta)$ is naturally isomorphic to $\cT$.
% such that for every $u \in V\Delta$ there exist $v=v(u) \in V\Gamma$ and $g=g(u) \in G$ with $H_u=H \cap g G_v g^{-1}$ in $G$, and for every $f \in E\Delta$ there exist $e=e(f) \in E\Gamma$ and $h=h(f) \in G$ with $\alpha_f(H_f)=H \cap h \alpha_e(G_e) h^{-1}$ in $G$

%%%%%%%%%%%%%%%%%%%                             %%%%%%%%%%%%%%%%
%%%%%%%%%%%%%%%%%%%           New Section       %%%%%%%%%%%%%%%%
%%%%%%%%%%%%%%%%%%%                             %%%%%%%%%%%%%%%%

\section{Splitting virtually abelian groups in finite index supergroups}\label{sec:splitting}
This section develops an auxiliary tool allowing us to embed a virtually abelian group into a semidirect product of a torsion-free abelian group with a finite group, in a functorial manner. The embedding itself (see \Cref{def:compl}) and the fact that the target group always splits (\Cref{prop:splits}) are folklore; the authors are indebted to P. Kropholler for pointing these out to them.
    
	\begin{defn}\label{def:categ} Consider a category $\mathcal C$, whose objects are pairs $(G,A)$, where $G$ is a group and $A \n G$ is a torsion-free abelian normal subgroup of $G$. A morphism $\varphi: (G_1,A_1) \to (G_2,A_2)$ in this category is a group homomorphism $\varphi:G_1 \to G_2$ such that $A_1=\varphi^{-1} (A_2)$. We say that this morphism is \emph{injective} (\emph{surjective})  if the group homomorphism $\varphi:G_1 \to G_2$ is injective (respectively, surjective).
	\end{defn}
	
	\begin{rem} The condition that $\varphi^{-1}(A_2)=A_1$ from Definition~\ref{def:categ} can be restated by saying that we have a commutative diagram of the form
		\[
		% https://tikzcd.yichuanshen.de/#N4Igdg9gJgpgziAXAbVABwnAlgFyxMJZABgBpiBdUkANwEMAbAVxiRAB13gBGTgXxB9S6TLnyEU3clVqMWbAIIB9boOEgM2PASIAmadXrNWiEAHEVakVvFEAzAdnG2ARUtDrYnSgAsjo-KmnDz8Vhqi2hIkpNwyASYcXLzsAh7hNt7IUrGGcgnKumGaXlH6OU6B5kqFacWR9jFxea7VRRG2vo25zkFJoXwyMFAA5vBEoABmAE4QALZIZCA4EEhSFSZgTAwM1Ax0AEYwDAAK7d4gDDATOGHTc6vUy0j660ib27sHR6cZEhdXN0edCwDDYAAsIBAANa3GbzRBrJ6IABs3UqnEYaDBdBAn0OJzOf0u11h90QLyRDleiHeOwuXwJvzYxMBIDBMDoUDYOAA7hB2ZyEGk7vCKStEAB2NEJThYrC4+n4n4lZkA0nwqlIvzU2l476E1Uk4VwpCa8UADmlby2dL2SoNphZCpwwNBpgh0PVSAArI9xaidTa9YyVY61cayQGkVLAx9FfqmWGSUCQeDITCI-CY0jLbHbQzlfUk6yBVzTLz+RyoEL1CKkLmkQBOK00oPxkNF-5Gih8IA
\begin{tikzcd}
\{1\} \arrow[r] & A_1 \arrow[r, hook] \arrow[d] & G_1 \arrow[r, two heads] \arrow[d, "\varphi"] & Q_1 \arrow[r] \arrow[d, hook] & \{1\} \\
\{1\} \arrow[r] & A_2 \arrow[r, hook]                     & G_2 \arrow[r, two heads]                   & Q_2 \arrow[r]                 & \{1\}
\end{tikzcd}\quad ,
		\] 
         where the rows in this diagram are exact sequences and the right-most vertical map is an injective homomorphism from $Q_1 \coloneq G_1/A_1$ to $Q_2 \coloneq G_2/A_2$. 
         
		If we fix a group $Q$ and consider only objects $(G,A) \in \mathcal C$ such that $G/A \cong Q$ then we get a full subcategory of $\mathcal C$ which appears as $\left( \frac{Q}{}\right)$ in Gruenberg's notes \cite[Section~9.1]{Gru}
		
	\end{rem}

	If $(G,A) \in \mathcal{C}$ then $G$ acts on $A$ by conjugation, which allows us to consider the semidirect product $A \rtimes G=\{(a,g) \mid a \in A,~g \in G\}$, where multiplication is defined by $(a,g)(b,h)=(a (g bg^{-1}),gh)$, for all $(a,g), (b,h) \in A \rtimes G$.

	\begin{defn}\label{def:compl} Given any pair $(G,A) \in \mathcal{C}$ and any $n \in \N$ we define the group $P_n=P_n(G,A)$ as the quotient of $A \rtimes G$ by the normal subgroup \[N_n=\{(a^n,a^{-1}) \mid a \in A\}.\]
		We let $\psi_n:A \rtimes G \to P_n$ denote the natural quotient map and define $B_n=B_n(G,A)$ by $B_n\coloneq \psi_n((A,1)) \n P_n$, where $(A,1)=\{(a,1) \mid a \in A\} \n A \rtimes G$. We will also denote by $\xi_n:G \to P_n$ the induced map $\xi_n(g)\coloneq \psi_n((1,g))$, $g \in G$.
	\end{defn}
	
	One readily verifies that $N_n$ is indeed a normal subgroup of $A \rtimes G$ because $A \n G$ is abelian. Moreover, since $A$ is torsion-free, the homomorphism $\psi_n$ is injective on the copies $(A,1)$ and $(1,G)=\{(1,g) \mid g \in G\}$ of $A$ and $G$ in $A \rtimes G$. Intuitively, $P_n$ is obtained from $G$ by adding $n$-th roots to all elements of $A$, because $(a,1)^n=(a^n,1)$ is identified with $(1,a)$ in $P_n$, so that $\xi_n(A)=(B_n)^n$. In particular, when $A \cong \Z^m$, for some $m \in \N \cup \{0\}$, then $\xi_n(G)\cong G$ has index $n^m$ in $P_n$.
	
	\begin{lemma}\label{lem:xi_n_is_morphism} For each $n \in \N$ and every $(G,A) \in \mathcal{C}$, the pair $(P_n,B_n)$ belongs to $\mathcal C$ and $\xi_n:G \to  P_n$ defines an injective morphism between $(G,A)$ and $(P_n,B_n)$.    Moreover, $B_n \cong A$ and $\xi_n$ induces a natural isomorphism between the quotients $G/A$ and $P_n/B_n$.
	\end{lemma}
	
	\begin{proof} We have already noted that $\xi_n$ is injective on $G$. Since $(A,1) \cap N_n=\{(1,1)\}$ in $A \rtimes G$, we have 
		\[B_n \cong (A,1)/((A,1) \cap N_n)=(A,1)\cong A, \]
		so $B_n$ is a torsion-free abelian normal subgroup of $P_n$, thus $(P_n,B_n) \in \mathcal{C}$. 
		
		Observe that $\xi_n(g) \in B_n$ in $P_n$ if and only if $(1,g) \in (A,1)N_n=(A,A)$ in $A \rtimes G$, which is equivalent to $g \in A$. Thus $\xi_n^{-1}(B_n)=A$ in $G$, so $\xi_n$ gives rise to a morphism between $(G,A)$ and $(P_n,B_n)$ in $\mathcal{C}$.
		
		The final claim of the lemma follows from the second isomorphism theorem:
		\begin{align*}
			P_n/B_n &\cong (A \rtimes G)/\bigl((A,1)N_n\bigr) = \bigl((A,A)(1,G)\bigr)/(A,A) \\
			&\cong (1,G)/(1,A)  \cong G/A.    \qedhere
		\end{align*} 
	\end{proof}
	
	For every $n \in \N$, \Cref{lem:xi_n_is_morphism} allows us to define a map $\mathcal{F}_n: \mathcal{C} \to \mathcal{C}$, by $\mathcal{F}_n((G,A)) \coloneq (P_n,B_n)$, given in \Cref{def:compl}.

	\begin{rem} In view of Lemma~\ref{lem:functorial} below, it is not difficult to see that for every $n \in \N$ the map $\mathcal{F}_n: \mathcal{C} \to \mathcal{C}$ defines a covariant functor from the category $\mathcal C$ to itself.  We will not need this fact here, so we leave its verification to the reader.  
	\end{rem}
    
	If $Q=G/A$, the previous lemma implies that $P_n$ fits into the commutative diagram \eqref{eq:comm_diag_for_P_n}, where $\theta_n:A \to B_n$ is given by \[\theta_n(a) \coloneq \psi_n((1,a))=\psi_n((a^n,1)), ~\text{ for all }a \in A.\]
	\begin{equation}\label{eq:comm_diag_for_P_n}
		\begin{tikzcd}
			\{1\} \arrow[r] & A \arrow[r, hook] \arrow[d,  hook,"\theta_n"] & G \arrow[r,two heads] \arrow[d,  hook,"\xi_n"] & Q \arrow[r] \arrow[d, equal] & \{1\} \\
			\{1\} \arrow[r] & B_n \arrow[r, hook]                       & P_n \arrow[r,two heads]                   & Q \arrow[r]                                & \{1\}
		\end{tikzcd}
	\end{equation}
	Since $A$ is abelian, the conjugation action of $G$ on $A$ gives rise to an action of $Q$ on $A$. The action of $Q$ on $B_n$ is inherited from it, i.e., \[q . \psi_n((a,1)) \coloneq \psi_n((q. a,1)),~ \text{ for all } q \in Q \text{ and } a \in A.\] In particular, $\theta_n$ is a monomorphism of $Q$-modules.
	
	The main reason why we need the construction from \Cref{def:compl} is the following fact.
	
	\begin{prop}\label{prop:splits}
		Suppose that $(G,A) \in \mathcal{C}$, $Q=G/A$ is finite and $n \in \N$ is divisible by the order $|Q|$.  If $(P_n,B_n)=\mathcal{F}_n((G,A)) \in \mathcal{C}$, 
        then $P_n$ splits as the semidirect product $B_n \rtimes Q$, where the action of $Q$ on $B_n$ is inherited from the action of $Q$ on $A$ in $G$.  
	\end{prop}
	
	\begin{proof} Let $f:Q \times Q \to A$ be a $2$-cocycle determining the extension \[A \to G \to Q. \]
		In view of Lemma~\ref{lem:xi_n_is_morphism} and diagram \eqref{eq:comm_diag_for_P_n}, the composition $\overline f=\theta_n \circ f:Q \times Q \to B_n$ is a $2$-cocycle, determining the extension \[B_n \to P_n \to Q.\]
		
		By construction,  $\frac 1n \theta_n(A) \subseteq B_n$ (we are using additive notation on $A$ and $B_n$ as we treat them as $Q$-modules here), so $h=\frac1n \bar f:Q\times Q \to B_n$ is also a $2$-cocycle. But multiplication by $n$ annihilates the cohomology group $H^2(Q,B_n)$ because $|Q|$ divides $n$ (see \cite[Corollary III.10.2]{Brown}), so $\bar f=n h$ is a coboundary, hence the extension $B_n \to P_n \to Q$ splits and $P_n \cong B_n \rtimes Q$.
	\end{proof}

	\begin{lemma}\label{lem:functorial}
		Given $(G_1,A_1), (G_2,A_2) \in \mathcal{C}$ and $n \in \N$, any morphism $\varphi:(G_1,A_1) \to (G_2,A_2)$  defines a morphism $\tilde\varphi:(P_{n,1},B_{n,1})\to (P_{n,2},B_{n,2})$  such that the corresponding group homomorphisms fit into the following commutative diagram (here $P_{n,i}=P_n(G_i,A_i)$, $B_{n,i}=B_n(G_i,A_i)$ and $\xi_{n,i}:G_i \to P_{n,i}$, $i=1,2$, are given by Definition~\ref{def:compl}):
		\begin{equation}\label{eq:phi_tilde}
			% https://tikzcd.yichuanshen.de/#N4Igdg9gJgpgziAXAbVABwnAlgFyxMJZABgBpiBdUkANwEMAbAVxiRAHEB9ARhAF9S6TLnyEU3clVqMWbLgCZ+gkBmx4CRMtyn1mrRCABKnYGFLc+SoWtFEJ26rtkHjp0vMt8pMKAHN4RKAAZgBOEAC2SGQgOBBIEtJ6bAA6yWgAFlhWIKERSPLUsUgAzNQMWGD6IFB0cOk+II4yVal4DLAABKkZWQLBYZGI0UWIBYnOIKkAHlgmZhbZuYMJI6XjLckzc+6eFHxAA
			\begin{tikzcd}
				G_1 \arrow[r, "\varphi"] \arrow[d, "{\xi_{n,1}}"] & G_2 \arrow[d, "{\xi_{n,2}}"] \\
				{P_{n,1}} \arrow[r, "\tilde \varphi"]     & {P_{n,2}}                \end{tikzcd}  \quad .
		\end{equation}
		Moreover, if $\varphi$ is injective (respectively, surjective), then so is  $\tilde \varphi$.
	\end{lemma}
	
	\begin{proof} Since $\varphi(A_1) \subseteq A_2$, we can define the map 
    \begin{equation}\label{eq:def_of_hat_varphi}
      \hat\varphi:A_1 \rtimes G_1 \to A_2\rtimes G_2,~~(a,g)\stackrel{\hat\varphi}{\mapsto} (\varphi(a),\varphi(g)),  
    \end{equation} that fits into the commutative diagram:
		\begin{equation}\label{eq:phi_hat}
			\begin{tikzcd}
				G_1 \arrow[r, "\varphi"] \arrow[d] & G_2 \arrow[d] \\
				{A_1\rtimes G_1} \arrow[r, "\hat \varphi"]     & {A_2\rtimes G_2}                \end{tikzcd}  \quad ,
		\end{equation}   
		where the vertical maps are the natural embeddings $g \mapsto (1,g)$, of $G_i$ into $A_i \rtimes G_i$, $i=1,2$.
		
		Let $N_{n,i} \n A_i\rtimes G_i$ denote the kernel of the epimorphism $\psi_{n,i}: A_i\rtimes G_i \to P_{n,i}$, $i=1,2$, given by \Cref{def:compl}. Note that $\hat\varphi(N_{n,1}) \subseteq N_{n,2}$ by \eqref{eq:def_of_hat_varphi}, hence $ \hat{\varphi}$ factors through the quotients, giving the commutative diagram		\begin{equation}\label{eq:second_half}
			\begin{tikzcd}
				A_1\rtimes G_1 \arrow[r, "\hat\varphi"] \arrow[d, "\psi_{n,1}"] & A_2\rtimes G_2 \arrow[d,"\psi_{n,2}"] \\
				{ P_{n,1}} \arrow[r, "\tilde \varphi"]     & { P_{n,2}}                \end{tikzcd}  \quad ,
		\end{equation}   
		where $\tilde\varphi:P_{n,1} \to P_{n,2}$ is the induced homomorphism.
        Thus, \eqref{eq:phi_tilde} follows by combining \eqref{eq:phi_hat} with \eqref{eq:second_half}.

        To check that $\tilde\varphi$ defines a morphism between the objects $(P_{n,1},B_{n,1})$ and $(P_{n,2},B_{n,2})$ in $\mathcal{C}$, we need to verify that $B_{n,1}=\tilde\varphi^{-1}(B_{n,2})$. The inclusion $B_{n,1} \subseteq \tilde\varphi^{-1}(B_{n,2})$ is clear from \eqref{eq:second_half} because $B_{n,i}=\psi_{n,i}((A_i,1))$, $i=1,2$, and $\hat\varphi((A_1,1)) \subseteq (A_2,1)$, as $\varphi$ is a morphism between $(G_1,A_1)$ and $(G_2,A_2)$. For the opposite inclusion, suppose that $p \in P_{n,1}$ satisfies $\tilde\varphi(p) \in B_{n,2}$. Since $\psi_{n,1}$ is surjective, there is $(a,g) \in A_1 \rtimes G_1$ such that $p=\psi_{n,1}((a,g))$, and \eqref{eq:second_half} implies that 
        \[(\varphi(a),\varphi(g))=\hat\varphi((a,g)) \in \psi_{n,2}^{-1}(B_{n,2})=(A_2,1)N_{n,2}=(A_2,A_2).\]
Thus $\varphi(g) \in A_2$, so $g \in \varphi^{-1}(A_2)=A_1$. It follows that 
\[p=\psi_{n,1}((a,g)) \in \psi_{n,1}((A_1,A_1))=\psi_{n,1}((A_1,1)N_{n,1})=B_{n,1},\]
i.e., $\tilde\varphi^{-1}(B_{n,2}) \subseteq B_{n,1}$, as required.
        
		If $\varphi$ is surjective then $\varphi(A_1)=A_2$, so, in view of \eqref{eq:def_of_hat_varphi},  $\hat \varphi$ is also surjective. Therefore, $\tilde \varphi$ is surjective by \eqref{eq:second_half}. Assume, now, that $\varphi$ is injective. In view of \eqref{eq:second_half}, to show injectivity of $\tilde \varphi$ it is enough to check that $\hat\varphi^{-1}(N_{n,2}) \subseteq N_{n,1}$. Indeed, suppose that $\hat\varphi((a,g)) \in N_{n,2}$, for some $(a,g) \in A_1\rtimes G_1$. Then \[\hat\varphi((a,g))=(\varphi(a),\varphi(g))=(c^n,c^{-1}), \text{ for some } c \in A_2,\]
		which implies that $\varphi(g)=c^{-1} \in A_2$, so $g \in \varphi^{-1}(A_2)=A_1$ (because $\varphi$ defines a morphism between $(G_1,A_1)$ and $(G_2,A_2)$ in $\mathcal C$). Consequently, $\varphi(a)=c^n=\varphi(g^{-n})$, hence $a=g^{-n}$ in $G_1$, as $\varphi$ is injective. The latter shows that $(a,g)=(g^{-n},g) \in N_{n,1}$, so $\hat\varphi^{-1}(N_{n,2}) \subseteq N_{n,1}$ and $\tilde \varphi$ is injective.
	\end{proof}
    % {\color{blue} Another compact way of proving $B_{n,1}=\tilde\varphi^{-1}(B_{n,2})$: 
    % By surjectivity of $\psi_{n,1}$ and commutativity of \eqref{eq:second_half}, we have that $B_{n,1}=\tilde\varphi^{-1}(B_{n,2})$ if and only if $ (A_1,1)N_{n,1} = \hat{\varphi}^{-1}((A_2,1)N_{n,2})$, which is true since $ (A_1,A_1) = (\hat{\varphi}^{-1}(A_2,A_2))$ by the definitions of $\hat{\varphi}$ and of morphisms in the category $\mathcal{C}$.
    %  } 

%%%%%%%%%%%%%%%%%%%                             %%%%%%%%%%%%%%%%
%%%%%%%%%%%%%%%%%%%           New Section       %%%%%%%%%%%%%%%%
%%%%%%%%%%%%%%%%%%%                             %%%%%%%%%%%%%%%%

\section{Permutational wreath products}\label{sec:wreath}
	Let $S$ be a group acting on a set $\Omega$ on the left. For any group $A$ we can define the \emph{permutational wreath product} $A \wr_\Omega S$ as the semidirect product $A^\Omega \rtimes S$, where $A^\Omega$ consists of all functions $f:\Omega \to A$ with finite support and $S$ acts on $A^\Omega$ according to the formula \[(g . f)(x) \coloneq f(g^{-1} . x),~\text{ for all } g \in S \text{ and } x \in \Omega.\] The subgroup $A^\Omega$ is said to be the \emph{base} of the wreath product $A \wr_\Omega S$.
	
	\begin{rem}\label{rem:embeddings_between_wreath_products} Suppose that $A$, $S$ are groups and $S$ acts on a set $\Omega$. If $T \leqslant S$ and $\Sigma \subseteq \Omega$ is a $T$-invariant subset then $A \wr_\Sigma T$ naturally embeds as a subgroup of $A \wr_\Omega S$.   
	\end{rem}
	
	In this embedding, any function $f:\Sigma \to A$ is sent to a function $\hat f:\Omega \to A$ such that $\hat f(x)=f(x)$, for every $x \in \Sigma$, and $\hat f(y)=1$, for every $y \in \Omega \setminus\Sigma$.

	If $\Omega=S$ and the action of $S$ on itself is by left translations, then the permutational wreath product $A \wr_S S$ is called the \emph{standard wreath product}, denoted by $A \wr S$.
	
	For any set $\Omega$, we will denote by $\Sym(\Omega)$ the group of all permutations of this set.  
	We will always assume that the action of $\Sym(\Omega)$ on $\Omega$ is the standard left action, unless specified otherwise. 
	
	If $k \in \N$ then we will write $\Omega_k\coloneq \{1,\dots,k\}$ and $S_k \coloneq \Sym(\Omega_k)$, the symmetric group on $k$ elements.
	Given any group $A$,  the group $A \wr_{\Omega_k} S_k$ is called the \emph{complete monomial group of $A$ of degree $k$} in some literature. It was introduced and studied by Ore \cite{Ore} and can be represented as a group of $k \times k$ matrices, where each row and column contain exactly one non-zero entry from $A$. Such \emph{monomial matrices} are multiplied in a natural way, establishing a group embedding $A \wr_{\Omega_k} S_k \hookrightarrow \mathrm{GL}(k,\Z A)$. In this section, we will produce embeddings of finitely generated virtually abelian groups into complete monomial groups satisfying additional properties (see \Cref{prop:embedding_for_virt_ab_gps}).

	If $L$ is a normal subgroup of a group $G$, any subgroup $Q \leqslant G$ satisfying $G=LQ$ and $Q \cap L=\{1\}$ will be called a \emph{complement to $L$} in $G$.  If $S$ is a group acting on a set $\Omega$ and $x \in \Omega$, we will write $S_x$ to denote the \emph{stabilizer} of $x$ in $S$.
	The next proposition is a simplification of a result of Parker and Quick \cite{PQ}, which generalized a classical result of Ore for complete monomial groups \cite[Theorem~I.11]{Ore}. 
	
	\begin{prop}[{\cite[Theorem~2.6]{PQ}}]\label{prop:P-Q}
		Let $A$ be any group, let $S$ be a group acting transitively on a finite set $\Omega$, and let $G=A \wr_\Omega S$. Suppose that for some $x \in \Omega$ the stabilizer $S_x$ does not admit non-trivial homomorphisms to $A$. Then there is exactly one conjugacy class of complements to $A^\Omega$ in $G$.
		%
		% for any $x \in \Omega$ the conjugacy classes of complements to $L=A^\Omega$ in $G$ are in one-to-one correspondence with the equivalence classes of homomorphisms $\varphi: S_x \to A$, where two homomorphisms $\varphi,\xi:S_x \to A$ are equivalent if and only if there is $\eta \in \mathrm{Inn}(A)$ such that $\xi=\eta \times\varphi$.
	\end{prop}
	
	We will extend this result to non-transitive actions, which will help us describe conjugacy classes of subgroups  that are disjoint from base groups in wreath products.
	
	\begin{prop}\label{prop:sbgps_disjoint_from_base_conj_into_S} Suppose that $A,S$ are groups, $S$ acts on a non-empty finite set $\Omega$, $G=A \wr_\Omega S$ and $Q \leqslant G$ is a subgroup which has trivial intersection with the base $L \coloneq A^\Omega$ of the wreath product $G$.  Let $R \leqslant S$ denote the image of $Q$ under the natural projection $G \to S$. If $R_x$ does not admit non-trivial homomorphisms to $A$, for every $x \in \Omega$, then there is $h \in L$ such that $Q= h Rh^{-1}$ in $G$.    
	\end{prop}
	
	Before proving this proposition we need an auxiliary lemma.
	\begin{lemma}\label{lem:disjoint_norm_sbgps} Suppose that $H$ is a group, $N_1,N_2 \n H$ are normal subgroups with trivial intersection, and $Q,R \leqslant H$ are two complements to $L\coloneq N_1N_2$ in $H$. Let $\alpha_i:H \to H/N_i$ be the natural homomorphisms, $i=1,2$. If  $\alpha_i(Q)$ is conjugate to $\alpha_i(R)$ in $H/N_i$, for each $i=1,2$, then there is $h \in L$ such that $Q=hRh^{-1}$ in $H$.    
	\end{lemma}
	
	\begin{proof} By the assumptions, for each $i=1,2$ there exists $h_i \in H$ such that $\alpha_i(Q)=\alpha_i(h_i R h_i^{-1})$. Since $H=N_1N_2R$ and $N_i=\ker\alpha_i$ for $i=1,2$, we can assume that $h_1 \in N_2$ and $h_2 \in N_1$. 
		
		Thus, given any $q \in Q$  there is $r_i=r_i(q) \in R$ such that 
		\begin{equation}\label{eq:alpha_i_of_r_i}
			\alpha_i(q)=\alpha_i(h_ir_ih_i^{-1}) \text{ in } H/N_i, \text{ for } i=1,2.
		\end{equation}
		
		Let $\rho: H \to R$ denote the natural retraction of $H$ onto $R$ with kernel $L=N_1N_2$. Then $\rho$ factors through $\alpha_i$, for each $i=1,2$, so \eqref{eq:alpha_i_of_r_i} implies that 
		\[\rho(q)=\rho(h_ir_ih_i^{-1})=\rho(r_i)=r_i, \text{ for } i=1,2,\]
		i.e., $r_1=r_2=\rho(q)$ in $R$. Recall that $\alpha_1(h_2)=1$, so \eqref{eq:alpha_i_of_r_i} yields
		\[\alpha_1(h_1h_2 \rho(q) h_2^{-1} h_1^{-1})=\alpha_1(h_1 r_1 h_1^{-1})=\alpha_1(q), \text{ for every } q \in Q.\]
        
		Similarly, $\alpha_2(h_1h_2 \rho(q) h_2^{-1} h_1^{-1})=\alpha_2(q)$, for all $q \in Q$. Since $\ker\alpha_1\cap\ker\alpha_2=N_1 \cap N_2=\{1\}$ by the assumptions,  we can conclude that 
		\begin{equation}\label{eq:conj_by_h_1h_2}
			q=h_1h_2 \rho(q) h_2^{-1}h_1^{-1} \text{ in }H, \text{ for every }q \in Q.    
		\end{equation}
        
		Since $h \coloneq h_1h_2 \in L=\ker\rho$ and the restriction of $\rho$ to $Q$ is an isomorphism between $Q$ and $R$, \eqref{eq:conj_by_h_1h_2} shows that $Q=hRh^{-1}$ in $H$, as claimed.
	\end{proof}

	\begin{proof}[Proof of \Cref{prop:sbgps_disjoint_from_base_conj_into_S}] Set $H=LR \leqslant G$ and observe that $Q \leqslant H$ and $H \cong A \wr_\Omega R$ is itself a wreath product, where the action of $R$ on $\Omega$ is induced by the action of $S$ on $\Omega$ (see \Cref{rem:embeddings_between_wreath_products}). Moreover, $Q$ and $R$ are complements to $L=A^\Omega$ in $H$.
		
		Let $x_1,\dots,x_k \in \Omega$ be a list of orbit representatives for the action of $R$ on $\Omega$. We will argue by induction on $k$. If $k=1$ then the desired statement follows from \Cref{prop:P-Q}, so assume that $k \ge 2$. 
		
		Let $\Sigma_1=R . x_1$ be the $R$-orbit of $x_1$ and $\Sigma_2=\Omega\setminus\Sigma_1$ be the union of the remaining orbits.
		Observe that the subgroups $N_1=A^{\Sigma_1}$ and $N_2=A^{\Sigma_2}$ are normal in $H$, $N_1 \cap N_2=\{1\}$, $L=N_1N_2$, and we have natural isomorphisms $H/N_1 \cong A \wr_{\Sigma_2} R$, $H/N_2 \cong A \wr_{\Sigma_1} R$. Since the image of $Q$ in $A \wr_{\Sigma_2} R$ is a complement to $A^{\Sigma_2}$ and $R$ acts with $k-1<k$ orbits on $\Sigma_2$, we can use the induction hypothesis to deduce that the images of $Q$ and $R$ are conjugate in $H/N_1$. Similarly, the images of $Q$ and $R$ are conjugate in $H/N_2$. Therefore, we can apply \Cref{lem:disjoint_norm_sbgps} to find $h \in L$ such that $Q=hRh^{-1}$ in $H$. This completes the proof of the proposition by induction.
	\end{proof}
	
	The following immediate corollary of \Cref{prop:sbgps_disjoint_from_base_conj_into_S} will be useful.
	
	\begin{cor}\label{cor:conj_class_of_fin_sbgps_in_wr_prod} Suppose that $A$ is a torsion-free group and $S$ is a  group acting on a finite set $\Omega$. If $G=A \wr_\Omega S$ and $Q \leqslant G$ is any finite subgroup then there exists $h \in A^\Omega $ such that $Q \subseteq hSh^{-1}$ in $G$.   
	\end{cor}
	
	If $S$ is a group acting on a set $\Omega$, then an element $\tau \in S$ is said to act \emph{freely} if it has no fixed points. A subgroup $T \leqslant S$ acts \emph{freely} on $\Omega$ if every non-trivial element $\tau \in T$ acts freely.
	
	\begin{prop}\label{prop:embedding_for_virt_ab_gps}
		Let $P$ be a group with a finite index normal subgroup $B \cong \Z^j$, for some $j \in \N$, and let $Q \leqslant P$ be a finite subgroup. Then there exist $k \in \N$ and a monomorphism $\eta: P \to \Z \wr_{\Omega_k} S_k$, where $\Omega_k=\{1,\dots,k\}$, such that all of the following hold:
		\begin{itemize}
			\item[(i)] $\eta(B) \subseteq \Z^{\Omega_k}$;
			\item[(ii)] $\eta(Q) \subseteq S_k$ and $\eta(Q)$ acts freely on ${\Omega_k}$.
		\end{itemize}
	\end{prop}
	
	\begin{proof} Set $T \coloneq P/B$, then, by the Universal Embedding Theorem of Krasner and Kaloujnine (see \cite[Theorem~2.6.A]{DM}), there is a monomorphism $\psi:P \to B \wr T$, where $B \wr T$ is the standard wreath product and $\psi(B) \subseteq B^T$. 
		
		By the assumptions, $B \cong \Z^j$ and we denote $k \coloneq j|T|\in \N \cup \{0\}$. It is easy to see that there is an isomorphism $\xi:B \wr T \to \Z \wr_{\Omega_k} T $, where ${\Omega_k}=\{1,\dots,k\}$ and $T$ acts on ${\Omega_k}$ freely with $j$ orbits, such that $\xi(B^T)=\Z^{\Omega_k}$ and the restriction of $\xi$ to $T$ is the identity map. It follows that $\xi\circ \psi: P \to \Z \wr_{\Omega_k} T$ is an embedding satisfying $(\xi\circ\psi)(B) \subseteq \Z^{\Omega_k}$.
		
		By \Cref{cor:conj_class_of_fin_sbgps_in_wr_prod}, there is $h \in \Z^{\Omega_k}$ such that $(\xi\circ \psi)(Q)\subseteq hTh^{-1}$. Therefore, after applying an inner automorphism of $\Z \wr_{\Omega_k} T$ (which preserves the base $\Z^{\Omega_k}$), we can assume that $(\xi\circ \psi)(Q)\subseteq T$ in $\Z \wr_{\Omega_k} T$.

		Finally, since $T$ acts on ${\Omega_k}$ freely, hence faithfully, we have an embedding $T \hookrightarrow S_k$ which extends to a monomorphism $\alpha: \Z \wr_{\Omega_k} T \to \Z \wr_{\Omega_k} S_k$ by \Cref{rem:embeddings_between_wreath_products}. We now see that the composition $\eta=\alpha \circ \xi \circ \psi: P \to \Z \wr_{\Omega_k} S_k$ satisfies both conditions from the statement of the proposition.    
	\end{proof}

%%%%%%%%%%%%%%%%%%%                             %%%%%%%%%%%%%%%%
%%%%%%%%%%%%%%%%%%%           New Section       %%%%%%%%%%%%%%%%
%%%%%%%%%%%%%%%%%%%                             %%%%%%%%%%%%%%%%

\section{Proof of \texorpdfstring{\Cref{thm:amalg_of_virt_ab}}{Theorem 1.1}}\label{sec:proof_of_thm}
	\begin{lemma}\label{lem:amalgamation}
		Given any  $k,l \in \N$, consider the groups $E_1 \coloneq \Z \wr_{\Omega_k} S_k$  and $E_2 \coloneq \Z \wr_{\Omega_l} S_l$. Suppose $c_1 \in \Z^{\Omega_k}$ and $c_2 \in \Z^{\Omega_l}$ are non-trivial elements in $E_1$ and $E_2$ respectively. Then for $m\coloneq kl \in \N$  there exist monomorphisms $\beta_i: E_i \to \Z \wr_{\Omega_m} S_m$, $i=1,2$, satisfying the following conditions: 
		\begin{itemize}
			\item $\beta_1(\Z^{\Omega_k}), \beta_2(\Z^{\Omega_l}) \subseteq \Z^{\Omega_m}$;
			\item $\beta_1(S_k) \subseteq S_m$ and if an element  $\sigma \in S_k$ acts freely on $\Omega_k$ then $\beta_1(\sigma)$ acts freely on $\Omega_m$;
			\item $\beta_2(S_l) \subseteq S_m$ and if an element  $\tau \in S_l$ acts freely on $\Omega_l$ then $\beta_2(\tau)$ acts freely on $\Omega_m$;
			\item $\beta_1(c_1)=\beta_2(c_2)$.
		\end{itemize}
	\end{lemma}
	
	\begin{proof}
		Recall that in our notation, $\Omega_k=\{1,\dots,k\}$ and $\Omega_l=\{1,\dots,l\}$. We set $\Omega=\Omega_k \times \Omega_l$ and $S=\Sym(\Omega)$. Evidently, for $m=kl$ there is an isomorphism between wreath products $E\coloneq \Z \wr_\Omega S$ and $\Z \wr_{\Omega_m} S_m$, taking $\Z^\Omega$ to $\Z^{\Omega_m}$ and $S$ to $S_m$. Therefore, we can work with $E$ instead of $\Z \wr_{\Omega_m} S_m$.
		
		In this proof we will think of the base group $\Z^{\Omega_k}$ in $E_1$ ($\Z^{\Omega_l}$ in $E_2$) as the set of all vectors in $\Z^k$ (respectively, $\Z^l$). It will therefore be natural to think of $\Z^\Omega \leqslant E$ as the set of all $k\times l$ matrices with integer entries, under addition.
		
		Let $e_1=(1,0,\dots,0)$, $\dots$, $e_k=(0,0,\dots,1)$ be the standard basis of $\Z^k$. Similarly, let $f_1,\dots, f_l$ be the standard basis of $\Z^l$. By the definition of the action of $S_k$ on $\Z^k \cong \Z^{\Omega_k}$, we have 
		\begin{equation}\label{eq:action_of_S_k}
			\sigma. e_i=e_{\sigma(i)}, \text{ for all } \sigma\in S_k \text{ and } i=1,\dots,k.    
		\end{equation}
		
		By the assumptions, $c_1=(u_1,\dots,u_k)$ is a non-zero vector in $\Z^k \leqslant E_1$, and $c_2=(v_1,\dots,v_l)$ is a non-zero vector in $\Z^l \leqslant E_2$.
		Define a homomorphism $\beta_1: E_1 \to E$ as follows. For each $i \in \{1,\dots,k\}$, we let $\beta_1(e_i)$ be the $k\times l$ matrix $L_i \in \Z^\Omega$ such that the $i$-th row of $L_i$ is the vector $(v_1,\dots,v_l)$ and all the other entries are $0$. For every permutation $\sigma \in S_k$, we let $\beta_1(\sigma) \in S$ be the corresponding permutation of rows of matrices in $\Z^\Omega$, i.e., 
		\[\bigl(\beta_1(\sigma)\bigr)(i,j)=(\sigma(i),j), \text{ for all } (i,j) \in \Omega.\]
		Observe that if $\sigma\in S_k$ acts freely on $\Omega_k$ then $\beta_1(\sigma)$ acts freely on $\Omega$. We also note that 
		\begin{equation}\label{eq:action_of_beta(S_k)}
			\beta_1(\sigma). L_i=L_{\sigma(i)}, \text{ for all } \sigma\in S_k \text{ and } i=1,\dots,k.    
		\end{equation}
		
		Since $(v_1,\dots,v_l) \neq (0,\dots,0)$, the matrices $L_1,\dots,L_k$ are linearly independent in $\Z^\Omega$. Combined with \eqref{eq:action_of_S_k} and \eqref{eq:action_of_beta(S_k)}, this easily implies that $\beta_1$ extends to an injective homomorphism from $E_1$ to $E$. Clearly, $\beta_1(\Z^{\Omega_k})\subseteq \Z^\Omega$ and $\beta_1(S_k) \subseteq S$.
		
		The homomorphism $\beta_2:E_2\to E$
		is defined similarly, but now using columns. For every $j\in \{1,\dots,l\}$ we let $\beta_2(f_j)\in \Z^\Omega$ be the $k\times l$ matrix $M_j$ whose $j$-th column vector is $(u_1,\dots,u_k)^T$ and all the other entries are $0$. And for each permutation $\tau \in S_l$, we let $\beta_2(\tau)\in S$ be the corresponding permutation of columns of $k\times l$ matrices. Again, we observe that elements of $S_l$ acting freely on $\Omega_l$ are sent to elements of $S$ acting freely on $\Omega$. As before, one can check that this gives an injective homomorphism $\beta_2$ from $E_2$ to $E$, satisfying $\beta_2(\Z^{\Omega_l})\subseteq \Z^\Omega$  and $\beta_2(S_l) \subseteq S$.
		
		Finally, since $c_1=\sum_{i=1}^k u_ie_i$ in $\Z^k$ and $c_2=\sum_{j=1}^l v_jf_j$
		in $\Z^l$, we see that
		\[\beta_1(c_1)=\sum_{i=1}^k u_iL_i=
		\begin{pmatrix}
			u_1 v_1 &\ldots &u_1 v_l   \\
			\vdots &  & \vdots \\
			u_k v_1 &\ldots & u_kv_l
		\end{pmatrix} = \sum_{j=1}^l v_jM_j=
		\beta_2(c_2).
		\] 
		This completes the proof of the lemma.
	\end{proof}

	\begin{lemma}\label{lem:cyc-by-fin_in_Z_wr_S} Let $H=C \rtimes Q$, where $C=\langle c \rangle$ is an infinite cyclic group and $Q$ is a finite group. Suppose that $E=A \wr_{\Omega} S_m $, where $A$ is any group, $m \in \N$, $\Omega\coloneq \{1,\dots,m\}$ and $\gamma_i: H \to E$, $i=1,2$, are two monomorphisms such that $\gamma_1(c)=\gamma_2(c) \in A^\Omega\setminus\{1\} $ and $\gamma_1(Q),\gamma_2(Q) \subseteq S_m$ in $E$. If $\gamma_i(Q)$ acts freely on $\Omega$, for $i=1,2$, then there exists $\sigma \in S_m$ such that $\sigma \in \C_E(\gamma_1(c))$ and  $\sigma \gamma_1(h) \sigma^{-1}=\gamma_2(h)$ in $E$, for all $h \in H$.    
	\end{lemma}
	
	\begin{proof} 
		Denote $f\coloneq \gamma_1(c)=\gamma_2(c) \in A^\Omega$ and let $A_f\coloneq f(\Omega)$ be the set of values in $A$ attained by $f$. For each $a \in A_f$ we let $\Omega_a\coloneq \{x \in \Omega \mid f(x)=a\}$, so that $\Omega=\bigsqcup_{a \in A_f} \Omega_a$.
		Throughout this proof we will think of $\Sym(\Omega_a)$ as the subgroup of $S_m=\Sym(\Omega)$ consisting of permutations supported on $\Omega_a$.
		
		Denote by $\St(f)$ the stabilizer of $f$ in $S_m \leqslant E$. By the definition of the action of $S_m$ on $A^\Omega$, a permutation $\sigma \in S_m$ belongs to $\St(f)$ if and only if $\sigma^{-1}.\Omega_a=\Omega_a$ (equivalently, $\sigma.\Omega_a=\Omega_a$), for every $a \in A_f$.

		By the assumptions, for each $i=1,2$ the group $Q_i\coloneq \gamma_i(Q) \leqslant S_m$ is isomorphic to $Q$ and acts freely on $\Omega$. Since $Q$ normalizes the infinite cyclic subgroup $C$ in $H$, for every $q \in Q$ either $qcq^{-1}=c$ or $qcq^{-1}=c^{-1}$.
		
		\medskip
		\noindent \emph{Case 1:} $c$ is central in $H$. Then $Q_1,Q_2 \leqslant \St(f)$ in $S_m$, so each of these subgroups preserves $\Omega_a$, for every $a \in A_f$.
		Therefore, for all $a \in A$ and $i=1,2$, $Q_i =\gamma_i(Q)$   acts freely on $\Omega_a$. Since any two free actions of the  group $Q$ on the same set $\Omega_a$ are conjugate in $\Sym(\Omega_a)$ (cf. \cite[Lemma~9 in Section~II.2.6]{Serre}), there exists a permutation $\sigma_a \in \Sym(\Omega_a)$ such that 
		\begin{equation}\label{eq:sigma_a}
			\left(\sigma_a^{-1} \gamma_2(q) \sigma_a\right).x=\gamma_1(q).x, \text{ for all } q \in Q \text{ and } x \in \Omega_a. 
		\end{equation}
		Note that $\sigma_a \in \St(f)$, for each $a \in A_f$, hence the product \[\sigma\coloneq  \prod_{a \in A_f} \sigma_a \in S_m\]  commutes with $f=\gamma_1(c)=\gamma_2(c)$ in $E$. In view of \eqref{eq:sigma_a}, we have
		\begin{equation}\label{eq:sigma}
			\sigma^{-1} \gamma_2(q) \sigma=\gamma_1(q), \text{ for all } q \in Q.    
		\end{equation}
		Since $H=\langle c,Q \rangle $, it follows that  $\sigma^{-1} \gamma_2(h) \sigma=\gamma_1(h)$, for all $h \in H$.
		
		\medskip
		\noindent \emph{Case 2:} there is $r \in Q$ such that $r^{-1}cr=c^{-1}$ in $H$. Then the centralizer $Q^+ \coloneq Q \cap \C_H(c)$ has index $2$ in $Q$ and $Q=Q^+ \sqcup rQ^+$. Set \[Q_i^+\coloneq \gamma_i(Q^+) \leqslant \C_{S_m}(f) ~\text{ and }~ r_i\coloneq \gamma_i(r) \in S_m, ~\text{ for }i=1,2.\]
		
		As before, the action of $Q_i^+$
		preserves $\Omega_a$ setwise, for each $a \in A_f$. On the other hand, since $r_i^{-1} f r_i=f^{-1}$ in $E$, for each $s \in r_iQ_i^+=Q_i^+ r_i$ we have 
       \[f(s.x)=(s^{-1} f s)(x)=(r_i^{-1} f r_i)(x)=(f(x))^{-1},~\text{ for each } x \in \Omega.\] 
        Therefore, $A_f^{-1}=A_f$ and $s.\Omega_a =\Omega_{a^{-1}}$, for all $a \in A_f$, $i=1,2$ and $s \in r_iQ_i^+$. 
		
		Let $I$ denote the set of all involutions in $A_f$ (i.e., elements satisfying $a=a^{-1}$). Choose a single element from each pair $\{a,a^{-1}\}$, where $a \in A_f\setminus I$, and let $J$ denote the set of all such representatives. Thus $A_f=I \sqcup J\sqcup J^{-1}$ and 
		\begin{equation}\label{eq:decomp_of_Omega}
			\Omega=\bigsqcup_{a \in I} \Omega_a \sqcup \bigsqcup_{a \in J} \left(\Omega_a \sqcup \Omega_{a^{-1}} \right).    
		\end{equation}
		
		If $a \in I$ is an involution, then the action of $Q_i$ preserves $\Omega_a$ setwise, for $i=1,2$, and we can argue as in Case~1 to find a permutation $\sigma_a \in \Sym(\Omega_a)$ such that \eqref{eq:sigma_a} holds.
		
		Now, consider any $a \in J$, so that $a \neq a^{-1}$. Then the action of $Q_i$ preserves the subset $\Omega_a \sqcup \Omega_{a^{-1}}$ setwise, so it decomposes in the union of $Q_i$-orbits:
		\[\Omega_a \sqcup \Omega_{a^{-1}}=\bigsqcup_{k=1}^l O_{i,k}, \text{ for } i=1,2.\]
		Note that, since $Q_i$ acts freely on $\Omega$, $|O_{i,k}|=|Q|$, for all $i=1,2$ and $k=1,\dots,l$, in particular, $l=|\Omega_a \sqcup \Omega_{a^{-1}}|/|Q|$ is independent of $i$. For each $k \in \{1,\dots,l\}$ the $Q_i$-orbit $O_{i,k}$ splits into two $Q_i^+$-orbits $O_{i,k} \cap \Omega_a$ and $O_{i,k} \cap \Omega_{a^{-1}}$, that are interchanged by the action of $r_i$, $i=1,2$. 
		
		Choose arbitrary basepoints $p_{i,k} \in O_{i,k} \cap \Omega_a$, for all $k=1,\dots,l$ and $i=1,2$. Let $\sigma_a \in \Sym(\Omega_a\sqcup\Omega_{a^{-1}})$ be the permutation defined as follows: 
		\begin{equation}\label{eq:sigma_a_on_J-def} 
			\sigma_a.(\gamma_1(t). p_{1,k}) \coloneq \gamma_2(t). p_{2,k}\,, \text{ for all } t \in Q  \text{ and } k=1,\dots,l.    
		\end{equation}
		We shall now  check that 
		\begin{equation}\label{eq:conj_by_sigma_a_in_J}
			\left(\sigma_a^{-1} \gamma_2(q) \sigma_a\right).x=\gamma_1(q).x, \text{ for all } q \in Q \text{ and } x \in \Omega_a \sqcup\Omega_{a^{-1}}.     
		\end{equation}
		Indeed, given any $x \in \Omega_a \sqcup\Omega_{a^{-1}}$, there exist unique $k \in \{1,\dots,l\}$ and $t \in Q$  such that $x=\gamma_1(t). p_{1,k}$. Then, in view of \eqref{eq:sigma_a_on_J-def}, for any $q \in Q$ we have
		\begin{align*}
			\left(\sigma_a^{-1}\gamma_2(q) \sigma_a\right).x &= \left(\sigma_a^{-1}\gamma_2(q) \sigma_a\right).(\gamma_1(t). p_{1,k}) =  \left(\sigma_a^{-1}\gamma_2(q)\right).(\gamma_2(t). p_{2,k}) \\
			&=\sigma_a^{-1}. (\gamma_2(qt) . p_{2,k})=\gamma_1(qt) . p_{1,k}=\gamma_1(q). (\gamma_1(t) . p_{1,k}) \\ &=\gamma_1(q).x .
		\end{align*}
		Thus, \eqref{eq:conj_by_sigma_a_in_J} has been verified.
		
		For any $x \in \Omega_a$ there are unique $k \in \{1,\dots,l\}$ and $t \in Q$ such that $x=\gamma_1(t). p_{1,k}$. Note that $t \in Q^+$ because both $x$ and $p_{1,k}$ are in $\Omega_a$ and any element from $\gamma_1(rQ^+)$ interchanges $\Omega_a$ with $\Omega_{a^{-1}}$. Applying \eqref{eq:sigma_a_on_J-def}, we see that $\sigma_a.x \in Q_2^+. p_{2,k}$ is also in $\Omega_{a}$, thus $\sigma_a.\Omega_a \subseteq \Omega_a$. Since $\Omega_a\sqcup \Omega_{a^{-1}}$ is a finite set and $\sigma_a \in \Sym(\Omega_a \sqcup\Omega_{a^{-1}})$, we can conclude that 
		\[\sigma_a.\Omega_a=\Omega_a \text{ and } \sigma_a.\Omega_{a^{-1}}=\Omega_{a^{-1}}.\]
		It follows that $\sigma_a \in \St(f) \leqslant \C_{E}(f)$ (again, we are treating $\Sym(\Omega_a \sqcup\Omega_{a^{-1}})$ as a subgroup of $S_m$).
		
        We can now define the permutation $\sigma \in S_m$ by the formula
		\[\sigma \coloneq \prod_{a \in I \sqcup J} \sigma_a,\]
		and conclude that it belongs to the centralizer $ \C_E(f)$ in $E$. Moreover, in view of \eqref{eq:decomp_of_Omega}, \eqref{eq:sigma_a} (when $a \in I$) and $\eqref{eq:conj_by_sigma_a_in_J}$ (when $a \in J$), we see that 
		\eqref{eq:sigma} is satisfied. Therefore, $\sigma^{-1} \gamma_2(h) \sigma=\gamma_1(h)$, for all $h \in H$, and the lemma is proved.
	\end{proof}

The following lemma takes care of the easy case in \Cref{thm:amalg_of_virt_ab}, when the amalgamated subgroup $G_0$ is finite.

\begin{lemma}\label{lem:amalg_over_fini_sbgp}
Suppose that $G=G_1*_{G_0} G_2$ is an amalgamated free product of finitely generated virtually abelian subgroups $G_1,G_2$ over a common finite subgroup $G_0$. Then there is a finitely generated virtually abelian group $E$ and a homomorphism $\nu:G \to E$ that is injective on each $G_i$, $i=1,2$.    
\end{lemma}

\begin{proof} The groups $G_1,G_2$ have property (VRC) by \Cref{lem:vabs_have_LR}, and since $|G_0|<\infty$, the amalgamated free product $G$ has (VRC) by \cite[Corollary~6.5]{virtprops}. According to \cite[Proposition~1.5]{virtprops}, $G_1,G_2 \vr G$, therefore, the existence of $E$ and $\nu:G \to E$ follows from \Cref{lem:further_props_of_VRC}.    
\end{proof}

	We are finally ready to prove the main technical result for amalgamated products of virtually abelian groups over a virtually cyclic subgroup.

	\begin{proof}[Proof of \Cref{thm:amalg_of_virt_ab}] \Cref{lem:amalg_over_fini_sbgp} takes care of the case when $|G_0|<\infty$, so we can further assume that $G_0$ is infinite.
    
    Let $\varphi_i:G_0 \to G_i$ be the subgroup inclusions, $i=1,2$, so that the amalgamated product $G$ has the presentation 
		\begin{equation}%\label{eq:pres_of_G-new}
			G=\langle G_1,G_2 \mid \varphi_1(g)=\varphi_2(g), \text{ for all } g \in G_0\rangle.    
		\end{equation}
		
		Choose finite index free abelian subgroups $N_i \n_f G_i$, $i=1,2$. Since $G_1,G_2$ are finitely generated and virtually abelian, they are both quasipotent and cyclic subgroup separable (for example, by \cite[Theorem~5.5]{Bur-Mar}), hence we can apply \cite[Lemma~3.3]{Bur-Mar} to find finite index normal subgroups $A_i \n_f G_i$, such that $A_i \subseteq N_i$, $i=1,2$, and $A_1 \cap G_0=A_2 \cap G_0$. In other words, if we let $A_0\coloneq A_1 \cap G_0=A_2 \cap G_0$, then $\varphi_i$ becomes an injective morphism in the category $\mathcal{C}$ (see \Cref{def:categ}) between the objects $(G_0,A_0)$ and $(G_i, A_i)$, for each $i=1,2$.
		Note that $A_0 \cong \Z$ as it is free abelian and has finite index in the infinite virtually cyclic group $G_0$.

		Set $Q\coloneq G_0/A_0$ and $n\coloneq |Q| \in \N$. Let $P_{n,i}=P_n(G_i,A_i)$ and $B_{n,i}=B_n(G_i,A_i)$, $i=0,1,2$, be the groups together with their normal abelian subgroups given by \Cref{def:compl}, and let $\xi_{n,i}:(G_i,A_i) \to (P_{n,i},B_{n,i})$ be the corresponding injective morphisms (see \Cref{lem:xi_n_is_morphism}), $i=0,1,2$. Recall that $B_{n,i} \cong A_i$ and $P_{n,i}/B_{n,i} \cong G_i/A_i$, for $i=0,1,2$, by \Cref{lem:xi_n_is_morphism}. 
		
		Now, by \Cref{lem:functorial}, the injective morphisms $\varphi_i:(G_0,A_0) \to (G_i,A_i)$ extend to  injective morphisms $\tilde\varphi_i:(P_{n,0},B_{n,0}) \to (P_{n,i},B_{n,i}) $, $i=1,2$. Therefore, we have the following commutative diagram, where all of the maps are injective:
		\begin{equation*}
			% https://tikzcd.yichuanshen.de/#N4Igdg9gJgpgziAXAbVABwnAlgFyxMJZABgBoBGAXVJADcBDAGwFcYkQBxAfXJAF9S6TLnyEUAJgrU6TVu27j+gkBmx4CRMuOkMWbRCAAKXYGAp8lQtaKKTtNXXIPHTpcRYFWRGlOVLEdWX1OLmJLFWF1MWQ-Kgcg9hczYg9pGCgAc3giUAAzACcIAFskMhAcCCRJGT12AB06gA8sEzNyCxpGegAjGEZDSJsDfKwMgAsccILipD9yysQAZnjagwbm1rcOkC7e-sGfHZhcyc8QaZLEABYaCtKVpxAGtDGW3k6evoHrQ5Hx0+UFyQN3mswewWery4ijOQMQAFZbgtljVHg08IxYJCWjDAYVLojQYhqo4IXUMVi6i83iAPntvt4xCA-hMpvjgUikITSfUmi1XCl+JQ+EA
			\begin{tikzcd}
				& G_0 \arrow[ld, "\varphi_1"'] \arrow[rd, "\varphi_2"] \arrow[d, "{\xi_{n,0}}"] &                              \\
				G_1 \arrow[d, "{\xi_{n,1}}"'] & {P_{n,0}} \arrow[rd, "\tilde\varphi_2"] \arrow[ld, "\tilde\varphi_1"']        & G_2 \arrow[d, "{\xi_{n,2}}"] \\
				{P_{n,1}}                     &                                                                         & {P_{n,2}}                   
			\end{tikzcd}    
		\end{equation*}
		This allows us to consider the amalgamated free product
		\[P\coloneq P_{n,1}*_{P_{n,0}} P_{n,2} \coloneq\langle P_{n,1},P_{n,2} \mid \tilde\varphi_1(g)=\tilde\varphi_2(g), \text{ for all } g\in P_{n,0} \rangle .\]
		Since $\tilde \varphi_i$ extends $\varphi_i$, for $i=1,2$, we have a homomorphism 
		\begin{equation}\label{eq:varkappa}
			\varkappa:G=G_1*_{G_0} G_2 \to P = P_{n,1}*_{P_{n,0}} P_{n,2},    
		\end{equation}
		whose restriction to $G_i$ is $\xi_{n,i}$, for $i=1,2$. In particular, $\varkappa$ is injective on $G_1$ and $G_2$ (however, it need not be injective on all of $G$).
		
		Recall that by the choice of $n=|Q|$ and \Cref{prop:splits}, the group $P_{n,0}$ splits as a semidirect product $B_{n,0} \rtimes Q$, where $B_{n,0} \cong A_0 \cong \Z$. Set $C\coloneq B_{n,0}$ and $H \coloneq P_{n,0}$, so that $\tilde\varphi_i(C) \subseteq B_{n,i}$ is infinite cyclic and $\tilde\varphi_i(Q)$ is a finite subgroup of $P_{n,i}$, for $i=1,2$.
		We can now apply \Cref{prop:embedding_for_virt_ab_gps} to find $k,l \in \N$ and group embeddings \[\eta_1:P_{n,1} \to E_1\coloneq \Z \wr_{\Omega_k} S_k,\quad\eta_2:P_{n,2} \to E_2\coloneq \Z \wr_{\Omega_l} S_l,\] such that $\Omega_k=\{1,\dots,k\}$, $\Omega_l=\{1,\dots,l\}$, and for $\alpha_i \coloneq\eta_i\circ \tilde\varphi_i: H \to E_i$, $i=1,2$, we have
		\[\alpha_1(C) \subseteq \Z^{\Omega_k},~\alpha_2(C) \subseteq \Z^{\Omega_l},~\alpha_1(Q) \subseteq S_k,~\alpha_2(Q) \subseteq S_l,\]
		and $\alpha_i(Q)$ acts freely on $\Omega_i$, $i=1,2$.
		
		Observe that we have a natural monomorphism 
		\begin{equation}\label{eq:lambda}
			\lambda: P = P_{n,1}*_{P_{n,0}} P_{n,2} \to E_1*_H E_2,   
		\end{equation}
		where the latter amalgamated product is defined by the presentation
		\begin{equation}\label{eq:pres_of_E_1*E_2}
			E_1*_H E_2\coloneq \langle E_1,E_2 \mid \alpha_1(h)=\alpha_2(h), \text{ for all } h \in H\rangle.
		\end{equation}
		
		Let $c$ be a generator of $C$ and set $c_1 \coloneq \alpha_1(c) \in E_1$, $c_2\coloneq \alpha_2(c) \in E_2$. By \Cref{lem:amalgamation}, for $m\coloneq kl \in \N$ and $E\coloneq \Z \wr_{\Omega_m} S_m$, there exist monomorphisms $\beta_i: E_i \to E$ enjoying all the properties from its claim. In particular, if we define $\gamma_i\coloneq \beta_i \circ \alpha_i:H \to E$ then all the conditions of \Cref{lem:cyc-by-fin_in_Z_wr_S} will be satisfied, so we can find $\sigma \in S_m$ such that 
		\begin{equation}\label{eq:sigma-gamma_1}
			\sigma \gamma_1(h) \sigma^{-1}=\gamma_2(h) \text{ in } E, \text{ for all } h \in H.
		\end{equation}
		
		Let $\delta:E \to E$ be the inner automorphism corresponding to conjugation by $\sigma$. Equation \eqref{eq:sigma-gamma_1} shows that after we replace $\beta_1$  by $\delta\circ\beta_1$, we will have
		\[\beta_1(\alpha_1(h))=\beta_2(\alpha_2(h)), \text{ for all } h \in H.\]
		Since $E_1*_H E_2$ has presentation \eqref{eq:pres_of_E_1*E_2}, it follows that we have a homomorphism $\mu:E_1*_H E_2 \to E$ such that the restriction of $\mu$ to $E_i$ is $\beta_i$, $i=1,2$. In particular, $\mu$ is injective on $E_1$ and $E_2$. Therefore, the homomorphism \[\nu\coloneq \mu \circ\lambda\circ\varkappa: G \to E,\]
		where $\varkappa$ and $\lambda$ are given by \eqref{eq:varkappa} and \eqref{eq:lambda} respectively, is injective on $G_i$, for $i=1,2$. Thus the theorem is proved.
	\end{proof}

%%%%%%%%%%%%%%%%%%%                             %%%%%%%%%%%%%%%%
%%%%%%%%%%%%%%%%%%%           New Section       %%%%%%%%%%%%%%%%
%%%%%%%%%%%%%%%%%%%                             %%%%%%%%%%%%%%%%

\section{Amalgamated products of (VRC) groups over virtually cyclic subgroups have (VRC)} \label{sec:amalgams_of_VRC}
In this section we give first applications of \Cref{thm:amalg_of_virt_ab}; in particular, we prove \Cref{thm:amalgam_of_(VRC)_gps} from the Introduction. The following important consequence of \Cref{thm:amalg_of_virt_ab} will be used throughout the rest of the paper.

% \begin{cor}\label{cor:amalg_of_vabs_has_VRC} If $G$ is an amalgamated free product of two finitely generated virtually abelian groups over a virtually cyclic subgroup then $G$ has (VRC) and every finitely generated virtually abelian subgroup is separable in $G$.    
% \end{cor}

% \begin{proof} The first claim follows immediately from \Cref{thm:amalg_of_virt_ab} and \cite[Corollary~6.7]{MM-vr_in_free_constr}. And the second statement follows from \Cref{lem:VRC->vabs_sep}.    
% \end{proof}

\begin{prop}\label{prop:maps_from_amalgams_to_vab_gps} Let $G_1,G_2$ be groups with a common virtually cyclic subgroup $G_0$, and let $G \coloneq G_1*_{G_0} G_2$. Suppose that we are given normal subgroups $N_i \n G_i$, $i=1,2$, and $N \n G$, such that $G_i/N_i$, $i=1,2$, and $G/N$ are finitely generated and virtually abelian. If $G_0 \vr G_i$, for $i=1,2$, then for arbitrary finitely generated virtually abelian virtual retracts $H_i \vr G_i$, $i=1,2$, there exists a finitely generated virtually abelian group $P$ and a homomorphism $\psi:G \to P$ such that 
\begin{itemize}
    \item $\ker\psi \cap G_i \subseteq N_i$, for each $i=1,2$;
    \item $\ker\psi \subseteq N$;
    \item $\psi$ is injective on $G_0$ and on each subgroup $H_{i}$, $i=1,2$.
\end{itemize}
\end{prop}

\begin{proof} 
    According to \Cref{lem:further_props_of_VRC}, for each $i=1,2$ there exist a finitely generated virtually abelian group $P_i$ and a homomorphisms $\varphi_i\colon G_i \to P_i$, such that $\ker \varphi_i \subseteq N_i$ and $\varphi_i$ is injective on $H_i$ and on $G_0$. The latter  allows us to view $G_0$ as a subgroup of each $P_i$, $i=1,2$. Then, using the universal property of amalgamated free products, we obtain a homomorphism
    \[
        \varphi:  G \to P_1 *_{G_0} P_2,
    \]
    such that $\varphi|_{G_i}=\varphi_i$, for $i=1,2$. By \Cref{thm:amalg_of_virt_ab}, there exists a homomorphism $\nu\colon P_1 *_{G_0} P_2 \to E$ such that $E$ is a finitely generated virtually abelian group and $\nu$ is injective on $P_1$ and $P_2$. We then have that the group $P\coloneq E \times G/N$ and the homomorphism 
    \[
        \psi \coloneq (\nu \circ \varphi) \times \pi \colon  G \to P,
    \]
    where $\pi:G \to G/N$ is the quotient map, satisfies all conditions of the statement, by construction.    
\end{proof}

\begin{proof}[Proof of \Cref{thm:amalgam_of_(VRC)_gps}]
Let $G=G_1*_{G_0} G_2$ be the free amalgamated product of two groups $G_1$ and $G_2$ over a virtually cyclic subgroup $G_0$. Suppose that $G_1$ and $G_2$ satisfy (VRC); in particular, $G_0 \vr G_i$, for $i=1,2$, by \Cref{lem:VRC->vabs_sep}.  
Now, according to \cite[Corollary~8.5]{virtprops} (or by \Cref{lem:VRC->vabs_sep} and \Cref{cor:res_fin_amalg} below), $G$ is cyclic subgroup separable, hence the virtually cyclic subgroup $G_0$ is separable in $G$ (as a finite union of left cosets to its finite index cyclic subgroup). This allows us to apply \cite[Corollary~6.6]{MM-vr_in_free_constr} to deduce that for each \emph{hyperbolic element} $g \in G$ (that is, $g$ is not conjugate to an element of $G_i$, for $i=1,2$), we have $\langle g \rangle \vr G$.

Now, suppose that $g \in G$ is an \emph{elliptic element}, i.e., $g=x h x^{-1}$, where $h \in G_i$, for some $i \in \{1,2\}$, and $x \in G$. Since $G_i$ has (VRC), we know that $\langle h \rangle \vr G_i$, therefore,  by \Cref{prop:maps_from_amalgams_to_vab_gps}, there exist a finitely generated virtually abelian group $P$ and a homomorphism $\psi:G \to P$ such that $\psi$ is injective on $\langle h \rangle$ (and, consequently, on $\langle g \rangle$). By \Cref{lem:vabs_have_LR}, $\psi (\langle g \rangle) \vr P$, hence $\langle g \rangle \vr G$ by \Cref{lem:basic_props_of_virt_retr}.(a). 
Thus $G$ has (VRC).
\end{proof}

\begin{cor} Let $G$ be the fundamental group of a finite tree of groups, where all the vertex groups have (VRC) and all the edge groups are virtually cyclic. Then $G$ has (VRC).    
\end{cor}

The next corollary is useful for describing the structure of an amalgamated product of two groups over a common virtually cyclic virtual retract.

\begin{cor}\label{cor:G_0_is_vr_in_G} Suppose that $G=G_1*_{G_0} G_2$, where $G_0$ is virtually cyclic and $G_0 \vr G_i$, for $i=1,2$. Then
\begin{itemize}
    \item[(i)] every finitely generated virtually abelian virtual retract $H \vr G_i$, $i=1,2$, is a virtual retract of $G$; in particular, $G_0 \vr G$.
    \item[(ii)] If $|G_0|<\infty$ then there is a finite index subgroup $F \n_f G$ such that $F=F_0*F_1* \dots* F_k$, where $k \in \N_0$, $F_0$ is a finitely generated free group (possibly trivial) and each $F_j$ is isomorphic to a finite index subgroup of some $G_i$, for $i=i(j) \in \{1,2\}$.
        \item[(iii)] If $G_0$ is infinite then there is a finite index normal subgroup $K \n_f G$ isomorphic to a semidirect product
    \[K \cong (N_0*N_1* \dots * N_k)\rtimes \Z,\] where $k \in \N_0$, $N_0$ is a finitely generated free group (possibly trivial), and for each $j=1,\dots,k$ there is $i=i(j) \in \{1,2\}$ such that $N_j$ is isomorphic to some $M_{j} \n (G_i \cap K)$ with $(G_i \cap K)/M_j \cong \Z$.
\end{itemize}
\end{cor}

\begin{proof} (i) According to 
\Cref{prop:maps_from_amalgams_to_vab_gps}, there
exists a group homomorphism $\psi:G \to P$, where $P$ is a finitely generated virtually abelian group,  such that $\psi$ is injective on $H$. \Cref{lem:vabs_have_LR} tells us that $\psi(H) \vr P$, so, by \Cref{lem:basic_props_of_virt_retr}.(a), $H \vr G$.

(ii) By part (i), $G_0 \vr G$, so there is a  subgroup $G' \leqslant G$ such that $G_0$ normalizes $G'$, $G' \cap G_0=\{1\}$ and $|G:G_0G'|<\infty$. Since $|G_0|<\infty$, we can conclude that $|G:G'|<\infty$, so there exists $F \n_f G$, with $F \subseteq G'$. In particular, $F \cap G_0=\{1\}$, hence, by \Cref{thm:kurosh}, $F$ has a free product decomposition with the desired properties.

(iii) Again, since $G_0 \vr G$, we can apply \cite[Lemma~12.2]{MM-vr_in_free_constr} to find $K\n_f G$ admitting a homomorphism $\xi:K \to \Z$ such that $\xi$ is injective on $K \cap g G_0 g^{-1}$, for all $g \in G$. By \Cref{thm:kurosh}, $K$ decomposes as the fundamental group of a finite graph of groups $(\mathcal{K},\Gamma)$, where each vertex group $K_v$, $v \in V\Gamma$, is isomorphic to $K \cap G_i$, for some $i \in \{1,2\}$, and each edge group $K_e$, $e\in E\Gamma$, is isomorphic to $K \cap G_0$.

Since $|G_0|=\infty$, $K \cap G_0 \n_f G_0$ and $\xi$ is injective on $\alpha_e(K_e)$, we see that 
$\xi(K)$ is non-zero (so we can assume that $\xi$ is surjective) and $\xi(\alpha_e(K_e))$ has finite index in $\Z$, for all $e \in E\Gamma$. Hence, for $N\coloneq \ker\xi \n K$ we have
\[K \cong N \rtimes \Z,~N \cap \alpha_e(K_e)=\{1\}\text{ and }
|K:N\alpha_e(K_e)|<\infty, \text{ for all } 
e \in E\Gamma.\]  
Therefore, we can combine \Cref{cor:normal_sbgp_in_graph_of_gps_is_fg} below with \Cref{thm:kurosh} to conclude that $N$ splits as the free product $N_0*N_1*\dots*N_k$, where $N_0$ is free of finite rank and for each $j=1 \dots, k$, $N_j$ is isomorphic to a normal subgroup $M_j$ of some $K_v$, $v \in V\Gamma$, with $K_v/M_j \cong \Z$.
\end{proof}

%%%%%%%%%%%%%%%%%%%                             %%%%%%%%%%%%%%%%
%%%%%%%%%%%%%%%%%%%           New Section       %%%%%%%%%%%%%%%%
%%%%%%%%%%%%%%%%%%%                             %%%%%%%%%%%%%%%%

\section{Residual properties of amalgamated free products}\label{sec:res_props}
In this section we give applications of \Cref{thm:amalg_of_virt_ab} to residual properties of amalgamated free products over virtually cyclic subgroups. 

\begin{notation}\label{not:G_for_res_props}
Throughout this section we assume that $G=G_1*_{G_0} G_2$ is an amalgamated free product of groups $G_1,G_2$, where $G_0$ is virtually cyclic and $G_0 \vr G_i$, for $i=1,2$.     
\end{notation}
The next statement establishes claim (ii) of \Cref{cor:virt_res_solv}.

\begin{cor}\label{cor:claim_(ii)-virt_res_solv} If $G_1$ and $G_2$ are virtually residually solvable then so is $G$.   \end{cor}

\begin{proof}
Let $N_i \n_f G_i$ be a finite index normal residually solvable subgroup, $i=1,2$. By \Cref{prop:maps_from_amalgams_to_vab_gps}, there exist a finitely generated virtually abelian group $P$ and $\psi:G \to P$ such that $\psi$ is injective on $G_0$ and 
\begin{equation}\label{eq:int_of_kerpsi_with_G_i}
\ker\psi \cap G_i \subseteq N_i,~\text{ for } i=1,2. \end{equation}
Note that $N \coloneq \ker\psi\n G$ trivially intersects each conjugate of $G_0$ in $G$. Therefore, by a generalization of Kurosh's theorem \cite[Theorem~14 in Section~I.5.5]{Serre} (or by \Cref{thm:kurosh}), $N$ is isomorphic to a free product $*_{j \in J} M_j*F$, where $F$ is a free group and for each $j \in J$ there is $g_j \in G$ and $i =i(j) \in \{1,2\}$ such that $M_j=N \cap g_jG_i g_j^{-1}$. 

Inclusion \eqref{eq:int_of_kerpsi_with_G_i} tells us that for every $j \in J$, $M_j \leqslant g_jN_{i(j)} g_j^{-1}$, so $M_j$ is residually solvable. Since $F$ is also residually solvable, we can apply a result of Gruenberg \cite[Corollary after Theorem~4.1]{Gruenberg-root} to deduce that $N$ is residually solvable. Let $A \n_f \psi(G)$ be an abelian subgroup of finite index and let $H=\psi^{-1}(A) \n_f G$.
Then $H$ is an extension of the residually solvable group $N$ by the abelian group $A$,  hence $H$ is residually solvable (see \cite[Lemma~1.5]{Gruenberg-root}), and the proof is complete.
%
% Let $A \n_f P$ be such that $A \cong \Z^n$, for some $n \in \N_0$. By making $A$ smaller, we can always ensure that 
% \begin{equation}\label{eq:int_of_A_with_psi_G_i} A \cap \psi(G_i) \subseteq \psi(N_i), ~\text{ for }i=1,2\end{equation}
% (this can be easily seen by using the classification of finitely generated abelian groups, or by combining LERF-ness of $A$ with \cite[Lemma~4.17]{Min-Min}).
% Denote $H \coloneq \psi^{-1}(A) \n_f G$, and observe that, in view of \eqref{eq:int_of_A_with_psi_G_i} and \eqref{eq:int_of_kerpsi_with_G_i}, we have
% \begin{equation}\label{eq:int_of_H_with_conj_of_G_i}
% H \cap gG_ig^{-1}=g(H \cap G_i)g^{-1} \subseteq gN_ig^{-1},~\text{ for all } g \in G \text{ and }i=1,2.      
% \end{equation}
\end{proof}

We now focus on showing that if $G_1$ and $G_2$ are residually finite, then so is $G$.
We will say that a subgroup $H$ of  $G$ is \emph{topologically embedded} if the profinite topology on $G$ induces the full profinite topology on $H$ (i.e., every closed subset of $H$ is an intersection of $H $ with a closed subset of $G$). This is equivalent to saying that for each $H' \leqslant_f H$ there is $G' \leqslant_f G$ such that $H \cap G' \subseteq H'$ (cf. \cite[p.~1707]{Lorensen}).

\begin{lemma}\label{lem:factors_are_top_embedded}  For each $i=0,1,2$, $G_i$ is topologically embedded in $G$.
\end{lemma}

\begin{proof} Recall that  $G_0 \vr G$, by \Cref{cor:G_0_is_vr_in_G}.(i). Thus there exist $K \leqslant_f G$, with $G_0 \subseteq K$, and a retraction $\rho:K \to G_0$. So, for any $G_0' \leqslant_f G_0$, we have 
\[G' \coloneq \rho^{-1}(G_0') \leqslant_f G~\text{ and }~G' \cap G_0=G_0'. \]
Therefore, $G_0$  is topologically embedded in $G$.

Now, consider any  $i \in \{1,2\}$, and suppose that $N_i \leqslant_f G_i$. By \Cref{prop:maps_from_amalgams_to_vab_gps}, there is a finitely generated virtually abelian group $P$ and a homomorphism $\psi:G \to P$ such that $\ker\psi \cap G_i\subseteq N_i$. Since $\psi(N_i)$ is separable in $P$ (by \Cref{rem:basic_sep_facts}), there exists a finite index subgroup $L\leqslant_f P$ such that $L \cap \psi(G_i)=\psi(N_i)$ (cf. \cite[Lemma~4.17]{Min-Min}). Then $K \coloneq \psi^{-1}(L) \leqslant_f G$ satisfies $K \cap G_i=N_i$. Thus $G_i$ is topologically embedded in $G$.    
\end{proof}

\begin{lemma}\label{lem:edge-approx} Suppose that $G_0$ is separable in $G_i$, for $i=1,2$. Then $G$ is {edge-approximated by the family of finitely generated virtually abelian groups}, in the sense of \cite[Definition~6.1]{MM-vr_in_free_constr}.
This means that for arbitrary finite subset $F_i \subseteq G_i$, $i=1,2$, there is a finitely generated virtually abelian group $P$ and a homomorphism $\psi:G \to P$ such that 
\begin{equation}\label{eq:edge-approx} \psi(F_i\setminus G_0) \subseteq P \setminus \psi(G_0),~ \text{ for }i=1,2,\end{equation}
and $\psi(G_0)$ is separable in $P$.
\end{lemma}

\begin{proof} 
Let $F_i \subseteq G_i$ be an arbitrary finite subset, $i=1,2$. Since $G_0$ is separable in $G_i$, according to   \cite[Lemma~2.2]{MM-vr_in_free_constr}there exists $N_i \n_f G_i$ such that 
\begin{equation}\label{eq:not_in_N_iG_0}
F_i \cap G_0N_i=F_i \cap G_0,~ \text{ for } ~i=1,2.
\end{equation}
Now, by \Cref{prop:maps_from_amalgams_to_vab_gps} there is a finitely generated virtually abelian group $P$ and a homomorphism $\psi:G \to P$ such that $\ker\psi \cap G_i \subseteq N_i$, for $i=1,2$. If $f \in F_i$, for some $i \in \{1,2\}$, and $\psi(f) \in \psi(G_0)$ in $P$, then 
\[f \in \psi^{-1}(\psi(G_0)) \cap G_i=G_0 \ker\psi \cap G_i=G_0(\ker\psi \cap G_i) \subseteq G_0N_i,\]
hence $f \in G_0$ by \eqref{eq:not_in_N_iG_0}. Thus \eqref{eq:edge-approx} is satisfied. Moreover, $\psi(G_0)$ is separable in $P$,   by \Cref{rem:basic_sep_facts}.
\end{proof}

\begin{prop}\label{prop:factors_are_sep} Assume that $G_0$ is separable in $G_1$ and in $G_2$. Then
\begin{itemize}
    \item[(i)]  $G_1$, $G_2$ and $G_0$ are all separable in $G$. It follows that  every separable subset of $G_i$ is separable in $G$, for $i=0,1,2$.  
\item[(ii)] If $g \in G$ is not conjugate to an element of $G_1 \cup G_2$  then $\langle g \rangle \vr G$.    
\end{itemize}
\end{prop}

\begin{proof} In view of \Cref{lem:edge-approx}, we can apply \cite[Proposition~6.4]{MM-vr_in_free_constr} to deduce that $G_1$ and $G_2$ and $G_0=G_1 \cap G_2$ are separable in $G$. Since $G_i$ is topologically embedded in $G$ (see \Cref{lem:factors_are_top_embedded}), for $i=0,1,2$, this implies  claim (i). 
Claim (ii) follows from \Cref{lem:edge-approx} and \cite[Theorem~6.3]{MM-vr_in_free_constr}.
\end{proof}

In view of \Cref{def:sep_props} and the fact that  virtual retracts of residually finite groups are always separable (see \cite[Lemma~2.2]{virtprops}), \Cref{prop:factors_are_sep} yields the following amplification of \Cref{cor:virt_res_solv}.(i).

\begin{cor}\label{cor:res_fin_amalg} Assume that $G=G_1 *_{G_0} G_2$, where $G_0$ is virtually cyclic and $G_0 \vr G_i$, for $i=1,2$.
If $G_1$ and $G_2$ are  residually finite (respectively, cyclic subgroup separable) then $G$ is residually finite (respectively, cyclic subgroup separable).
\end{cor}

% In view of \cite[Lemma~8.2]{virtprops}, an alternative proof of \Cref{cor:res_fin_amalg} can probably be extracted from the paper of Burillo and Martino \cite{Bur-Mar}.

\begin{rem} An amalgamated product of two residually finite groups over a common retract is well-known to be residually finite (see \cite[Theorem~1]{Boler-Evans}), but this may fail if the amalgamated subgroup is only assumed to be a \emph{virtual} retract of the factors. The most prominent examples of this failure are the simple groups constructed by Burger and Mozes \cite{Burg-Moz}, which decompose as amalgamated free products of finite rank free groups over subgroups that have finite index in both factors. Another class of interesting examples was given by Leary and Minasyan in \cite[Section~11]{Lea-Min}, who constructed amalgamated products of virtually $\Z^2$ groups over finite index subgroups that are not residually finite (and can even be non-Hopfian).
\end{rem}

% The argument in \Cref{cor:virt_res_solv} used the fact that solvability is a \emph{root} property in the sense of Gruenberg \cite{Gruenberg-root}, hence every (residually solvable)-by-solvable group is residually solvable. Unfortunately, amenability is not a root property, and there exist (residually amenable)-by-amenable groups that are not residually amenable, as pointed out by Berlai in \cite[Proposition~1.4]{Berlai}. Therefore, the proof of \Cref{cor:res_amen} from the introduction requires more work. We start by establishing a special case.

% \begin{lemma}\label{lem:amalgam_of_am_is_res_am} Suppose that $G=G_1*_{G_0} G_2$, where $G_1,G_2$ are amenable groups. If $G_0 \vr G$ then $G$ is residually amenable.    
% \end{lemma}

% \begin{proof}
    
% \end{proof}

%%%%%%%%%%%%%%%%%%%                             %%%%%%%%%%%%%%%%
%%%%%%%%%%%%%%%%%%%           New Section       %%%%%%%%%%%%%%%%
%%%%%%%%%%%%%%%%%%%                             %%%%%%%%%%%%%%%%

\section{Finiteness properties of normal subgroups in   graphs of groups}\label{sec:fin_props}
In this section we study necessary and sufficient conditions for a normal subgroup of the fundamental group of a finite graph of groups to be finitely generated and to  satisfy higher finiteness properties.

Recall that an action of a group on a tree is said to be \emph{minimal} if there is no non-empty proper invariant subtree. 
If $(\mathcal{G},\Gamma)$ is a graph of groups, following Bass \cite[Section~7]{BASS1993} we will say that a vertex $v \in V\Gamma$ is \emph{terminal} if there is a unique edge $e \in E\Gamma$ such that $\alpha(e)=v$ (in particular, $\alpha(e) \neq \omega(e)$) and $G_v=\alpha_e(G_e)$. Clearly, any terminal vertex can be removed from the graph of groups without changing its fundamental group. 

\begin{lemma}[{\cite[Proposition~7.12]{BASS1993}}]
\label{lem:minmal_action<=>no_terminal_vertices}
Let $(\mathcal{G},\Gamma)$ be  a finite graph of groups. Fix a maximal tree $T$ and an orientation on $\Gamma$, and suppose that $G=\pi_1(\mathcal{G},\Gamma,T,E\Gamma^+)$ and $\mathcal T$ is the corresponding Bass-Serre tree. Then the action of $G$ on $\mathcal{T}$ is minimal if and only if $(\mathcal{G},\Gamma)$ contains no terminal vertices.
\end{lemma}

\begin{lemma}\label{lem:minimal_action_for_fi_sbgp}
If $G$ is a finitely generated group acting minimally on a tree $\mathcal T$ then every finite index subgroup $G'\leqslant_f G$    also acts minimally on $\mathcal T$.
\end{lemma}

\begin{proof} Without loss of generality we may assume that $\mathcal T$ has at least one edge (otherwise the statement is obvious).

Let $M \n_f G$ be a finite index normal subgroup contained in $G'$. If $M$ fixes a vertex of $\cT$ then the finite group $G/M$ acts on the subtree $\mathrm{Fix}(M)$, consisting of $M$-fixed points, hence $G/M$ must fix a vertex $v \in \mathrm{Fix}(M)$ (see \cite[Example~6.3.1 in Section~I.6]{Serre}). Therefore, $v$ is a global fixed vertex for the action of $G$ on $\mathcal T$, contradicting the assumption that $\mathcal T$ has at least one edge and the action of $G$ on it is minimal.

Thus, the action of $M$ does not fix a vertex of $\cT$. Since $M$ is finitely generated (because $G$ is), it must contain at least one hyperbolic element by \cite[Corollary~7.3]{BASS1993}. Therefore, $\cT$ contains a unique $M$-invariant subtree $\mathcal{S}$  such that the action of $M$ on $\mathcal{S}$ is minimal (see \cite[Proposition~7.5]{BASS1993}). Since $M \n G$, $\mathcal{S}$ will be invariant under the action of $G$, hence $\mathcal{S}=\cT$, by minimality. It follows that the action of $M$ on $\cT$ is minimal, yielding the same statement for the action of $G'$ on $\cT$.    
\end{proof}

The  following statement, observed by Bridson and Howie in \cite{Bri-How}, is the reason why minimality of the action is important for us. 

\begin{lemma}\label{lem:Bri-How}
Let $(\mathcal{G},\Gamma)$ be a finite graph of groups, whose fundamental group $G$ acts minimally on the corresponding Bass-Serre tree $\cT$ (this is equivalent to the absence of terminal vertices, by \Cref{lem:minmal_action<=>no_terminal_vertices}). If $N \n G$ is a finitely generated normal subgroup then at least one of the following is true: 
\begin{equation}\label{eq:N_is_in_G_e}
N \subseteq \alpha_e(G_e), ~\text{ for each }e \in E\Gamma,
\end{equation}
or 
\begin{equation}\label{eq:NG_e_has_f_i}
|G:N\alpha_e(G_e)|<\infty, ~\text{ for each }e \in E\Gamma. 
\end{equation}
\end{lemma}

\begin{proof} If $N \subseteq \alpha_f(G_f)$, for some $f \in E\Gamma$, then $N$ fixes an edge of $\cT$, and the set of fixed points of $N$ forms a non-empty subtree of $\cT$. Since $N \n G$ this subtree is $G$-invariant, hence it must be the whole of $\cT$, by minimality. Therefore, $N$ must fix every edge of $\cT$, which shows that \eqref{eq:N_is_in_G_e} is true.

Thus we can assume that $N \not\subseteq \alpha_e(G_e)$, for every $e \in E\Gamma$. Choose any edge $e \in E\Gamma$. By the definition of the Bass-Serre tree, $\alpha_e(G_e)$ is the $G$-stabilizer of some edge $a$ of $\cT$. By contracting all edges outside the orbit $G.a$ to points, we obtain a new $G$-tree $\mathcal{S}$. Since the action of $G$ on $\cT$ was minimal, so is the action of $G$ on $\mathcal{S}$. 
The group $G$ acts on $\mathcal{S}$ with one orbit of edges and $N$ does not fix any edge. 
We can now apply the argument from \cite[Proposition~2.2]{Bri-How} to deduce that $|G:N\alpha_e(G_e)|<\infty$, as required.
%\comment{Alternative second part of the proof: Since there are no terminal vertices, then we claim that $G$ naturally splits over all of its edge groups. Indeed, if an edge is not a cut edge, then $G$ naturally decomposes as a HNN extension over that edge, and thus splits over it. If the edge cuts the graph into two connected components (both without terminal vertices), $G$ decomposes as an amalgam over that edge. If this decomposition is trivial, then the edge group must be equal to the fundamental group of one of the connected components, and by minimality (see \cite{BASS1993}[Proposition 7.12 (b)]) this is only possible if that connected component is a vertex, thus making it a terminal vertex, i.e., a contradiction. So \cite[Proposition~2.2]{Bri-How} applies on all edge groups, giving \Cref{eq:NG_e_has_f_i} since $N$ is not contained in any edge group by the first paragraph.}
\end{proof}

\begin{lemma}\label{lem:normal_sbgp_in_graph_of_gps-prelim}
Let $(\mathcal{G}, \Gamma)$ be a finite graph of groups with fundamental group $G$, and let $\mathcal{T}$ be the corresponding Bass-Serre tree. Suppose that $N \n G$ is a normal subgroup satisfying \eqref{eq:NG_e_has_f_i}.
Then the induced action of $N$ on $\mathcal{T}$ is cocompact (i.e., $N\backslash \mathcal{T}$ is a finite graph). 
\end{lemma}

\begin{proof} From Bass-Serre Theory \cite[Section~I.5.3]{Serre}, we know that $G$ acts on $\mathcal T$ with finitely many orbits of vertices, finitely many orbits of edges and without edge inversion. Every vertex stabilizer in this action is a conjugate of $G_v \leqslant G$, for some $v \in V\Gamma$, and every edge stabilizer is a conjugate of $\alpha_e(G_e)$, for some $e \in E\Gamma$.

Let $e_1,\dots,e_k \in E\mathcal{T}$ be a finite set of representatives of the $G$-orbits of edges in $\mathcal T$, and let $E_i \leqslant G$ be the $G$-stabilizer of $e_i$, $i=1,\dots,k$. For each $i=1,\dots,k$, condition \eqref{eq:NG_e_has_f_i} implies that $|G:NE_i|<\infty$, hence there is $n_i \in \N$ and a collection of elements  $g_{i,1}, \dots, g_{i,n_i} \in G$ such that \[G= \bigsqcup_{j=1}^{n_i} g_{i,j} NE_i.\]
If $e \in E\mathcal{T}$ is an arbitrary edge, then there is a unique $i \in \{1,\dots,k\}$ such that $e \in G.e_i$. Since $N \n G$, it follows that for some $j \in \{1,\dots,n_i\}$ we have
\[e \in (g_{i,j}NE_i).e_i=(g_{i,j}N).e_i=N.(g_{i,j}.e_i).\]
Thus the induced action of $N$  on $\mathcal{T}$ has at most $\sum_{i=1}^k n_i$ orbits of edges. Consequently, there are also finitely many $N$-orbits of vertices in $\mathcal T$, so the quotient $ N \backslash \mathcal{T}$ is a finite graph, as claimed.
\end{proof}

\begin{cor}\label{cor:normal_sbgp_in_graph_of_gps_is_fg} Let $G$ be the fundamental group of a finite graph of groups $(\mathcal{G},\Gamma)$ and suppose that $N \n G$ satisfies condition \eqref{eq:NG_e_has_f_i}.    
Then $N$ splits as the fundamental group of a finite graph of groups $(\mathcal{N},\Delta)$, where
for every $w \in V\Delta$ and each $f \in E\Delta$ there are $v \in V \Gamma$ and $e \in E\Gamma$ such that  $N_w \cong N\cap G_v \leqslant G$ and $N_f \cong N \cap \alpha_e(G_e) \leqslant G$. And conversely, for each $v \in V\Gamma$ there is $w \in V\Delta$ such that $N \cap G_v \cong G_w$.
\end{cor}

\begin{proof} The statement follows from \Cref{lem:normal_sbgp_in_graph_of_gps-prelim} and the Structure Theorem of Bass-Serre Theory \cite[Theorem~13 in Section~I.5.4]{Serre}. 
\end{proof}

In the remainder of this section $R$ will denote a non-zero commutative ring with unity. 
Let us briefly recall the finiteness properties $F_m$ and $FP_m(R)$ (see \cite{Geogh} and \cite{Bieri-book} for details). A group $G$ is said to be of type $F_m$, if it admits an Eilenberg-Maclane space $K(G,1)$ with finite $m$-skeleton. In particular, $G$ is of type $F_1$ if and only if it is finitely generated, and of type $F_2$ if and only if it is finitely presented. A group $G$ is said to be of type $FP_m$ over $R$
if the trivial $RG$-module $R$ admits a resolution by projective $RG$-modules that are finitely generated up to degree $m$. 
Note that  $F_1$ coincides with $FP_1$ over $R$ (for any $R$), but for $m \ge 2$, $F_m$ is generally stronger and not always equivalent to $FP_m$ over $R$.

\begin{lemma}[{\cite[Corollary~7.2.4]{Geogh},\cite[Proposition~2.5 in Chapter~I]{Bieri-book}}]\label{lem:finiteness_props_of_fi_sbgps}
If $G$ is a group and $H \leqslant_f G$ is a subgroup of finite index then for each $m \in \N$
\begin{itemize}
\item    $G$ if of type $F_m$ if and only if $H$ is of type $F_m$;
\item   $G$ is of type $FP_m$ over $R$ if and only if $H$ is of  type $FP_m$ over $R$.
\end{itemize}
\end{lemma}

We will explore the following well-known connection between finiteness properties of the fundamental group of a finite graph of groups $(\mathcal{G},\Gamma)$ and those of the vertex groups, assuming that the edge groups are sufficiently well-behaved. The latter will always be satisfied for virtually polycyclic edge groups, as such groups are of type $F_m$ (and $FP_m$ over $R$), for all $m \in \N$.

\begin{prop}\label{prop:fin_props_of_graphs_of_gps} Suppose that $G$ is the fundamental group of a finite graph of groups with virtually polycyclic edge groups. Then, given any $m \in \N$, $G$ is of  type $F_m$ ($FP_m$ over $R$) if and only if $G_v$ is of  type $F_m$ (respectively, $FP_m$ over $R$), for all $v \in V\Gamma$.    
\end{prop}

\begin{proof}
In the case of the homological finiteness conditions $FP_m$ over $R$, the statement follows from \cite[Proposition~2.13 in Chapter~I]{Bieri-book}, using induction on the number of edges in $\Gamma$.

In the case of the homotopical finiteness properties $F_m$, the sufficiency is given by \cite[Exercise~3 in Section~7.2]{Geogh} and the necessity was proved in \cite[Theorem~1.2]{Hag-Wise-F-n} (see also \cite[Lemma~4.7 and Proposition~4.9]{Gui-Lev}, which give the necessity in the case $m \le 2$ and can be combined with \cite[Proposition~2.13 in Chapter~I]{Bieri-book} to deduce it for all $m \in \N$). 
\end{proof}

By combining \Cref{cor:normal_sbgp_in_graph_of_gps_is_fg} with \Cref{prop:fin_props_of_graphs_of_gps} we obtain the following.

\begin{cor} \label{cor:N_fp->interec_with_vertex_gps_are_fp}
Assume that $(\mathcal{G},\Gamma)$ is a finite graph of groups  with fundamental group $G$ and with virtually polycyclic edge groups. If $N \n G$ is a normal subgroup satisfying \eqref{eq:NG_e_has_f_i} then for each $m \in \N$
\begin{itemize}
    \item $N$ is of type $F_m$ if and only if  $N \cap G_v$ is of type $F_m$, for all $v \in V\Gamma$;
    \item $N$ is of type $FP_m$ over $R$ if and only if  $N \cap G_v$ is of type $FP_m$ over $R$, for all $v \in V\Gamma$.
\end{itemize}
\end{cor}

\begin{rem} In the case when $m=1$ and $\Gamma$ has only one edge, \Cref{cor:N_fp->interec_with_vertex_gps_are_fp} was proved by Ratcliffe in \cite[Theorems~1 and 2]{Rat-paper}.
\end{rem}

%%%%%%%%%%%%%%%%%%%                             %%%%%%%%%%%%%%%%
%%%%%%%%%%%%%%%%%%%           New Section       %%%%%%%%%%%%%%%%
%%%%%%%%%%%%%%%%%%%                             %%%%%%%%%%%%%%%%

\section{Background on BNSR invariants} \label{sec:BNSR}
In this section we will briefly summarize basic properties of the Bieri--Neumann--Strebel--Renz (BNSR) invariants that play a crucial role in the study of (virtual) fibering. 
Throughout this section we assume that $G$ is a finitely generated group and $R$ is a non-zero commutative ring with unity.

\begin{defn}\label{def:characters}
    We say that a homomorphism $\chi: G \to \R$  is a \emph{character} of $G$ and we use $\mathrm{Hom}(G,\R)$ to denote the set of all characters. If the $\Q$-rank of $\chi(G)$ is one (i.e., $\chi(G) \cong \Z$), we say that the character $\chi$ is \emph{discrete} (or \emph{rational}). Equivalently, $\chi$ is discrete if and only if there exists a homomorphism $\chi'\colon G \to \Z$  and a positive real number $r$ such that $\chi = r \chi'$. 
    
    We say that two characters $\chi_1$, $\chi_2$ of $G$ are \emph{equivalent} if $\chi_1 = r \chi_2$ for some positive real $r$. The set of equivalence classes
    \[
    S(G)\coloneq \{ [\chi] \mid \chi \in \mathrm{Hom}(G,\R)\setminus \{0\} \},
    \]
    with the induced structure coming from the finite dimensional normed real vector space $\mathrm{Hom}(G,\R)$, 
     will be called the \emph{character sphere} of $G$. For a subgroup $H\leqslant G$, we denote by $S(G,H)$ the \emph{great subsphere} of $S(G)$ consisting of equivalence classes of characters of $G$ that vanish on $H$.
\end{defn}

Note that the dimension  of the sphere $S(G)$ is $n-1$ (by a sphere if dimension $-1$ we mean the empty set), where $n \in \N_0$ is the $\Q$-rank of the abelianization $G/[G,G]$
(i.e., $G/[G,G] \cong \Z^n\times B$, for some finite abelian group $B$). If $H \leqslant G$ is a subgroup whose image in $G/[G,G]$ is infinite, then the great subsphere $S(G,H)$ has strictly lower dimension than the sphere $S(G)$. This observation implies claim (i) of the following lemma; claim (ii) is given by \cite[Lemma~B3.24]{Strebel-notes}.
\begin{lemma}\label{lem:discrete_chars_are_dense} Let $H$ be a subgroup of a finitely generated group $G$.
\begin{itemize}
    \item[(i)] If the natural image of $H$ in $G/[G,G]$ is infinite (equivalently, if there is a homomorphism $\varphi:G \to A$, for a torsion-free abelian group $A$, such that $\varphi(H) \neq \{0\}$) then $S(G,H)$ is a closed nowhere dense subset of $S(G)$.
    \item[(ii)] The set of equivalence classes of discrete characters is dense in ${S}(G)$. 
\end{itemize}
\end{lemma}

Within the character sphere $S(G)$, the \emph{invariant $\Sigma^1(G)$} was defined by Bieri, Neumann and Strebel \cite{BNS}. For the higher invariants ($m \ge 2$), one distinguishes between the \emph{homotopical BNSR invariants  $\Sigma^m(G) \subseteq S(G)$}, introduced by Renz in \cite{RenzThesis}, and the \emph{homological BNSR invariants $\Sigma^m(G,R) \subseteq S(G)$}, defined by Bieri and Renz in \cite{BR88}. We will not give formal definitions of these invariants here, since we will only need their basic properties stated in \Cref{cor:main_props_of_BNSR_invar} below.
A nice introduction to $\Sigma^1(G)$ is given in Strebel's notes \cite{Strebel-notes}, and main properties of the higher invariants are summarized in \cite[Sections~1.2 and 1.3]{BGK}.

BNSR invariants are important because they control finiteness properties of co-abelian normal subgroups.

\begin{thm}[{\cite{BR88,RenzThesis,Renz89}}]\label{thm:BNSR} 
    Let $G$ be a group of type $F_m$ ($FP_m$ over $R$), for some $m \in \N$, with a normal subgroup $N$ such that $G/N$ is abelian. Then $N$ is of type $F_m$ (respectively, $FP_m$ over $R$) if and only if $S(G,N) \subseteq \Sigma^m(G)$ (respectively, $S(G,N) \subseteq \Sigma^m(G,R)$).
\end{thm}
The main fact about these invariants is that they form open subsets of the character sphere $S(G)$, see \cite[Theorem~A (2.7) in Section~IV.2]{RenzThesis} and \cite[Theorem~A]{BR88}.

Note that in the case when $N=\ker\chi$, for a discrete character $\chi:G \to \R$, we have $S(G,N)=\{[\chi],-[\chi]\}$. Since we are interested in studying finiteness properties of the kernels of discrete characters, it makes sense to consider the \emph{symmetric BNSR invariants} in $S(G)$:
\[ \Sigma_{\pm}^m(G)\coloneq \Sigma^m(G) \cap \left(-\Sigma^m(G)\right)~\text{ and }~ \Sigma_{\pm}^m(G,R)\coloneq \Sigma^m(G,R) \cap \left(-\Sigma^m(G,R)\right).\]
A combination of the above results gives the following. 

\begin{cor}[{\cite{BR88,RenzThesis,Renz89}}]
\label{cor:main_props_of_BNSR_invar}
Let $m \in \N$ and let $G$ be a group of type $F_m$ ($FP_m$ over $R$). Then 
\begin{itemize}
    \item[(i)] the invariant $\Sigma_{\pm}^m(G)$ (respectively, $\Sigma_{\pm}^m(G,R)$) is open in $S(G)$;
    \item[(ii)] if $\chi:G \to \R$ is a discrete character then $\ker \chi$ is of type $F_m$ (respectively, $FP_m$ over $R$) if and only if $[\chi] \in \Sigma_{\pm}^m(G)$ (respectively, $[\chi] \in \Sigma_{\pm}^m(G,R)$).
\end{itemize}
\end{cor}

% \begin{rem}\label{rem:restricting_char_to_fi_sbgp} Let $G$ be a group with a finite index subgroup $H \leqslant_f G$, and let $m \in \N$. For any non-zero character $\chi:G \to \R$, we have $[\chi] \in \Sigma^m_{\pm}(G)$ if and only if $[\chi|_H] \in \Sigma^m_{\pm}(H)$.    
% \end{rem}

\begin{lemma}\label{lem:lemma_3} Let $H$ and $L$ be finitely generated groups, and let
$\psi\colon H\rightarrow L$ be a homomorphism. Then for any open subset $\Upsilon \subseteq S(H)$ the set 
    \[
    A\coloneq \{ [\xi] \mid \xi \in \mathrm{Hom}(L,\R) \text{ s.t. } [\xi \circ \psi] \in \Upsilon\}
    \]
    is open in $S(L)$. In particular, if $H$ is of type $F_m$, for some $m \in \N$, then 
        \[
\{ [\xi] \mid \xi \in \mathrm{Hom}(L,\R) \text{ s.t. } [\xi \circ \psi] \in \Sigma^m_{\pm}(H) \}
    \] is an open subset of $S(L)$.
\end{lemma}
\begin{proof}
Let $ \Phi \colon \mathrm{Hom}(L, \R)\rightarrow \mathrm{Hom} (H,\R)$ be the linear map defined by precomposition with $\psi$. Since these vector spaces are finite dimensional, the map $\Phi$ is necessarily continuous, so we have the following commutative diagram of continuous maps:
\begin{equation}
    \begin{tikzcd}
    % [column sep={0.001cm,2cm}, cells={nodes={inner sep=0pt, text width=5.5cm, align=center}}]
          \mathrm{Hom}(L,\R) \arrow[r,"\Phi"]& \mathrm{Hom}(H,\R) &\\
         B\coloneq \Phi^{-1}\left(\mathrm{Hom}(H,\R) \setminus \{0\}\right) \arrow[u,hookrightarrow] \arrow[r,"\Phi_B"]& \mathrm{Hom}(H,\R)\setminus \{0\}\arrow[u,hookrightarrow]
         \arrow[r,"q_H"]         & \mathrm{S}(H)
         % B\coloneq\mathrm{Hom}(L,\R) \setminus \Phi^{-1}(\{0\}) \arrow[u,hookrightarrow]\arrow[d,"q_L"] \arrow[r,mapsto,"\Phi|_B"]& \mathrm{Hom}(H,\R)\setminus \{0\}\arrow[u,hookrightarrow]\arrow[d,"q_H"]\\
         % \hspace{-1.5cm}\mathrm{S}(L) \supseteq q_L(B) \arrow[r,"\overbar{\Phi}"]& \mathrm{S}(H).
    \end{tikzcd},
\end{equation}
where $B$ is an open subset of $\mathrm{Hom}(L,\R)\setminus\{0\}$ (as $\Phi$ is continuous), $\Phi_B$ is the restriction of $\Phi$ to $B$, and $q_H$ is the quotient map sending each character to its equivalence class.

The continuity of $\Phi_B \circ q_H$ implies that
%\Cref{cor:main_props_of_BNSR_invar}.(i) and continuity tell us that 
the full preimage 
\[C \coloneq (q_H \circ \Phi_B)^{-1}(\Upsilon)=\{ \xi\in  \mathrm{Hom}(L,\R)\mid [\xi \circ \psi] \in \Upsilon \}\] is open in $B$, hence it is also open in $\mathrm{Hom}(L,R)\setminus\{0\}$. Observe that the quotient map $q_L: \mathrm{Hom}(L,\R)\setminus\{0\}\to S(L)$ is an open map, by definition, therefore $A=q_L(C)$ is open in $S(L)$, and the first claim of the lemma is proved. The second claim now follows from \Cref{cor:main_props_of_BNSR_invar}.(i).
\end{proof}

\begin{rem} The same argument can also be used to get a homological version of \Cref{lem:lemma_3}, where $F_m$ and $\Sigma^m_{\pm}(H)$ are replaced by $FP_m$ over $R$ and $\Sigma^m_{\pm}(H,R)$, respectively.    
\end{rem}

%%%%%%%%%%%%%%%%%%%                             %%%%%%%%%%%%%%%%
%%%%%%%%%%%%%%%%%%%           New Section       %%%%%%%%%%%%%%%%
%%%%%%%%%%%%%%%%%%%                             %%%%%%%%%%%%%%%%

\section{Fibering  graphs of groups}\label{sec:fib_graphs_of_gps}
In this section we establish necessary and sufficient criteria for $F_m$-fibering of fundamental groups of graphs of groups with virtually polycyclic edge groups, and we use it to prove \Cref{prop:crit_for_fibering_of_amalg}  from the Introduction.

The following observation stems from the fact that
every non-trivial subgroup of $\Z$ has finite index.

\begin{rem}\label{rem:non-triv_image_in_Z<=>NG_e_has_fi} Let $G$ be the fundamental group of a finite graph of groups $(\mathcal{G},\Gamma)$. If $\chi:G \to \R$ is a non-zero discrete character and $N \coloneq \ker\chi \n G$ then condition \eqref{eq:NG_e_has_f_i} is equivalent to the condition 
\begin{equation}\label{eq:non-triv_image_in_Z} 
\chi(\alpha_e(G_e)) \neq 0,~\text{ for all } e \in E\Gamma.    
\end{equation}    
\end{rem}

The following statement can be regarded as an analogue of \cite[Theorem~1.2]{Cash-Lev} for higher BNSR invariants. Instead of assuming that the graph of groups is reduced (as it is done in \cite{Cash-Lev}) we suppose that it has no terminal vertices, because the latter behaves well under passing to finite index subgroups (see Lemmas~\ref{lem:minmal_action<=>no_terminal_vertices} and \ref{lem:minimal_action_for_fi_sbgp}), while the former does not.

\begin{prop}\label{prop:graph_of_gps_crit_for_chi_to_be_in_higher_invar}
Let $G$ be the fundamental group of a finite graph of groups $(\mathcal{G},\Gamma)$ with virtually polycyclic edge groups and no terminal vertices. Suppose that $G$ is of type $F_m$, for some $m \in \N$,  and 
$\chi:G \to \R$ is a non-zero discrete character.

Then $[\chi] \in \Sigma^m_{\pm}(G)$ provided the following two conditions hold:
\begin{itemize}
    \item[(i)] $\chi(\alpha_e(G_e)) \neq \{0\}$, for all $e \in E\Gamma$;
    \item[(ii)] $\left[ \chi|_{G_v}\right] \in \Sigma_{\pm}^m(G_v)$, for each $v \in V\Gamma$.
\end{itemize}  

And conversely, if $[\chi] \in \Sigma^m_{\pm}(G)$ and  $\ker\chi \neq \alpha_e(G_e)$, for every $e \in E\Gamma$, then both conditions (i) and (ii) are satisfied.
\end{prop}

\begin{proof} Note that, by \Cref{prop:fin_props_of_graphs_of_gps}, each vertex group $G_v$ has type $F_m$, so it makes sense to talk about its BNSR invariant $\Sigma_{\pm}^m(G_v)$. Moreover, $G$ acts minimally on its Bass-Serre tree $\cT$, by \Cref{lem:minmal_action<=>no_terminal_vertices}.

Assume, first, that conditions (i) and (ii) hold, and denote $N \coloneq \ker\chi \n G$. Then $N \cap G_v$ has type $F_m$ by \Cref{cor:main_props_of_BNSR_invar}.(ii), for every $v \in V\Gamma$. \Cref{rem:non-triv_image_in_Z<=>NG_e_has_fi} allows us to apply \Cref{cor:N_fp->interec_with_vertex_gps_are_fp} to conclude that $N=\ker\chi$ is of type $F_m$, thus $[\chi] \in \Sigma^m_{\pm}(G)$ by \Cref{cor:main_props_of_BNSR_invar}.(ii).

Conversely, suppose that 
$[\chi] \in \Sigma^m_{\pm}(G)$ and  $\ker\chi \neq \alpha_e(G_e)$, for every $e \in E\Gamma$. In view of \Cref{cor:main_props_of_BNSR_invar}.(ii), this means that the normal subgroup $N=\ker\chi \n G$ is of type $F_m$, in particular, it is finitely generated. Therefore, we can apply \Cref{lem:Bri-How} to deduce that either 
\eqref{eq:N_is_in_G_e} or \eqref{eq:NG_e_has_f_i} is true. If the former is true then $N=\ker\chi \subseteq \alpha_e(G_e)$, for all $e \in E\Gamma$, which, combined with the assumption that $N \neq \alpha_e(G_e)$, implies condition \eqref{eq:non-triv_image_in_Z}, hence \eqref{eq:NG_e_has_f_i} is satisfied by  \Cref{rem:non-triv_image_in_Z<=>NG_e_has_fi}.
Thus we can assume that \eqref{eq:NG_e_has_f_i} is true. Then (i) holds by \Cref{rem:non-triv_image_in_Z<=>NG_e_has_fi} and (ii) holds by Corollaries~\ref{cor:N_fp->interec_with_vertex_gps_are_fp} and \ref{cor:main_props_of_BNSR_invar}.(ii).
\end{proof}

\begin{rem} The assumption that $\ker\chi \neq \alpha_e(G_e)$, for every $e \in E\Gamma$, in the converse direction of \Cref{prop:graph_of_gps_crit_for_chi_to_be_in_higher_invar} is important. Indeed, any semidirect product $G \coloneq H \rtimes \Z$ can be considered as an HNN-extension over $H$. If $H$ is of type $F_m$ then for the natural projection 
$\chi: G \to \Z$, with $\ker\chi=H$, we have $[\chi] \in \Sigma^m_{\pm}(G)$ and $\chi(H)=\{0\}$, in particular, $[\chi|_H] \notin \Sigma^m_{\pm}(H)$.  \end{rem}

It is well-known that every infinite virtually cyclic group $G_0$ has a finite normal subgroup $K$ such that $G_0/K$ is either infinite cyclic or infinite dihedral (see, for example, \cite[Lemma~2.5]{Far-Jon}). Clearly, in the latter case $G_0$ cannot map onto $\Z$, therefore we can make the following observation.

\begin{rem}\label{rem:struct_of_virt_cyc_gp} Let $G_0$ be a virtually cyclic group admitting a non-zero homomorphism to $\Z$. Then there is a finite normal subgroup $M \n G_0$ and an infinite order element $c \in G_0$ such that $G_0=M \langle c \rangle \cong M \rtimes \langle c \rangle$. Every homomorphism $\chi:G_0 \to \Z$ sends $M$ to $\{0\}$, thus $\chi$ is completely determined by the image  $\chi(c) \in \Z$.  
\end{rem}

Let us now prove the criterion for $F_m$-fibering of amalgamated free products mentioned in the Introduction.

\begin{proof}[Proof of \Cref{prop:crit_for_fibering_of_amalg}] We can treat $G$ as the fundamental group of a graph of groups $(\mathcal{G},\Gamma)$, where $\Gamma$ consists of two vertices and two mutually inverse edges joining them, with vertex groups $G_1$ and $G_2$ and with the edge group $G_0$. Note that $(\mathcal{G},\Gamma)$ has no terminal vertices because $G_0$ embeds as a proper subgroup of $G_1$ and $G_2$, and 
$G$ is of type $F_m$ by  \Cref{prop:fin_props_of_graphs_of_gps}.

Observe that $G_0$ cannot be the kernel of a character $G \to \R$ because if $G_0 \n G$ then $G/G_0 \cong G_1/G_0*G_2/G_0$ splits as a non-trivial free product, so $G/G_0$ is necessarily non-abelian.
Therefore, if $G$ $F_m$-fibers then (i) and (ii) hold by \Cref{prop:graph_of_gps_crit_for_chi_to_be_in_higher_invar} and \Cref{cor:main_props_of_BNSR_invar}.(ii).

Thus, it remains to prove the sufficiency, so suppose that (i) and (ii) are true.
Take any $i \in \{1,2\}$. Condition (i), together with \Cref{cor:main_props_of_BNSR_invar}, imply that $\Sigma^m_{\pm}(G_i)$   is a non-empty open subset of the sphere $S(G_i)$. In view of condition (ii), we can apply \Cref{lem:discrete_chars_are_dense} to deduce that there exists a non-zero homomorphism $\chi_i:G_i \to \Z$ such that $[\chi_i] \in \Sigma^m_{\pm}(G_i)$ and $\chi_i(G_0) \neq \{0\}$.
Let $c \in G_0$ be an infinite order element provided by \Cref{rem:struct_of_virt_cyc_gp}, and denote $n_i \coloneq \chi_i(c) \in \Z \setminus\{0\}$. Since $\Sigma^m_{\pm}(G_i)=-\Sigma^m_{\pm}(G_i)$, we can replace $\chi_i$ by $-\chi_i$, if necessary, to assume that $n_i>0$.

Observe that \[n_2\chi_1(c)=n_2 n_1=n_1 \chi_2(c),\] so, in view of \Cref{rem:struct_of_virt_cyc_gp}, we see that 
the homomorphisms \[n_2 \chi_1:G_1 \to \Z~\text{ and }~n_1\chi_2:G_2 \to \Z\] agree on the subgroup $G_0$. Therefore, we can define a homomorphism $\chi:G \to \Z$ by $\chi|_{G_1} \coloneq n_2 \chi_1$ and $\chi|_{G_2} \coloneq n_1 \chi_2$. 

By construction, $\chi(G_0) \neq \{0\}$ and $[\chi|_{G_i}]=[\chi_i] \in \Sigma^m_{\pm}(G_i)$, for $i=1,2$, hence $[\chi] \in \Sigma^m_{\pm}(G)$ by \Cref{prop:graph_of_gps_crit_for_chi_to_be_in_higher_invar}. Thus $G$ $F_m$-fibers, by \Cref{cor:main_props_of_BNSR_invar}.(ii).    
\end{proof}

\begin{defn}
Let $H$ be a group with a subgroup $K \leqslant H$. 
We will say that $H$ \emph{$F_m$-fibers relative to} $K$ if there is a non-zero homomorphism $\chi:H \to \Z$ such that $K \subseteq \ker \chi$ and $\ker\chi$ is of type $F_m$. 
% The group $G$ \emph{virtually fibers relative to}  $H$
% if there is a finite index subgroup $G' \leqslant_f G$ such that $H \subseteq G'$ and $G'$ fibers relative to $H$.
\end{defn}

The next theorem gives a new fibering criterion for fundamental groups of graphs of groups and may be of independent interest.

\begin{thm} \label{thm:fib_crit_for_graph_of_gps}
Let $(\mathcal{H},\Delta)$ be a finite graph of groups with fundamental group $H$, such that   $H$ is of type $F_m$, for some $m \in \N$, and all edge groups are virtually polycyclic.
Suppose that there exist a free abelian group $L$ of rank $n \in \N$ and homomorphisms $\psi:H \to L$ and  $\varphi:H \to \Z$ such that all of the following conditions hold:
\begin{itemize}
    \item[(a)] $\ker\psi \subseteq \ker\varphi$;
    \item[(b)] $\psi(\alpha_e(H_e))$ is non-trivial in $L$, for each $e \in E\Delta$;
    \item[(c)] for every $v \in V\Delta$ either $[\varphi|_{H_v}] \in \Sigma^m_{\pm}(H_v)$ or $S(H_v, K_v) \subseteq \Sigma^m(H_v)$, where $K_v \coloneq H_v \cap \ker\psi$.
\end{itemize}
Then the group $H$ $F_m$-fibers relative to $\ker\psi$.    
\end{thm}

\begin{proof} Without loss of generality, we can assume that $\Delta$ has no terminal vertices and $\psi$ is surjective. 
Let $U \subseteq V\Delta$ be the subset consisting of all vertices $u \in V\Delta$ such that $[\varphi|_{H_u}] \in \Sigma^m_{\pm}(H_u)$. Then \Cref{lem:lemma_3} tells us that the subset 
\[C \coloneq \bigcap_{u \in U} \{[\xi] \in S(L) \mid \xi\in \mathrm{Hom}(L,\R) \text{ s.t. } [\xi \circ \psi|_{H_u}] \in \Sigma^m_{\pm}(H_u)\}\] is open in $S(L)$ (if $U=\emptyset$ then we set $C\coloneq S(L)$).
According to (a), there is a homomorphism $\eta: L \to \Z$ such that $\varphi=\eta \circ \psi$. If $U \neq \emptyset$ then $[\eta] \in C$, by the definition of $C$, otherwise $C=S(L)$. In either case, we have verified that $C$ is a non-empty open subset of $S(L)$.     

Now, since $L \cong \Z^n$, condition (b) together with \Cref{lem:discrete_chars_are_dense}.(i) tell us that for each $e \in E\Delta$, the subset 
\[D_e \coloneq S\bigl(L,\psi(\alpha_e(H_e))\bigr)=\{[\xi] \in S(L) \mid (\xi\circ \psi)(\alpha_e(H_e))=\{0\} \} \] is closed and nowhere dense in $S(L)$. Consequently, the finite union $\cup_{e \in E\Delta} D_e$ is also closed and nowhere dense in $S(L)$. On the other hand, the set of discrete characters is dense in $S(L)$ by \Cref{lem:discrete_chars_are_dense}.(ii), hence there exists a non-zero discrete character $\xi_0:L \to \Z$ such that $[\xi_0] \in C$ and $[\xi_0] \notin D_e$, for all $e \in E\Delta$. It follows that the character $\chi\coloneq \xi_0 \circ \psi:H \to \Z$ satisfies 
\begin{itemize}
    \item $[\chi|_{H_u}] \in \Sigma^m_{\pm}(H_u)$, for each $u \in U$, and 
    \item $\chi(\alpha_e(H_e)) \neq \{0\}$, for all $e \in E\Delta$.
\end{itemize}

Since each vertex of $\Delta$ is incident to some edge, the latter condition implies 
that the restriction of $\chi$ to $H_v$ is non-zero, for all $v \in V\Delta$. To apply \Cref{prop:graph_of_gps_crit_for_chi_to_be_in_higher_invar}, it remains to check that $[\chi|_{H_v}] \in \Sigma^m_{\pm}(H_v)$, for each $v \in V\Delta \setminus U$. Note that $\ker \psi \subseteq \ker\chi $, by construction, hence $N_v \coloneq H_v \cap \ker\chi$ contains $K_v=H_v \cap \ker\psi$, so 
\[S(H_v,N_v) \subseteq S(H_v,K_v) \subseteq \Sigma^m(H_v), ~\text{ for every } v \in V\Delta \setminus U, \]
by the second part of condition (c). Thus, in view of \Cref{thm:BNSR} and \Cref{cor:main_props_of_BNSR_invar}.(ii), we see that $[\chi|_{H_v}] \in \Sigma^m_{\pm}(H_v)$, for all $v \in V\Delta$, so $[\chi] \in \Sigma^m_{\pm}(H)$ by
\Cref{prop:graph_of_gps_crit_for_chi_to_be_in_higher_invar}. \Cref{cor:main_props_of_BNSR_invar}.(ii) now tells us that $\ker\chi$ is of type $F_m$, showing that $H$ $F_m$-fibers relative to $\ker\psi$.
\end{proof}

\begin{rem} Essentially the same  argument proves the homological versions of \Cref{prop:graph_of_gps_crit_for_chi_to_be_in_higher_invar} and \Cref{thm:fib_crit_for_graph_of_gps}, where $F_m$ replaced by $FP_m$ over $R$ and $\Sigma^m_{\pm}(G)$ replaced by $\Sigma^m_{\pm}(G,R)$ (for some non-zero commutative ring $R$ with unity).    
\end{rem}

%%%%%%%%%%%%%%%%%%%                             %%%%%%%%%%%%%%%%
%%%%%%%%%%%%%%%%%%%           New Section       %%%%%%%%%%%%%%%%
%%%%%%%%%%%%%%%%%%%                             %%%%%%%%%%%%%%%%

\section{Virtual fibering}\label{sec:virt_fibering}
In this section we prove \Cref{thm:virt_fib_of_amalg}. 
The following statement, giving necessary conditions for virtual $F_m$-fibering of the fundamental group of a graph of groups with virtually cyclic edge groups, is a straightforward consequence of Bass-Serre Theory and \Cref{prop:graph_of_gps_crit_for_chi_to_be_in_higher_invar}.

\begin{prop}\label{prop:necessary_crit_for_virt_fib} Let $G$ be the fundamental group of a  finite graph of groups $(\mathcal{G},\Gamma)$ without terminal vertices and with virtually cyclic edge groups. Suppose that 
 $G$ is of type $F_m$, for some $m \in \N$. If $G$ virtually $F_m$-fibers and is not virtually abelian then every vertex group $G_v$ virtually $F_m$-fibers, $v \in V\Gamma$, and for each $e \in E\Gamma$, the edge group $G_e$ is infinite and satisfies $\alpha_e(G_e) \avr G$.     
\end{prop}

\begin{proof} Let $H \leqslant_f G$ be a finite index subgroup admitting a non-zero character $\chi:H \to \Z$ with $\ker\chi$ of type $F_m$, so that $[\chi] \in \Sigma^m_{\pm}(H)$, by \Cref{cor:main_props_of_BNSR_invar}.

The given splitting of $G$ induces a splitting of $H$ as the fundamental group of a finite graph of groups $(\mathcal{H},\Delta)$, which will have virtually cyclic edge groups and will be without terminal vertices by Lemmas~\ref{lem:minmal_action<=>no_terminal_vertices} and \ref{lem:minimal_action_for_fi_sbgp}. 

If $\ker\chi=\alpha_f(H_f)$, for some $f \in E\Delta$, then $H$ is virtually abelian (because it will be an extension of the virtually cyclic group $\alpha_f(H_f)$ by $\Z$). Hence, the group $G$  will also be virtually abelian, contradicting the assumptions. 

Thus we must have $\ker\chi \neq \alpha_f(H_f)$, for all $f \in E\Delta$.
The group $H$ is of type $F_m$, as a finite index subgroup of $G$, by \Cref{lem:finiteness_props_of_fi_sbgps}, so \Cref{prop:graph_of_gps_crit_for_chi_to_be_in_higher_invar} tells us that $\chi(\alpha_f(H_f)) \neq \{0\}$, for all $f \in E\Delta$, and $[\chi|_{H_u}] \in \Sigma^m_{\pm}(H_u)$, for all $u \in V\Delta$. The latter means that every vertex group $H_u$ is $F_m$-fibered (see \Cref{cor:main_props_of_BNSR_invar}.(ii)).

For each $v \in V\Gamma$, $G_v$ is the $G$-stabilizer of some vertex $\tilde{v}$ of the Bass-Serre tree $\cT$ for $G$. 
The graph of groups $(\mathcal{H},\Delta)$ is constructed from the induced action of $H$ on $\cT$ (see the paragraphs below \Cref{thm:kurosh} in Subsection~\ref{subsec:graphs_of_gps}), so for all $v \in V\Gamma$ there exist $g=g(v) \in H$ and $u=u(v) \in V\Delta$ such that $G_v \cap H=\mathrm{St}_H(\tilde{v})=gH_ug^{-1}$. Since $G_v \cap H \leqslant_f G_v$, we can conclude that $G_v$ virtually $F_m$-fibers.

Similarly, for each $e \in E\Gamma$ there are $h=h(e) \in H$ and
$f=f(e) \in E\Delta$ such that  $\alpha_e(G_e) \cap H=h \alpha_f(H_f) h^{-1} \leqslant_f \alpha_e(G_e)$ in $G$. Since 
\[\chi(h\alpha_f(H_f)h^{-1})=\chi(\alpha_f(H_f)) \neq \{0\},\]
we see that $|G_e|=\infty$ and, by  \Cref{lem:virt_cyc-avr_crit},  $\alpha_e(G_e) \avr G$, as desired. 
\end{proof}

In view of \Cref{prop:fin_props_of_graphs_of_gps} and \Cref{lem:basic_props_of_virt_retr}.(c), part (a) of \Cref{thm:virt_fib_of_amalg} from the Introduction is a special case of \Cref{prop:necessary_crit_for_virt_fib}. For the second part of this theorem we will need an auxiliary lemma.

\begin{lemma}\label{lem:aux_for_virt_fib}
Let $G=G_1*_{G_0} G_2$, where $G_0$ is infinite virtually cyclic and $G_0 \vr G_i$, for $i=1,2$. Suppose that we have  homomorphisms $\chi_i:G_i \to \Z$, $i=1,2$, such that $\chi_1(G_0)=\{0\}$ and $\chi_2(G_0) \neq \{0\}$. Then there is a finitely generated virtually abelian group $P$ and homomorphisms $\psi:G \to P$ and $\varphi:G \to \Z$ such that all of the following conditions are true:
\begin{itemize}
    \item $\ker\psi \subseteq \ker\varphi$;
      \item $\varphi|_{G_1}=\chi_1$;
    \item $\ker\psi \cap G_2$ has finite index in $\ker\chi_2$; 
    \item  $\psi$ is injective on $G_0$.
\end{itemize}
\end{lemma}

\begin{proof}
Since $\chi_2(G_0) \neq \{0\}$, \Cref{rem:struct_of_virt_cyc_gp} tells us that $G_0$ decomposes as a semidirect product $M \rtimes C$, where $M \n G_0$ is a finite normal subgroup and $C$ is infinite cyclic. As $G_0 \vr G_2$,  we know, by \Cref{lem:vab+avr+t-f=>vr}, that for each $x \in M\setminus\{1\}$ there is a finite index normal subgroup $N_x \n_f G_2$ such that $x \notin N_x$. Then $N \coloneq \bigcap_{x \in M\setminus\{1\}} N_x$ is a finite index normal subgroup of $G_2$ intersecting $M$ trivially.

The assumption $\chi_2(G_0) \neq \{0\}$ implies that $M = \ker\chi_2 \cap G_0$ (see \Cref{rem:struct_of_virt_cyc_gp}). It follows that \[O \coloneq \ker\chi_2 \cap N \n G_2\] is a normal subgroup in $G_2$, satisfying $O \cap G_0=\{1\}$.

Observe that the group $A \coloneq G_2/O$ is virtually cyclic, as $O \n_f \ker\chi_2$ and $G_2/\ker\chi_2 \cong \Z$. Since $O \cap G_0=\{1\}$ the natural homomorphism $\alpha: G_2 \to A$ is injective on $G_0$. Therefore, after identifying $G_0$ with its image in $A$ and using the  universal property of amalgamated free products, we obtain a homomorphism \[\beta:G=G_1*_{G_0} G_2 \to F \coloneq G_1*_{G_0} A,~\text{ such that } \beta|_{G_1}=\mathrm{Id}_{G_1} \text{ and } \beta|_{G_2}=\alpha.\]

Since $\chi_1(G_0)=\{0\}$, we have a homomorphism $\gamma: F \to \Z$ whose restriction to $G_1$ is $\chi_1$ and whose restriction to $A$ is the zero map. Since $G_0 \vr G_1$, by the assumptions, and $G_0 \vr A$ (in fact, $|A:G_0|<\infty$ as $|G_0|=\infty$ and $A$ is virtually cyclic), we can use \Cref{prop:maps_from_amalgams_to_vab_gps} to find a finitely generated virtually abelian group $P$ and an epimorphism $\delta: F \to P$ such that $\ker\delta \subseteq \ker \gamma$ and $\delta$ is injective on $A$. Denote 
\[\psi\coloneq \delta \circ \beta:G \to P~\text{ and }~\varphi \coloneq \gamma \circ \beta: G \to \Z, \]
so that we have the commutative diagram \eqref{eq:diagram_of_maps}.
\begin{equation}\label{eq:diagram_of_maps}
% https://tikzcd.yichuanshen.de/#N4Igdg9gJgpgziAXAbVABwnAlgFyxMJZARgBoAGAXVJADcBDAGwFcYkQBxAXg4H1iAVL2B9yAXwAEfAEwgxpdJlz5CKMsWp0mrdgDEe-ISN7iJAQTkKQGbHgJFypaZoYs2iEAAVLi2yqLSTi7a7iAAOmEAtvQ4ABYARvHAAFpicpowUADm8ESgAGYAThCRSI4gOBBIgSDxMGBQSAC0AMzlrjoeEWjYIDSM9HWMnkp2qiCFWFmxOD4gRSVlNJVILTR1DavtIewRDIVosVhzC6WINSuIayCMWGChUPRwsZknxWfll2RabrthdTh6G9Fohvpcah1QhFYIxAX0boMYMNRv4PJNprN5AV3kgwVUrjRIX8svRItF0mIgA
\begin{tikzcd}
                     & G=G_1*_{G_0} G_2 \arrow[ldd, "\psi"', bend right] \arrow[rdd, "\varphi", bend left] \arrow[d, "\beta"] &            \\
                     & F=G_1*_{G_0} A \arrow[ld, "\delta"'] \arrow[rd, "\gamma"]                                              &            \\
P \arrow[rr, dashed] &                                                                                                        & \mathbb{Z}
\end{tikzcd}    
\end{equation}
By construction, we have $\ker\psi \subseteq \ker\varphi$ and $\varphi|_{G_1}=\chi_1$. We also have 
\[\ker\psi \cap G_2=\ker\beta \cap G_2=\ker\alpha=O \n_f \ker\chi_2,\]
which implies that $\ker\psi \cap G_0=O \cap G_0=\{1\}$, as required.    
\end{proof}

Statement (b) of \Cref{thm:virt_fib_of_amalg} is given by the following.

\begin{thm}\label{thm:virt_fib_for_amalg-sufficiency} Suppose that $G=G_1*_{G_0} G_2$, where $G_0$ is infinite virtually cyclic and $G_1$, $G_2$ are of type $F_m$, for some $m \in \N$. If  the group $G_i$  $F_m$-fibers and $G_0 \vr G_i$, for every $i=1,2$, then $G$ virtually $F_m$-fibers.    
\end{thm}

\begin{proof}  By the assumptions and \Cref{prop:fin_props_of_graphs_of_gps}, the group $G$ is of type $F_m$.  We also know that for each $i=1,2$ there exists a non-zero homomorphism $\chi_i:G_i \to \Z$ such that $\ker\chi_i$ is of type $F_m$. We will consider $3$ different cases depending on the behaviors of these homomorphisms on $G_0$.

\medskip
\noindent \emph{Case 1.} $\chi_i(G_0) \neq \{0\}$, for $i=1,2$. Then $G$ $F_m$-fibers by \Cref{prop:crit_for_fibering_of_amalg}.

\medskip
\noindent \emph{Case 2.}  $\chi_i(G_0) = \{0\}$, for $i=1,2$. Then these homomorphisms agree on $G_0$, so there is a homomorphism $\varphi:G \to \Z$ such that $\varphi|_{G_i}=\chi_i$, for $i=1,2$.

Since $G_0 \vr G_i$, for $i=1,2$, by \Cref{prop:maps_from_amalgams_to_vab_gps}, there is a finitely generated virtually abelian group $P$ and a homomorphism $\psi:G \to P$ such that $\ker \psi \subseteq \ker\varphi$ and $\psi$ is injective on $G_0$. Let $L \n_f P$ be a free abelian finite index normal subgroup of finite rank, and let $H \coloneq \psi^{-1}(L) \n_f G$.

The decomposition of $G$ as the amalgamated free product of $G_1$ and $G_2$ over $G_0$ induces a splitting of $H$ as the fundamental group of a finite graph of groups $(\mathcal{H},\Delta)$, as described in \Cref{thm:kurosh}. We will now aim to apply \Cref{thm:fib_crit_for_graph_of_gps} to show that $H$ $F_m$-fibers relative to $\ker\psi \cap H$.

Let $\psi':H \to L$ and $\varphi':H \to \Z$ denote the restrictions of $\psi$ and $\varphi$ to $H$, respectively. Clearly we still have $\ker \psi' \subseteq \ker \varphi'$. Since $H \n_f G$, for each edge $e \in E\Delta$ there exists $g \in G$ and such that 
\[\alpha_e(H_e)=H \cap g G_0 g^{-1}\n_f gG_0 g^{-1}~ \text{ in  } G.\]
Thus $H_e$ is infinite virtually cyclic and $\psi'$ is injective on $\alpha_e(H_e)$, as $\psi$ is injective on $G_0$. In particular, $\psi'(\alpha_e(G_e))$ is non-trivial in $L$, for all $e \in E\Delta$.

Similarly, for every $v \in V\Delta$ there are $g \in G$ and $i\in \{1,2\}$ such that
$H_v=H \cap gG_ig^{-1} \n_f gG_i g^{-1}$ in  $G$. Hence, 
\begin{equation}\label{eq:kernel_of_varphi'}
\ker(\varphi'|_{H_v})=\ker\varphi \cap gG_ig^{-1} \cap H \n_f g (\ker\varphi \cap G_i)g^{-1}=g (\ker\chi_i) g^{-1}.\end{equation} Therefore, according to  \Cref{lem:finiteness_props_of_fi_sbgps}, $H_v$ and $\ker(\varphi'|_{H_v})$ are of type $F_m$. In view of \Cref{cor:main_props_of_BNSR_invar}.(ii), the latter shows that $[\varphi'|_{H_v}] \in \Sigma^m_{\pm}(H_v)$, for every $v \in V\Delta$. We can now apply \Cref{thm:fib_crit_for_graph_of_gps} to conclude that $H$ $F_m$-fibers relative to $\ker\psi'$, so $G$ virtually $F_m$-fibers.

\medskip
\noindent \emph{Case 3.} $\chi_j(G_0) =\{0\}$ and $\chi_k(G_0) \neq \{0\}$, where $\{j,k\}=\{1,2\}$. Without loss of generality, we can assume that $j=1$ and $k=2$.
Then we can apply \Cref{lem:aux_for_virt_fib} to find a finitely generated virtually abelian group $P$ and homomorphisms $\psi:G \to P$ and $\varphi:G \to \Z$ from its claim.

Now, as in Case 2, we let $L \n_f P$ be a finitely generated free abelian normal subgroup of finite index, denote $H \coloneq \psi^{-1}(L) \n_f G$, and let $\psi':H \to L$ and $\varphi':H \to \Z$ be the restrictions of the homomorphisms $\psi$ and $\varphi$ to $H$, respectively. Then $\ker\psi' \subseteq \ker\varphi'$, by construction. As before, $H$ decomposes as the fundamental group of a finite graph of groups $(\mathcal{H},\Delta)$, and we can check that for each $e \in E\Delta$, $H_e$ is infinite cyclic and $\psi'$ is injective on $\alpha_e(H_e)$. In particular, $\psi'(\alpha_e(H_e))$ is non-trivial in $L$.

Consider any vertex $v \in V\Delta$. By \Cref{thm:kurosh}, there are $i \in \{1,2\}$ and $g \in G$ such that $H_v=H \cap g G_i g^{-1}$ (in particular, $H_v$ is of type $F_m$, by \Cref{lem:finiteness_props_of_fi_sbgps}). If $i=1$ then, as in \eqref{eq:kernel_of_varphi'}, we see that 
\[\ker(\varphi'|_{H_v}) \n_f g (\ker\chi_1) g^{-1},\] which implies that $[\varphi'|_{H_v}] \in \Sigma^m_{\pm}(H_v)$. 

Let us now suppose that $i=2$, i.e., $H_v=H \cap gG_2g^{-1}$. Then 
\begin{equation}\label{eq:prop_of_K_v}
K_v \coloneq \ker\psi' \cap H_v \n_f \ker\psi \cap gG_2g^{-1}=g(\ker\psi \cap G_2) g^{-1}.    
\end{equation} But $\ker\psi \cap G_2$ has finite index in $\ker\chi_2$, by \Cref{lem:aux_for_virt_fib}, so \eqref{eq:prop_of_K_v} implies that
$K_v$ is isomorphic to a finite index subgroup of $\ker\chi_2$. In view of \Cref{lem:finiteness_props_of_fi_sbgps}, we deduce that $K_v$ is of type $F_m$. Hence, $S(H_v,K_v) \subseteq \Sigma^m(H_v)$, by \Cref{thm:BNSR}.

Thus, we have verified that the maps $\psi'$ and $\varphi'$ satisfy all the conditions of \Cref{thm:fib_crit_for_graph_of_gps}. We can therefore conclude that $H$ $F_m$-fibers (relative to $\ker\psi'$), and so $G$ virtually $F_m$-fibers.
\end{proof} 

%%%%%%%%%%%%%%%%%%%                             %%%%%%%%%%%%%%%%
%%%%%%%%%%%%%%%%%%%           New Section       %%%%%%%%%%%%%%%%
%%%%%%%%%%%%%%%%%%%    

\section{An adaptation to general open  invariants}\label{sec:adapt}
Given a group $G$, one may be interested in studying properties of the kernels of homomorphisms $G \to \Z$ other than $F_m$ or $FP_m$. To this end, it makes sense to introduce invariants corresponding to any such property, in a similar spirit to the BNSR invariants $\Sigma^m (G)$. Here we will restrict ourselves to the symmetric rational (i.e., discrete) invariants.

For a finitely generated group $G$, we let \[S_{\Q}(G)\coloneq \{[\chi] \mid \chi \in \mathrm{Hom}(G,\Q)\setminus\{0\}\}\] be the \emph{rational character sphere of $G$}. The equivalence of two characters is defined as in \Cref{def:characters}. Of course, $S_{\Q}(G)$ can be naturally identified with a dense subset of $S(G)$ (see \Cref{lem:discrete_chars_are_dense}.(ii)), and so it has a natural topology induced by the topology of $S(G)$.
Given a property of groups  $\cP$ we define the invariant $\Sigma^{\cP}_{\Q}(G) \subseteq S_{\Q}(G)$ by
\[\Sigma^{\cP}_{\Q}(G) \coloneq \{[\chi] \in S_{\Q}(G) \mid \ker\chi \text{ has } \cP\}.\]
Equivalent characters have the same kernel, so this invariant is well-defined.

\begin{defn}\label{def:rationally_open_prop} Let $\cP$ be a property of groups and let $\mathfrak G$ be a class of finitely generated groups.
We will say that \emph{$\cP$ is rationally open in $\fG$} if $\Sigma^{\cP}_{\Q}(G)$ is open in $S_{\Q}(G)$, for every $G \in \fG$.    
\end{defn}

Of course, having type $F_m$ (or $FP_m$ over a ring $R$) are examples of rationally open properties in the class of all finitely generated groups, by \Cref{cor:main_props_of_BNSR_invar}. However, by the recent work of Fisher  \cite{Fisher-Novikov_cohom}, there are also other interesting rationally open properties in the class of groups of type $FP$, such as being a free group (possibly of infinite rank) or having cohomological dimension strictly smaller than that of the whole group.

We will look at group properties $\cP$ satisfying the following conditions.

 \begin{enumerate}[label={\normalfont (C\arabic*)}]
    \item\label{cond:C1} $\Z$ has $\cP$;
   \item\label{cond:C2} if $A$ and $B$ have $\cP$ then so does $A*B$;
    \item\label{cond:C3} if $G$ has $\cP$ and $H\n_f G$ then $H$ has $\cP$.
\end{enumerate}

We will say that a finitely generated group $G$ is \emph{$\cP$-fibered} if $\Sigma^{\cP}_{\Q}(G) \neq \emptyset$, i.e.,  if 
there is a non-zero homomorphism $\chi:G \to \Z$ such that $\ker\chi$ has $\cP$. A group $G$ is \emph{virtually $\cP$-fibered} if it has a $\cP$-fibered subgroup of finite index.
We can now generalize one direction of \Cref{prop:crit_for_fibering_of_amalg} to rationally open properties (here we need to assume that $G_0 \cong \Z$).

\begin{prop}\label{prop:amalg_P-fibers} Let $\cP$ be property of groups satisfying \ref{cond:C1}--\ref{cond:C2}. Suppose that 
$G=G_1 *_{G_0} G_2$, where $G_0$ is infinite cyclic and $G_i$ is finitely generated  and $\cP$-fibered, for $i=1,2$. If the image of $G_0$ in the abelianization $G_i/[G_i,G_i]$ is infinite, for $i=1,2$,
and $\cP$ is rationally open in the class $\{G_1,G_2\}$ then $G$ is $\cP$-fibered. \end{prop}

\begin{proof} The ``sufficiency'' part of the argument from the proof of \Cref{prop:crit_for_fibering_of_amalg}, given in \Cref{sec:fib_graphs_of_gps}, goes through verbatim to show that there is a character $\chi:G \to \Z$ such that $\chi(G_0) \neq \{0\}$ and $[\chi|_{G_i}] \in \Sigma^{\cP}_{\Q}(G_i)$, for $i=1,2$. The latter condition shows that $\ker\chi \cap G_i$ has property $\cP$, for $i=1,2$, and the former condition implies that $\ker\chi \cap G_0=\{1\}$ (because $G_0$ is infinite  cyclic). 

In view of \Cref{rem:non-triv_image_in_Z<=>NG_e_has_fi}, we can apply \Cref{cor:normal_sbgp_in_graph_of_gps_is_fg} to conclude that $\ker\chi$ splits as the fundamental group of a finite graph of groups, where all edge groups are trivial and all vertex groups have $\cP$. Thus $\ker\chi$ is the free product of finitely many groups with $\cP$ and a finitely  generated free group, so $\ker\chi$ satisfies $\cP$ by conditions \ref{cond:C1} and \ref{cond:C2}.
\end{proof}

For finitely generated groups $G_1$, $G_2$ we denote by $\mathfrak{FI}(G_1,G_2)$ the class of groups consisting of finite index subgroups of $G_1$ and $G_2$. The following statement generalizes \Cref{thm:virt_fib_for_amalg-sufficiency}.

\begin{thm}\label{thm:amalg_virt_P-fibers}
Let $\cP$ be property of groups satisfying \ref{cond:C1}--\ref{cond:C3}. Suppose that 
$G=G_1 *_{G_0} G_2$, where $G_0$ is infinite virtually cyclic and $G_i$ is finitely generated  and $\cP$-fibered, for $i=1,2$. If $G_0 \vr G$, for $i=1,2$, 
and $\cP$ is rationally open in the class $\mathfrak{FI}(G_1,G_2)$ then $G$ is virtually $\cP$-fibered.
\end{thm}

\begin{proof} We only sketch the argument, leaving the details to the reader.

By the assumptions, for each $i \in \{1,2\}$ there is a non-zero character $\chi_i:G_i \to \Z$, where $\ker\chi_i$ satisfies property $\cP$.
We can argue as in each of the three cases of the proof of \Cref{thm:virt_fib_for_amalg-sufficiency} (using \Cref{prop:maps_from_amalgams_to_vab_gps}), to produce a finitely generated virtually abelian group $P$ and homomorphisms $\varphi:G \to \Z$ and $\psi:G \to P$ such that 
\begin{itemize}
    \item $\ker\psi \subseteq \ker\varphi$; %and $\ker\psi \cap G_i \subseteq \ker\chi_i$, for $i=1,2$;
    \item $\psi$ is injective on $G_0$;
    \item for each $i \in \{1,2\}$, either $\varphi|_{G_i}=\chi_i$ or $\ker\psi \cap G_i$ has finite index in $\ker\chi_i$.
\end{itemize}
Next, we choose a free abelian subgroup $L \n_f P$, set $H\coloneq \psi^{-1}(N) \n_f G$, and denote by $\psi':H \to L$ and $\varphi':H \to \Z$ the restrictions of $\psi$ and $\varphi$ to $H$, respectively. 
As before, $H$ will have an induced decomposition as a fundamental group of a finite graph of groups $(\mathcal{H},\Delta)$, and $\psi'$ will be injective on the edge groups. Since all the edge groups are infinite virtually cyclic and $L$ is torsion-free, we deduce that each $H_e$ is infinite cyclic. 

For any vertex group $H_v$, there will be $g \in G$ and $i \in \{1,2\}$ such that  $H_v=g(H \cap G_i)g^{-1}$.
If $i \in \{1,2\}$ is such that $\varphi|_{G_i}=\chi_i$, then $\ker(\varphi'|_{H_v}) \n_f g (\ker\chi_i) g^{-1}$ has property $\cP$ by \ref{cond:C3}, thus $[\varphi'|_{H_v}] \in \Sigma^{\cP}_{\Q}(H_v)$. In the other case, when $\ker\psi \cap G_i$ has finite index in $\ker\chi_i$, we have 
\begin{equation}\label{eq:prop_of_K_v-2}
 K_v \coloneq H_v \cap \ker\psi' \n_f g(\ker\chi_i)g^{-1},   
\end{equation}
thus $K_v$ has $\cP$ by \ref{cond:C3}. On the other hand, \eqref{eq:prop_of_K_v-2} also shows that $H_v/K_v$ is infinite virtually cyclic, as it is commensurable up to finite kernels with $G_i/\ker\chi_i \cong \Z$. But $H_v/K_v$ is a subgroup of the torsion-free group $L$, hence $H_v/K_v \cong \Z$. It follows that $S(H_v,K_v) \subseteq \Sigma^{\cP}_{\Q}(H_v)$.

The argument from \Cref{thm:fib_crit_for_graph_of_gps} now goes through with minimal adaptation (after ``rationalizing'' everything and using the assumption that $ \Sigma^{\cP}_{\Q}(H_v)$ is open in $S_{\Q}(H_v)$, for each $v \in V\Delta$), to show that there is a non-zero character $\chi:H \to \Z$ such that $[\chi|_{H_v}] \in \Sigma^{\cP}_{\Q}(H_v)$, for all $v \in V\Delta$, and $\chi(\alpha_e(H_e)) \neq \{0\}$, for each $e \in E\Delta$. The latter condition implies that $\ker\chi \cap \alpha_e(H_e)=\{1\}$, for all $e \in E\Delta$ (because $H_e \cong \Z$). As in \Cref{prop:amalg_P-fibers}, we can now apply \Cref{cor:normal_sbgp_in_graph_of_gps_is_fg} to deduce that $\ker\chi$ splits as a free product of finitely many groups satisfying $\cP$ with a free group of finite rank. Conditions \ref{cond:C1} and \ref{cond:C2} then imply that $\ker\chi$ has $\cP$, thus $H$ $\cP$-fibers and $G$ virtually $\cP$-fibers.
\end{proof}

\begin{proof}[Proof of \Cref{cor:amalg_is_F-by_Z}] In \cite[Corollary C]{Fisher-Novikov_cohom} Fisher proved that being free is a rationally open property in the class of  groups of type $FP$. Finitely generated $F$-by-$\Z$ groups are of type $FP$ (and even of type $F$) by the work of Feighn and Handel \cite[Theorem~1.2]{FH}, and this class of groups is obviously closed under taking finite index subgroups. Therefore, statements (i) and (ii) of 
\Cref{cor:amalg_is_F-by_Z} follow from \Cref{prop:amalg_P-fibers} and 
\Cref{thm:amalg_virt_P-fibers} respectively.
\end{proof}

%%%%%%%%%%%%%%%%%%%                             %%%%%%%%%%%%%%%%
%%%%%%%%%%%%%%%%%%%           New Section       %%%%%%%%%%%%%%%%
%%%%%%%%%%%%%%%%%%%    

\section{Examples and open questions}\label{sec:examples}
If one analyzes the proof of \Cref{thm:amalg_of_virt_ab}, one will notice that we have a good control over the intersection of the images $(\beta_1\circ \alpha_1)(G_1)$ and $(\beta_2\circ \alpha_2)(G_2)$ in $E$, but we may lose this control after replacing $\beta_1$ by $\delta \circ \beta_1$. Therefore, the next question arises naturally, and the positive answer to it may be useful in future applications.
\begin{question} Given an amalgamated free product $G=G_1*_{G_0} G_2$ of two finitely generated virtually abelian groups $G_1$, $G_2$ over a common virtually cyclic subgroup $G_0$,  is it always possible to find a finitely generated virtually abelian group $E$ and a homomorphism $\nu:G \to E$ such that $\nu$ is injective on each $G_i$, $i=1,2$, and the intersection 
$\nu(G_1) \cap \nu(G_2)$ is virtually cyclic?  \end{question}

\begin{ex} \Cref{thm:amalgam_of_(VRC)_gps} and \Cref{lem:vabs_have_LR} imply that any amalgamated free product $G$ of two finitely generated virtually abelian groups over a common virtually cyclic subgroup has (VRC). However, such $G$ does not necessarily have the stronger property (LR) (i.e., some finitely generated subgroup may not be a virtual retract of $G$). Indeed,
Long and Niblo \cite[Theorem~4]{LN} gave an example of a double  of the $(4,4,2)$ triangle group (which is virtually $\Z^2$), over an infinite cyclic subgroup, that is not LERF, so this double does not have (LR) by \cite[Lemma~5.1.(iii)]{virtprops}. 
\end{ex}

The next example shows that amalgamated free products of solvable groups over virtually cyclic virtual retracts need not  be residually solvable, so one cannot drop the word ``virtually'' from claim (ii) of  \Cref{cor:virt_res_solv}.

\begin{ex}\label{ex:amalg_not_res_solv}
    Let $C_3$ denote the cyclic group of order $3$, $G_0 \coloneq (C_3)^5$, and let $G_1$ and $G_2$ be of the form $G_0 \rtimes C_3$, where each $C_3$ acts on $G_0$ by shuffling its factors using the permutations $(123)$ and $(345)$ respectively. Then $G_1$, $G_2$ are finite metabelian groups, and, of course, $G_0 \vr G_i$, for $i=1,2$. However,   the group $G = G_1 *_{G_0} G_2$ is not residually solvable. Indeed, if $G$ were residually solvable, there would exist a solvable group $H$ and an epimorphism $\varphi\colon G \to H$ that is  injective on $G_0$ (because $G_0$ is finite and the direct product of a finite number of solvable groups is solvable).
    Then $H$ acts on $\varphi(G_0) \n H$ by conjugation in the same way as $G$ acts on $G_0$, i.e., we have the following commutative diagram:
    \begin{equation}
        \begin{tikzcd}
            G \arrow[r, "\varphi"]\arrow[d] & H  \arrow[d] \\
            \Aut(G_0) \arrow[r, " \phi "]& \Aut(\varphi(G_0)) 
        \end{tikzcd},
    \end{equation}
where the vertical arrows come from the actions of $G$ and $H$ on $G_0$ and $\varphi(G_0)$ by conjugation, and  $\phi$ is the isomorphism defined by  $\psi \mapsto \varphi \circ \psi \circ \varphi^{-1}$, for all $\psi \in \Aut(G_0)$. By construction of $G$, the image of the left vertical map is isomorphic to the alternating group  $A_5$, and since $\phi$ is an isomorphism, the same holds for the image of $H$ under the right vertical map. Since $H$ is solvable, the latter would mean that $A_5$ is solvable, giving a contradiction.
\end{ex}

%The next question is prompted by \Cref{cor:virt_res_solv}. 

\begin{question}\label{q:res_amen} Suppose that $G=G_1*_{G_0} G_2$, where $G_1$, $G_2$ are residually amenable, $G_0$ is virtually cyclic and $G_0 \vr G_i$, for $i=1,2$. Is $G$ residually amenable?    
\end{question}

Berlai \cite[Proposition~1.4]{Berlai} showed that there exist (residually ame\-nable group)-by-amenable groups that are not residually amenable (in other words, amenability is not a root property in the sense of Gruenberg \cite{Gruenberg-root}). Hence, our argument for proving \Cref{cor:virt_res_solv}.(ii) (see \Cref{cor:claim_(ii)-virt_res_solv}) does not  give an answer to \Cref{q:res_amen}.

Since virtually special groups are prime examples of groups with (VRC), \Cref{thm:amalgam_of_(VRC)_gps} naturally prompts the following.
\begin{question}\label{q:vir_special}
    Is the amalgam of two finitely generated virtually special groups over a virtually cyclic subgroup virtually special?
\end{question}

The next question addresses the gap between conditions in parts (a) and (b) of \Cref{thm:virt_fib_of_amalg}.

\begin{question}\label{q:vir_fibered}
   Let $G=G_1*_{G_0} G_2$, with $G_0$ infinite virtually cyclic. Suppose that $G_i$ is virtually fibered and $G_0 \vr G_i$, for $i=1,2$.  Must $G$ be virtually fibered? 
\end{question}

We note that a positive answer to \Cref{q:vir_special} implies a positive answer to \Cref{q:vir_fibered} in the case when $G_1$ and $G_2$ are virtually special. Indeed, virtually special groups are virtually RFRS by \cite[Theorem~2.2]{Agol-virt_fib}, and $\beta_1^{(2)}(G_i)=0$, for $i=1,2$, by \cite[Theorem~2.27]{Kielak2020}. Then $\beta_1^{(2)}(G)=0$
by \cite[Theorem~1]{Fernos-Valette2017}, so we can apply Kielak's result \cite[Theorem~5.3]{Kielak2020} to see that $G=G_1*_{G_0} G_2$ is virtually fibered.  

We don't even know the answer to the following weaker version of \Cref{q:vir_fibered} in the case when $\Gamma$ is the path of length $2$.

\begin{question}
Let $G$ be the fundamental group of a finite graph of groups $(\mathcal{G},\Gamma)$. Suppose that $\Gamma$ is a tree, each vertex group $G_v$ is fibered, each edge group $G_e$ is infinite virtually cyclic and $\alpha_e(G_e) \vr G_{\alpha(e)}$, for all $e \in E\Gamma$. Does it follow that $G$ is virtually fibered?
\end{question}

The following example shows that the assumption that $G_0 \vr G_i$ in part (b) of \Cref{thm:virt_fib_of_amalg} cannot be weakened to $G_0 \avr G_i$, for $i=1,2$.

\begin{ex}\label{ex:ind_avr_but_not_virt_fibered}
Let $D=\langle a,b \mid a^2=b^2=1 \rangle$ be the infinite dihedral group, let $C=\langle c \mid c^2=1\rangle$ be cyclic of order $2$, and let 
\[F \coloneq \langle D \times C ,t \mid ~tat^{-1}=c\rangle\] be the HNN-extension of $D \times C$ with associated subgroups $\langle a \rangle$ and $\langle c \rangle$. We define $G_1 \coloneq F \times E$, where 
$E=\langle e \mid~\rangle$ is infinite cyclic.
Evidently, $G_1$ is fibered and it has (VRC) by \cite[Corollary~6.5 and Lemma~5.2]{virtprops}, therefore the virtually cyclic subgroup $\langle [a,b],c \rangle \cong C_\infty \times C_2$ is a virtual retract of $G_1$.

Now, let $G_2 \coloneq Z \times Q$, where $Z=\langle z \rangle$  is infinite cyclic and $Q$ is a finitely generated group such that the finite residual of $Q$ contains an involution $q \in R(Q)$.
% (for example, $P$ can be chosen as a finite central extension of  $\mathrm{Sp}(4,\Z)$, see \cite{Deligne}). 
Thus $G_2$ is fibered and $Z$ is a retract, so $\langle z,q \rangle \avr G_2$.

We define $G$ as the free product of $G_1$ and $G_2$ amalgamated along their infinite virtually cyclic subgroups $H_1 \coloneq \langle [a,b],c \rangle$ and $H_2 \coloneq  \langle z,q \rangle $, namely
\[G\coloneq \langle G_1,G_2 \mid [a,b]=z,~c=q\rangle. \]

Let us show that $G$ is not virtually fibered. 
Note that, by \Cref{lem:fin_residual_of_fin_ind_sbgp}, $c=q \in R(Q) \subseteq R(G)$, hence $a=t^{-1}ct \in R(G)$, so $[a,b] \in R(G)$, as $R(G)\n G$, hence $H_1 \subseteq R(G)$. Since every finitely generated virtually abelian group $P$ is residually finite,  \Cref{rem:props_of_finite_residual} implies that 
for any homomorphism $\varphi:G \to P$, we must have $\varphi(H_1)=\{1\}$.  \Cref{lem:virt_cyc-avr_crit} now tells us that $H_1$ cannot be an almost virtual retract of $G$, so $G$ does not virtually fiber by \Cref{prop:necessary_crit_for_virt_fib}.  
\end{ex}

In the next example, we construct an amalgamated free product as in \Cref{not:G} such that $G_1$ is virtually fibered, $G_2$ is fibered and $G_0 \avr G$ has infinite image in the abelianization of $G$, but $G$ is not virtually fibered. This example is in contrast with  \Cref{prop:crit_for_fibering_of_amalg} and \Cref{thm:virt_fib_of_amalg}.

\begin{ex}\label{ex:second_non_rf_example} Let $G_1 \coloneq D \times F$, where $D$ is the infinite dihedral group generated by involutions $a, b$ and $F$ is the free group with basis $\{x,y\}$. Let $G_2=Z \times Q$ be the group from \Cref{ex:ind_avr_but_not_virt_fibered}, with an involution $q \in R(Q)$.
Clearly, $G_1$  is virtually fibered (it has an index $2$ subgroup isomorphic to $C_\infty \times F$) and $G_2$ is fibered. 

We define $G$ as the free product of $G_1$ and $G_2$ amalgamated along their infinite virtually cyclic subgroups $H_1=\langle a,x \rangle \cong C_\infty \times C_2$ and $H_2= \langle z,q \rangle \cong C_\infty \times C_2$. Thus
\[G\coloneq\langle G_1,G_2 \mid a=q,~x=z\rangle. \]
Evidently, $G$ retracts onto its infinite cyclic subgroup $\langle x \rangle=\langle z \rangle$, with the kernel of the retraction normally generated by $a,b,y$ and $Q$, hence the edge group $G_0\coloneq H_1=H_2$ is an almost virtual retract of $G$, and it also has infinite images in the abelianizations of $G$, $G_1$ and $G_2$. 

Arguing as in \Cref{ex:ind_avr_but_not_virt_fibered}, one can see that  $a=q \in R(G)$, so  $D$ has finite image in every virtually abelian quotient of $G$, thus $D$ is not an almost virtual retract of $G$ by \Cref{lem:virt_cyc-avr_crit}. However, it is easy to see that $G$ splits as the amalgamated free product
$G=\langle D,x,G_2 \rangle *_{D} \langle D, y \rangle$,
so $G$ does not virtually fiber by \Cref{prop:necessary_crit_for_virt_fib}, as $D \nleqslant_{avr} G$.
\end{ex}

The following example shows that there is no converse to \Cref{prop:necessary_crit_for_virt_fib} for HNN-extensions even in the case when all edge groups are retracts.

\begin{ex}\label{ex:Gerstens_amalgam}  Consider the tubular group $G$ given by the presentation
\begin{equation}\label{eq:amalg_of_Gersten}
G \coloneq \langle a,b,d,s,t \mid [a,b]=[a,d]=1,~sbs^{-1}=ab,~tbt^{-1}=a^2b\rangle.    
\end{equation}

On one hand,  $G=\langle K,s,t \mid sbs^{-1}=ab,~tbt^{-1}=a^2b \rangle$ is a double HNN-extension of the fibered subgroup $K \coloneq \langle a,b,d \rangle \cong C_\infty \times F_2$. And presentation \eqref{eq:amalg_of_Gersten} shows that the associated cyclic subgroups $\langle b\rangle$, $\langle ab \rangle$ and $\langle a^2 b \rangle$, of this double HNN-extension, are all retracts of $G$, with the kernel of the retraction normally generated by $a,d,s$ and $t$ (in particular, $\langle b\rangle$, $\langle ab \rangle$ and $\langle a^2 b \rangle$ have infinite images in the abelianization of $G$).

Alternatively, $G$ is also an HNN-extension 
$\langle J, t \mid tbt^{-1}=a^2b \rangle$, 
of the tubular group \[J\coloneq \langle a,b,d,s \mid [a,b]=[a,d]=1,~sbs^{-1}=ab\rangle .\]

Now, the group $J$ has (VRC) by \cite[Corollary~10.3 and Remark~10.4]{MM-vr_in_free_constr}, because the images of the cyclic subgroups $\langle b \rangle$ and  $\langle ab \rangle=\{1\}$ have trivial intersection in the abelianization of the subgroup $\langle a,b,d \rangle$. Therefore, according to \cite[Proposition~12.3]{MM-vr_in_free_constr}, $J$ is virtually free-by-cyclic, in particular it virtually fibers.

On the other hand, $G$ can be viewed as an amalgamated free product of Gersten's free-by-cyclic group $H$, given by \eqref{eq:Gerstens_gp},
with the free abelian group $\langle a,d \rangle \cong \Z^2$ over $\langle a \rangle$. In \cite[Lemma~5.18]{Wu-Ye} it is shown that $\langle a \rangle$ is not a virtual retract of $H$, hence $\langle a \rangle  \nleqslant_{vr} G$, so $\langle a \rangle  \nleqslant_{avr} G$, by \Cref{lem:vab+avr+t-f=>vr} (as $a$ has infinite order in $G$). Therefore, $G$ does not virtually fiber by \Cref{prop:necessary_crit_for_virt_fib}.
\end{ex}

In our final example we construct an amalgam $G$ of two virtually abelian RFRS $\Z^2\rtimes \Z$ groups over a common normal virtual retract such that $G$ does not virtually fiber. This serves as further evidence that \Cref{thm:virt_fib_of_amalg}.(b) may fail when $G_0$ is not virtually cyclic.

\begin{ex}\label{ex:amalg_of_virt_ab_not_fibered} 
Let $G_0 \coloneq \Z^2$,  $G_1 \coloneq G_0\rtimes_{x} \Z$ and $G_2 \coloneq G_0 \rtimes_{y} \Z$, where $x,y$ are two finite order matrices, such that $ \langle x,y \rangle = \mathrm{SL}(2,\Z)$. Then $G_1$ and $G_2$ are finitely generated, locally indicable and virtually abelian, hence both of them are RFRS by \cite[Lemma~6.1]{Ok-Sch}. Moreover, $G_0 \vr G_i$, for $i=1,2$, as $x$ and $y$ have finite orders. 

Let us prove that  $G=G_1 *_{G_0} G_2$ is not virtually fibered. 
Indeed, suppose there exists $ H \leqslant_f G $ and a nontrivial homomorphism $\varphi\colon H \to \Z$ such that $\ker \varphi$ is finitely generated (so $[\varphi]\in \Sigma^1_{\pm}(H)$ by \Cref{cor:main_props_of_BNSR_invar}.(ii)). By \Cref{thm:kurosh}, $H$ is the fundamental group of a finite graph of groups $(\mathcal{H},\Delta)$ without terminal vertices, and since $G_0 \n G$,  we have
\begin{equation*}\label{eq:same proper edge group}
    \alpha_e(H_e) = H \cap G_0\n H, \text{ for all } e\in E\Delta.
\end{equation*}
Note that $G/G_0$ is the free group of rank $2$, by construction, so $H/\alpha_e(H_e) \leqslant_f G/G_0$ is non-abelian, in particular $\ker \varphi \neq \alpha_e(H_e)$, for all $e\in E\Delta$. Therefore, we can apply the converse direction of \Cref{prop:graph_of_gps_crit_for_chi_to_be_in_higher_invar} to deduce that $\varphi(\alpha_e(H_e)) \ne \{0\}$, for some $e\in E\Delta$. Since $\alpha_e(H_e) \cong \Z^2$, this implies that $\ker \varphi|_{\alpha_e(H_e)} \cong \Z$, so 
\begin{equation}\label{eq:v normalized}
\ker \varphi \cap \alpha_e(H_e) =\langle \vec v \rangle \n H,
\end{equation}
for some non-zero  vector $\vec v \in \Z^2$. Since  $ \langle x,y \rangle = \mathrm{SL}(2,\Z)$, the natural homomorphism $\psi \colon G \to \Aut(G_0)$ has image $\mathrm{SL}(2,\Z)$. Then \eqref{eq:v normalized} implies that all matrices from $\psi(H)$ must have $ \vec v $ as an eigenvector, which is clearly impossible 
as $\psi(H)\leqslant_f \mathrm{SL}(2,\Z)$.
% (matrices $\begin{pmatrix}
%     1 & n \\ 0 & 1
% \end{pmatrix}$ and $\begin{pmatrix}
%     1 & 0 \\ n & 1
% \end{pmatrix}$ do not share eigenvectors for all $n\in \N$). 
Thus $G$ does not virtually fiber.
\end{ex}

Since the conditions given in \Cref{cor:amalg_is_F-by_Z} are not necessary, the following problem remains open.

\begin{problem} Find a condition that is both necessary and sufficient for an amalgamated free product of two finitely generated $F$-by-$\Z$ groups over an infinite cyclic subgroup to be $F$-by-$\Z$ (or virtually $F$-by-$\Z$).     
\end{problem}

The discussion in \Cref{sec:adapt} naturally suggests the following.
\begin{problem}
Find other interesting examples of rationally open properties in natural classes of groups.
\end{problem}

%	\bibliographystyle{amsplain}
%\bibliography{refs}
\printbibliography 
\end{document}